\documentclass[11pt]{article}
\usepackage{t1enc}
\usepackage[utf8]{inputenc}
\usepackage[T1]{fontenc}
\usepackage{graphicx}
\usepackage{amsmath}
\usepackage{amsthm}
\usepackage{amssymb}
\usepackage{sseq}
\usepackage[pagebackref]{hyperref}
\usepackage{hyperref}
\usepackage[curve]{xypic}
\usepackage[left=3.2cm,right=3.2cm,top=3cm,bottom=3cm]{geometry}
\usepackage{pdfpages}
\usepackage{stmaryrd}
\usepackage{marvosym}
\usepackage{breakurl}

\numberwithin{equation}{section}
\newtheorem{thm}[equation]{Theorem}
\newtheorem*{thm*}{Theorem}
\newtheorem*{mthm*}{Main Theorem}

\let\oldmarginpar\marginpar
\renewcommand\marginpar[1]{\-\oldmarginpar[\raggedleft\footnotesize #1]%
{\raggedright\footnotesize #1}}

\newtheorem{lemma}[equation]{Lemma}

\newtheorem{cor}[equation]{Corollary}

\newtheorem{prop}[equation]{Proposition}

\newtheorem*{conjecture*}{Conjecture}

\newtheorem{question}[equation]{Question}
\newtheorem*{question*}{Question}

\theoremstyle{definition}
\newtheorem{defi}[equation]{Definition}

 \newtheorem{example}[equation]{Example}
  \newtheorem*{example*}{Example}

\theoremstyle{remark}

\newtheorem{remark}[equation]{Remark}

\newcommand{\G}{\mathbb{G}}
\newcommand{\Z}{\mathbb{Z}}
\newcommand{\F}{\mathbb{F}}
\newcommand{\Fb}{\overline{\mathbb{F}}}

\newcommand{\Q}{\mathbb{Q}}

\newcommand{\X}{\mathbb{X}}

\newcommand{\N}{\mathbb{N}}

\newcommand{\R}{\mathbb{R}}
\newcommand{\C}{\mathbb{C}}

\newcommand{\CC}{\mathcal{C}}
\newcommand{\DD}{\mathcal{D}}
\newcommand{\EE}{\mathcal{E}}
\newcommand{\FF}{\mathcal{F}}
\newcommand{\GG}{\mathcal{G}}

\newcommand{\LL}{\mathcal{L}}
\newcommand{\PP}{\mathcal{P}}

\newcommand{\XX}{\mathcal{X}}

\newcommand{\YY}{\mathcal{Y}}
\newcommand{\ZZ}{\mathcal{Z}}
\newcommand{\OO}{\mathcal{O}}

\newcommand{\MM}{\mathcal{M}}

\newcommand{\WW}{\mathcal{W}}
\newcommand{\MMb}{\overline{\mathcal{M}}}

\newcommand{\tensor}{\otimes}

\DeclareMathOperator{\fin}{fin}
\DeclareMathOperator{\TMF}{TMF}
\DeclareMathOperator{\Tmf}{Tmf}
\DeclareMathOperator{\tmf}{tmf}

\DeclareMathOperator{\pt}{pt}

\DeclareMathOperator{\Gal}{Gal}
\DeclareMathOperator{\CAlg}{CAlg}
\DeclareMathOperator{\Aut}{Aut}

\DeclareMathOperator{\Sp}{Sp}
\DeclareMathOperator{\Cat}{Cat}

\DeclareMathOperator{\Orb}{Orb}

\DeclareMathOperator{\Shv}{Shv}

\DeclareMathOperator{\Mod}{Mod}
\DeclareMathOperator{\Pic}{Pic}

\DeclareMathOperator{\ev}{ev}

\DeclareMathOperator{\modules}{\text{-}mod}

\DeclareMathOperator{\id}{id}

\DeclareMathOperator{\pr}{pr}
\DeclareMathOperator{\sm}{\wedge}

\DeclareMathOperator{\Ho}{Ho}

\DeclareMathOperator{\Ext}{Ext}
\DeclareMathOperator{\Hom}{Hom}
\DeclareMathOperator{\Spec}{Spec}
\DeclareMathOperator{\holim}{holim}

\DeclareMathOperator{\QCoh}{QCoh}

\DeclareMathOperator{\Fun}{Fun}
\DeclareMathOperator{\Map}{Map}

\newcommand{\vb}{\overline{v}}
\newcommand{\wb}{\overline{w}}

\DeclareMathOperator{\Spet}{Sp\acute{e}t}
\DeclareMathOperator{\Ell}{Ell}
\DeclareMathOperator{\DerSt}{DerSt}
\DeclareMathOperator{\tr}{tr}
\DeclareMathOperator{\Euni}{\mathcal{E}^{uni}}

\DeclareMathOperator{\Etop}{\mathcal{E}^{top}}

\begin{document}
\title{Topological modular forms with level structure: Decompositions and duality}
\author{Lennart Meier}
\maketitle

\begin{abstract}
 Topological modular forms with level structure were introduced in full generality by Hill and Lawson. We will show that these decompose additively in many cases into a few simple pieces and give an application to equivariant $\TMF$. Furthermore, we show which $\Tmf_1(n)$ are self-Anderson dual up to a shift, both with and without their natural $C_2$-action.
\end{abstract}
\tableofcontents

\section{Introduction}
Elliptic cohomology has been married to equivariance from its early days on. Grojnowski \cite{Grojnowski} developed in 1994 an $S^1$-equivariant theory over the complex numbers and Devoto \cite{Devoto} pioneered work on equivariance for finite groups. Recently the theory has ripened, especially by work of Lurie \cite{Lur07} \cite{LurEllII}, who defined $G$-equivariant topological modular forms for any compact Lie group $G$. Recall here that while traditional elliptic cohomology was associated to specific elliptic curves over an affine base, the spectrum $\TMF$ of topological modular forms is associated to the \emph{universal} elliptic curve $\Euni$ over the moduli stack $\MM_{ell}$ of elliptic curves. In particular, the non-equivariant version of $\TMF$ arises as the global sections of a sheaf of $E_{\infty}$-ring spectra on $\MM_{ell}$. 

Equivariant elliptic cohomology in its different guises has seen many applications (see e.g.\ \cite{Rosu}, \cite{GKV}, \cite{AganagicOkounkovElliptic} and several articles by Ganter, in particular \cite{GanterHecke}) and it is promising to extend these applications using the strongest form of the theory, i.e.\ equivariant $\TMF$. One of the stumbling blocks is that the fixed points $\TMF^G$ (i.e.\ the value of $G$-equivariant topological modular forms at a point) have not been computed for almost any group $G$, with the exceptions of $C_2$ \cite{Chua} and tori \cite{GM20}. (One may compare this to the situation for $KO$ one chromatic level down, where $KO^G$ is completely known in terms of the representation theory of $G$.)

Clearly, the cyclic groups $C_n$ are the finite groups with most hope of a computation. Essentially by definition, the spectrum $\TMF^{C_n}$ arises as the global sections of a sheaf of $E_{\infty}$-rings on the $n$-torsion points $\Euni[n]$ of the universal elliptic curve. Much of the difficulty of the computation of $\TMF^{C_n}$ arises from the difficulty of the stack $\Euni[n]$. On the other hand, after inverting $n$ it splits into the pieces of exact order $k$ for $k|n$. More precisely, if $\MM_1(k)$ denotes the stack classifying elliptic curves with a chosen point of (exact) order $k$ over $\Z[\tfrac1k]$-schemes, we have
\[\Euni[n]_{\Z[\tfrac1n]} \simeq \coprod_{k|n} \MM_1(k)_{\Z[\tfrac1n]}.\]
This results in a corresponding splitting
\[\TMF^{C_n}\left[\tfrac1n\right] \simeq \prod_{k|n}\TMF_1(k)[\tfrac1n], \]
where $\TMF_1(k)$ arises as the global sections of a sheaf of $E_{\infty}$-ring spectra on $\MM_1(k)$. 

For small $k$, the spectra $\TMF_1(k)$ are quite well-understood. By definition, $\TMF_1(1) = \TMF$ and furthermore we have
\[\pi_*\TMF_1(2) \cong \Z[\tfrac12][b_2, b_4,\Delta^{-1}]\]
and
\[\pi_*\TMF_1(3) \cong \Z[\tfrac13][a_1, a_3,\Delta^{-1}],\]
where $\Delta$ is a certain polynomial in the $b_i$ and the $a_i$, respectively.

For larger $k$, we have $\pi_*\TMF_1(k) \cong A[u^{\pm 1}]$ with $|u| = 2$ and $\MM_1(k) = \Spec A$. Regrettably, these rings $A$ have only been computed for small values of $k$ due to the complexity of the multiplication-by-$k$ map on an elliptic curve. We refer to \cite[Section 1]{B-O16}, \cite[Section 4.4]{HusElliptic} and \cite[Section 3]{MSZ89} for a general method of computation and examples for $k$ small.

In contrast to these difficulties our first main result provides a complete \emph{additive} understanding of the spectra $\TMF_1(k)$ and hence of $\TMF^{C_n}$ after $l$-completion. 

\begin{thm}\label{thm:firstmain}
	Let $n \geq 2$ and $l$ be a prime not dividing $n$. Then $\TMF_1(n)$ splits after $l$-completion into a sum of shifted copies of the $l$-completions of
	\begin{itemize}
		\item $\TMF_1(3)$ if $l=2$,
		\item $\TMF_1(2)$ if $l=3$, or
		\item $\TMF$ if $l>3$.
	\end{itemize}
	Consequently, we have an analoguous splitting for the complement of $\TMF$ in $\TMF^{C_n}$. 
\end{thm}
In many cases it will be possible to replace $l$-completions by $l$-localizations using compactified moduli stacks as we will explain below. Moreover we can replace $C_n$ by any finite abelian group of order prime to $l$. 

Because of their relation to power operations, maybe the most interesting groups to extend these computations to are the symmetric groups $\Sigma_n$. While a full computation of $\TMF^{\Sigma_n}$ still seems very daunting, there is a useful map 
\[\TMF^{\Sigma_n} \to \TMF_0(n) = \TMF_1(n)^{h(\Z/n)^\times}\]
 coming from the action of $(\Z/n)^\times$ on the subgroup $C_n \subset \Sigma_n$.\footnote{In conjunction with a hypothetical equivariant string orientation, this would lead to a genus $MString^{\Sigma_n} \to \TMF_0(n)$.} The only computations of $\pi_*\TMF_0(n)$ in the literature are for $n \leq 7$. In contrast, we will not only compute $\pi_*\TMF_0(n)$ in infinitely many cases, but also provide splitting results in these cases. 

\begin{thm}
	If $n\geq 3$ is odd and, $\varphi(n) = |(\Z/n)^\times|$ is not divisible by $4$, then the $\TMF$-module $\TMF_0(n)$ decomposes into copies of $(\Sigma^{k\rho}\TMF_1(3))^{hC_2}$ after $2$-completion. 
\end{thm}
Here $\Sigma^{k\rho}$ means smashing with the one point compactification of the representation $k\rho$, where $\rho$ is the regular real representation of $C_2$. Moreover, the $C_2$-action on $\TMF_1(n)$ is induced by the corresponding one on $\MM_1(n)$ that sends $(E, P)$ to $(E, -P)$, where $P$ is a point of order $n$. Note that in the case $k=0$, we obtain $\TMF_1(3)^{hC_2} \simeq \TMF_0(3)$ so that $(\Sigma^{k\rho}\TMF_1(3)^{hC_2}$ can be seen as equivariant shifts of $\TMF_0(3)$. The homotopy groups of $(\Sigma^{k\rho}\TMF_1(3))^{hC_2}$ are completely known \cite{H-M17}. 


We remark  that our theorem is a consequence of a more fundamental theorem that computes the slices of $\TMF_1(n)$ and related spectra with respect to their natural $C_2$-action and thus makes Theorem \ref{thm:firstmain} $C_2$-equivariant (see Example \ref{exa:tmf} and Theorem \ref{thm:C2}). The main method is to use a $C_2$-equivariant map $\TMF_1(n) \to E_2$ to Morava E-theory. The $C_2$-homotopy fixed point spectral sequence of the latter has been recently computed by Hahn and Shi \cite{HahnShi}. We provide a general method to transfer $RO(C_2)$-graded homotopy fixed point spectral sequence computations along suitable maps that are ``injective enough''. We expect our method to have further applications in related situations.

\subsection*{Compactifications and duality}
The theorems above have counterparts for the spectra $\Tmf_1(n)$ associated with the compactified moduli stacks $\MMb_1(n)$ and their homotopy fixed points $\Tmf_0(n) = \Tmf_1(n)^{h(\Z/n)^\times}$ as defined by Hill and Lawson \cite{HL13}. One reason for the importance of working with the compactifiactions is that in arithmetics, the ring of integral modular forms $M_*(\Gamma_1(n);\Z[\tfrac1n])$ of level $n$ arises as sections of line bundles on $\MMb_1(n)$ and indeed there is an isomorphism $\pi_{2*}\Tmf_1(n) \cong M_*(\Gamma_1(n);\Z[\tfrac1n])$ for $n\geq 2$. The spectra $\Tmf_1(n)$ and $\Tmf_0(n)$ are also the first step to define connective versions $\tmf_1(n)$ and $\tmf_0(n)$, as investigated in \cite{meier2021connective}. 

In the case $n=1$, the theory specializes to $\Tmf$ and $\tmf$, the latter having been central to many of the applications of topological modular forms (see e.g.\ \cite{BHHM08}, \cite{BHHM}). We expect that our results on $\Tmf_1(n)$ and $\Tmf_0(n)$ might be relevant for hypothetical equivariant versions of $\Tmf$ and $\tmf$. 

\begin{thm}\label{thm:introTmf1}
	Let $n\geq 2$ and $l$ be a prime not dividing $n$. Then $\Tmf_1(n)_{(l)}$ splits as a $\Tmf$-module into a sum of shifted copies of
	\begin{itemize}
		\item $\Tmf_1(3)_{(2)}$ if $l=2$,
		\item $\Tmf_1(2)_{(3)}$ if $l=3$, and
		\item $\Tmf_{(l)}$ if $l>3$
	\end{itemize}
	if and only if $\pi_1\Tmf_1(n)$ is $l$-torsionfree. This happens if and only if every modular form of weight $1$ over $\F_l$ for the group $\Gamma_1(n)$ can be lifted to a modular form of weight $1$ over $\Z_{(l)}$ for $\Gamma_1(n)$.\footnote{This condition is known to be satisfied for $n\leq 28$, but fails in general. Some known counterexamples are $l=3$ and $n=74$ and $l=199$ and $n=82$ \cite{Buz}.} 
\end{thm}
Similarly to before, there is a refined splitting result for $\Tmf_0(n)$. 
\begin{thm}\label{thm:Tmf0n}
	If $n\geq 3$ is odd, $\varphi(n)$ is not divisible by $4$ and $(\pi_1\Tmf_1(n))^{(\Z/n)^\times}$ vanishes, then the $\Tmf_{(2)}$-module $\Tmf_0(n)_{(2)}$ decomposes into copies of $(\Sigma^{k\rho}\Tmf_1(3)_{(2)})^{hC_2}$.
\end{thm}
Here the $C_2$-action on $\Tmf_1(n)$ is induced again by the corresponding one on $\MM_1(n)$ that sends $(E, P)$ to $(E, -P)$, where $P$ is a point of order $n$. The homotopy groups of $(\Sigma^{k\rho}\Tmf_1(3)_{(2)})^{hC_2}$ are completely known \cite{H-M17}. Again, Theorem \ref{thm:Tmf0n} is a special case of a more fundamental result, namely the $C_2$-equivariant refinement Theorem \ref{thm:C2cpt} of Theorem \ref{thm:introTmf1}, of which we give the following example.
\begin{example}
	We have a $C_2$-equivariant decomposition 
	\[\Tmf_1(7)_{(2)} \simeq \Tmf_1(3) \oplus \Sigma^{\rho}\Tmf_1(3)^{\oplus 2} \oplus \Sigma^{2\rho}\Tmf_1(3)^{\oplus 2}\oplus \Sigma^{3\rho}\Tmf_1(3).\]
	The corresponding non-equivariant decomposition results from changing each $\rho$ to $2$. Moreover, we obtain 
	\[\Tmf_0(7)_{(2)} \simeq \Tmf_0(3) \oplus (\Sigma^{3\rho}\Tmf_1(3))^{hC_2}.\]
	The concrete numbers in these decompositions are based on a comparison of dimensions of spaces of modular forms. See \cite[Section 4 and Appendix C]{MeiDecMod} for the analogous algebraic story. See also \cite{M-O20} for an analogous, but more subtle splitting at the prime $3$. 
\end{example}

A nice feature of the compactified moduli stacks $\MMb_1(n)$ is that they satisfy a form of Serre duality. This has applications for the Anderson duals $I_{\Z[\tfrac1n]}\Tmf_1(n)$ of $\Tmf_1(n)$. The Anderson dual of a spectrum $X$ is defined so that one has a short exact sequence
\[0 \to \Ext^1_{\Z}\left(\pi_{-k-1}X, \Z[\tfrac1n]\right) \to \pi_kI_{\Z[\tfrac1n]}X \to \Hom_{\Z}\left(\pi_{-k}X, \Z[\tfrac1n]\right) \to 0. \]
If $X$ is Anderson self-dual (up to suspension), then one obtains a convenient universal coefficient sequence for $X$ (see e.g.\ \cite{Kainen} or Section \ref{sec:AndersonDuality}). The following theorem was obtained with significant input from Viktoriya Ozornova.
\begin{thm}The Anderson dual $I_{\Z[\tfrac1n]}\Tmf_1(n)$ of $\Tmf_1(n)$ is equivalent to a suspension of $\Tmf_1(n)$ if and only if $1\leq n \leq 8$ or $n=11,14,15$ or $n=23$.
\end{thm}
The cases $n=1$ and $n=2$, of which the former is in some sense the most difficult, were already obtained by Stojanoska \cite{Sto12}. In Remark \ref{rem:AndersonC2} we will explain how to refine this result $C_2$-equivariantly with respect to the $C_2$-action on $\Tmf_1(n)$ mentioned before.

\subsection*{Structure of the article}
The structure of the present article is as follows. We begin in Section \ref{sec:TMFLevel} with a review of moduli stacks of elliptic curves with level structures and the associated spectra of topological modular forms; we also prove that $\Tmf_1(n),\, \Tmf(n)$ and $\Tmf_0(n)$ are finite as modules over $\Tmf[\tfrac1n]$. In Section \ref{sec:splittings}, we recall and extend certain algebraic splitting results from \cite{MeiDecMod} and deduce our main topological splitting results. The next two sections contain our applications to equivariant $\TMF$ (Section \ref{sec:equivariant}) and our results on Anderson duality (Section \ref{sec:dualityboth}).
In Section \ref{sec:C2}, we give compute the $C_2$-equivariant slices of $\Tmf_1(n)$ and prove $C_2$-equivariant refinements of our splitting results and Corollary \ref{cor:0n}. On the way, we introduce a new method to compute $RO(C_2)$-graded homotopy fixed point spectral sequences. Section \ref{sec:equivariant} and Sections \ref{sec:dualityboth} and \ref{sec:C2} can be read independently. Appendix \ref{app:squarefree} is about some subleties about quotients of $\MMb_1(n)$ if $n$ is not squarefree and Appendix \ref{sec:DAG} embeds our simplified treatment of derived algebraic geometry into the work of Lurie.  

\subsection*{Conventions}
We list a few conventions.
\begin{itemize}
	\item The symbol $\subset$ means for us the same as $\subseteq$, i.e.\ equality is possible,
	\item we often use the notation $C_2$ for the group with two elements,
	\item quotients of schemes by group actions are always understood to be stack quotients,
	\item tensor products over unspecified base are always over the structure sheaf (if this applies),
	\item smash product are always derived,
	\item homotopy groups of sheaves of spectra are always understood to be sheafified.
\end{itemize}

\subsection*{Acknowledgments}
I want to thank Viktoriya Ozornova for helpful discussions and for providing the proof of Proposition \ref{prop:degreecomp}. Moreover, I want to thank her, Dimitar Kodjabachev and the anonymous referee for comments on earlier versions of this paper. Furthermore, thanks to Mike Hill for suggesting to refine the main theorems $C_2$-equivariantly and for teaching me in our previous work enough equivariant homotopy theory to actually prove his conjecture.

\section{$\TMF$ and $\Tmf$ with level structure}\label{sec:TMFLevel}
Our main goal in this section is to recall the definition of topological modular forms with level structure and prove some of their basic properties. In particular, we will show that in the \emph{tame} cases the corresponding descent spectral sequence will collapse at $E_2$. In Subsection \ref{sec:DAG} we will demonstrate how topological modular forms with level structure arise from derived stacks. After proving some more general statements about quasi-coherent sheaves we will show in Proposition \ref{prop:compactness} an important finiteness condition.

\subsection{Definitions and basic properties}
Denote by $\MM_{ell}$ the moduli stack of elliptic curves and by $\MMb_{ell}$ its compactification. Let $\omega$ be the line bundle $p_*\Omega^1_{\Euni/\MMb_{ell}}$ for $p\colon \Euni \to \MMb_{ell}$ the universal generalized elliptic curve. Goerss, Hopkins and Miller defined a (hypercomplete) sheaf of $E_\infty$-ring spectra $\OO^{top}$ on the \'etale site of the compactified moduli stack $\MMb_{ell}$ of elliptic curves (see \cite{TMF}). It satisfies $\pi_0\OO^{top} \cong \OO_{\MMb_{ell}}$ and more generally $\pi_{2k}\OO^{top} \cong \omega^{\tensor k}$ and $\pi_{2k-1}\OO^{top} = 0$. One defines $\TMF = \OO^{top}(\MM_{ell})$ and $\Tmf = \OO^{top}(\MMb_{ell})$.

In this paper, we care about moduli with level structure and we define stacks
$\MM_0(n)$, $\MM_1(n)$ and $\MM(n)$ as the groupoid-valued (pseudo-)functors given by
\begin{align*}
 \MM_0(n)(S) &= \text{ Elliptic curves }E \text{ over }S\text{ with chosen cyclic subgroup }H\subset E\text{ of order }n \\
 \MM_1(n)(S) &= \text{ Elliptic curves }E \text{ over }S\text{ with chosen point }P\in E(S)\text{ of exact order }n \\
 \MM(n)(S) &= \text{ Elliptic curves }E \text{ over }S\text{ with chosen isomorphism }(\Z/n)^2\cong E[n](S).
\end{align*}
Here we always assume $n$ to be invertible on $S$ and $E[n]$ denotes the $n$-torsion points. More precisely, we demand for $\MM_1(n)$ that for every geometric point $s\colon \Spec K \to S$ the pullback $s^*P$ spans a cyclic subgroup of order $n$ in $E(K)$. A subgroup scheme $H \subset E$ is called \emph{cyclic of order $n$} if it is \'etale locally isomorphic to $\Z/n$. We note that $\MM_1(n)$ has a $(\Z/n)^\times$-action, where $[k]\in (\Z/n)^\times$ sends $(E,P)$ to $(E, kP)$. One can easily see that $\MM_0(n)$ is equivalent to the stack quotient $\MM_1(n)/(\Z/n)^\times$.

We will also use the notation $\MM(\Gamma)$ with $\Gamma = \Gamma_1(n), \Gamma(n)$ and $\Gamma_0(n)$ for $\MM_1(n)$, $\MM(n)$ and $\MM_0(n)$; this reflects the fact that the complex points of these stacks form orbifolds that are equivalent to $\Gamma \!\setminus\! \mathbb{H}$ for $\mathbb{H}$ the upper half-plane.\footnote{We refer to \cite{D-S05} for background on the congruence subgroups $\Gamma_1(n)$, $\Gamma(n)$ and $\Gamma_0(n)$ and their relationship to moduli of elliptic curves. This material is though barely necessary for the present paper as we use the congruence subgroups primarily as notation. We remark on one subletly though: The Weil pairing allows to define a map $\MM(n) \to \Spec \Z[\tfrac1n, \zeta_n]$. We only recover $\Gamma(n)\!\setminus\! \mathbb{H}$ as the $\C$-points of $\MM(n)$ if we view $\MM(n)$ as a scheme over $\Z[\tfrac1n, \zeta_n]$. As $\MM_1(n)$ and $\MM_0(n)$ do not allow generally for morphisms into $\Spec \Z[\tfrac1n, \zeta_n]$ and $\C$-points have in this article little relevance for us, we will generally view the stacks $\MM(n)$, $\MM_1(n)$ and $\MM_0(n)$ though as stacks over $\Spec \Z[\tfrac1n]$.} More generally, we define for a subgroup $\Gamma \subset \Gamma_0(n)$ containing $\Gamma_1(n)$ the moduli stack $\MM(\Gamma)$ to be $\MM_1(n)/(\Gamma/\Gamma_1(n))$, where $\Gamma/\Gamma_1(n)$ is identified with a subgroup of $(\Z/n)^\times$ by mapping an element of $\Gamma \subset SL_2(\Z)$ to its upper left entry.

\begin{defi}
We call $\Gamma\subset SL_2(\Z)$ a \emph{congruence subgroup of level $n$} if $\Gamma = \Gamma(n)$ or $\Gamma_1(n) \subset \Gamma \subset \Gamma_0(n)$.\footnote{Usually, this term means something slightly more general, but we will always use it in this sense.}
\end{defi}

 We can define the compactified moduli stacks $\MMb(\Gamma)$ as the normalization of $\MMb_{ell}$ in $\MM(\Gamma)$ \cite[IV.3]{D-R73}. We will generically use the notation $g$ for the projection maps $\MMb(\Gamma) \to \MMb_{ell,\Z[\tfrac1n]}$, which are finite maps as in \cite[Proposition 2.4]{MeiDecMod}. For more background on these stacks we refer to \cite{D-R73}, \cite{Con07} and \cite{MeiDecMod}.\\

 As $\MM_0(n), \MM_1(n)$ and $\MM(n)$ are \'etale over $\MMb_{ell}$, we can define $\TMF(\Gamma)$ as $\OO^{top}(\MM(\Gamma))$. We will also use the notations $\TMF_1(n)$ and $\TMF_0(n)$ in the cases of $\Gamma$ being $\Gamma_1(n)$ or $\Gamma_0(n)$.

In contrast, $\MMb(\Gamma)$ is in general not \'etale over $\MMb_{ell}$. As a remedy, Hill and Lawson extended in \cite{HL13} the (hypercomplete) sheaf $\OO^{top}$ to the log-\'etale site of the compactified moduli stack $\MMb_{ell}$. The maps $\MMb(\Gamma) \to \MMb_{ell}$ are log-\'etale by \cite[Proposition 3.19]{HL13}. Thus, we can additionally define $\Tmf(\Gamma) = \OO^{top}(\MMb(\Gamma))$. Again, we will also use the notations $\Tmf_1(n)$ and $\Tmf_0(n)$ for special cases. We stress again that the integer $n$ is implicitly inverted if $\Gamma$ is of level $n$.

By the functoriality of normalization, $\MMb_1(n)$ inherits a $(\Z/n)^\times$-action from $\MM_1(n)$. There is the subtlety that for $\Gamma_0(n)\subset \Gamma \subset \Gamma_1(n)$ the stack $\MMb(\Gamma)$ is in general not equivalent to the stack quotient $\MMb(\Gamma)' = \MMb(\Gamma_1(n))/(\Gamma/\Gamma_1(n))$ if $n$ is not square-free. Nevertheless, \cite[Theorem 6.1]{HL13} implies that $\Tmf(\Gamma)$ is always equivalent to the homotopy fixed points $\Tmf_1(n)^{h\Gamma/\Gamma_1(n)}$ and in particular $\Tmf_0(n) \simeq \Tmf_1(n)^{h(\Z/n)^\times}$. We will provide more details about $\MMb(\Gamma)$ in Appendix \ref{app:squarefree} and give in particular an alternative proof of the theorem of Hill and Lawson.\\

Next, we want to single out a class of congruence subgroups that is particularly easy to work with.
\begin{defi}\label{def:tame}Fixing a prime $l$, we say that a congruence subgroup $\Gamma$ is \emph{tame} if $n\geq 2$ and the prime $l$ does not divide $n$; in the case $\Gamma_1(n)\subset \Gamma\subset \Gamma_0(n)$ we demand additionally that $l$ does not divide $\gcd(6,[\Gamma_1(n):\Gamma])$. \end{defi}
This condition ensures that every automorphism of a point in $\MMb(\Gamma)$ has invertible order, in which case one also calls the stack  $\MMb(\Gamma)$ tame.  Note that the index $[\Gamma_1(n): \Gamma_0(n)]$ agrees with $\varphi(n)= |(\Z/n)^\times|$.

Generally, there is a descent spectral sequence
\begin{align}\label{eq:DSS}
E_2^{pq} = H^q(\MMb(\Gamma);g^*\omega^{\tensor p}) \Rightarrow \pi_{2p-q}\Tmf(\Gamma)
\end{align}
as in \cite[Chapter 5]{TMF}. Actually, we obtain for every localization $R$ of the integers a spectral sequence
\begin{align}\label{eq:DSSR}
E_2^{pq} = H^q(\MMb(\Gamma)_R;g^*\omega^{\tensor p}) \Rightarrow \pi_{2p-q}\Tmf(\Gamma) \tensor R
\end{align}
as \eqref{eq:DSS} has eventually a horizontal vanishing line by Proposition \ref{prop:0affine}.

 The basic structure of the $E_2$-term of the descent spectral sequence in the tame case is implied by the following proposition. If $\Gamma = \Gamma_0(n), \Gamma_1(n)$ or $\Gamma(n)$, then it is contained in \cite{MeiDecMod}, Propositions 2.4 and 2.13, and the general case is proven analogously using Lemma \ref{lem:compactificationsGamma}.

\begin{prop}\label{prop:coh}
 Let $R$ be a localization of the integers and $\Gamma$ be a congruence subgroup of some level that is tame for every prime $l$ not invertible in $R$. Then
\begin{enumerate}
 \item $H^0(\MMb(\Gamma)_R;g^*\omega^{\tensor m}) = 0$ for $m<0$ (there are no modular forms of negative weight),
 \item $H^1(\MMb(\Gamma)_R;g^*\omega^{\tensor m}) = 0$ is zero for $m>1$ and torsionfree for $m<1$,
 \item $H^i(\MMb(\Gamma)_R;\FF) = 0$ for $i>1$ and any quasi-coherent sheaf $\FF$ on $\MMb(\Gamma)_R$.
 \item $H^i(\MM(\Gamma)_R;\FF) = 0$ for $i>0$ and any quasi-coherent sheaf $\FF$ on $\MM(\Gamma)_R$.
\end{enumerate}
\end{prop}

Thus for $R$ and $\Gamma$ as in the proposition, there can be neither
 differentials nor extension issues in the descent spectral sequence for $\pi_*\Tmf(\Gamma)\tensor R$. Thus, $\pi_{even}\Tmf(\Gamma)\tensor R$ is just $H^0(\MMb(\Gamma);\omega^{\tensor *})\tensor R$, i.e.\
  the ring of $R$-integral modular forms for $\Gamma$.  The odd homotopy $\pi_{2k+1}\Tmf(\Gamma)\tensor R$ vanishes for $2k+1 > 1$. Moreover, $\Hom(\pi_{3-2k}\Tmf(\Gamma), R)$ can be identified with $S_k(\Gamma;R)$, the weight $k$ cusp forms for $\Gamma$ with coefficients in $R$, at least for $\MMb(\Gamma)$ representable (see e.g.\ \cite[Proposition 2.11]{MeiDecMod}).

  We picture the descent spectral sequence for $\pi_*\Tmf_1(23)$ in the range $-10\leq \ast \leq 10$. A boxed $k$ stands for a free $\Z[\tfrac1{23}]$-module of rank $k$. The vertical axis corresponds to $q$, the horizontal one to $2p-q$ in the convention from
  \eqref{eq:DSS}.

  \[\begin{sseq}[ylabelstep = 1, entrysize =5.7mm]
{-10...10}{0...1}
\ssdropboxed{1}
\ssmove{2}{0} \ssdropboxed{12}
\ssmove{2}{0} \ssdropboxed{33}
\ssmove{2}{0} \ssdropboxed{55}
\ssmove{2}{0} \ssdropboxed{77}
\ssmove{2}{0} \ssdropboxed{99}
\ssmoveto{1}{1} \ssdropboxed{1}
\ssmove{-2}{0} \ssdropboxed{12}
\ssmove{-2}{0} \ssdropboxed{33}
\ssmove{-2}{0} \ssdropboxed{55}
\ssmove{-2}{0} \ssdropboxed{77}
\ssmove{-2}{0} \ssdropboxed{99}
\end{sseq}
\]

We chose the example $\Tmf_1(23)$ because it is the first $\Tmf_1(n)$ with $n>1$ that has nontrivial $\pi_1$ and furthermore because it displays a remarkable symmetry, which can be explained by Anderson duality (see Section \ref{sec:dualityboth}).

Note that in general $\pi_1\Tmf_1(n)$ might contain torsion even for $n\geq 2$, namely if $H^1(\MMb_1(n); g^*\omega)$ contains torsion. As explained in \cite[Remark 3.14]{MeiDecMod}, this happens if and only if there is a weight $1$ modular form for $\Gamma_1(n)$ over $\F_l$ that does not lift to $\Z_{(l)}$. The minimal $n$ for $l=2$ where this occurs is $n=65$ and this seems to be the minimal $n$ known at any prime.

We finish this section with one result about the torsion in $3$-local (non-tame) $\Tmf_0(n)$, which we will not need later but which is an application of the results in \cite{MeiDecMod}.
\begin{prop}
For all $n\geq 1$ and $i\geq 1$, the map $\phi_*\colon \pi_i\Tmf_0(n)_{(3)} \to \pi_i\TMF_0(n)_{(3)}$ has all torsion classes in its image. If $i>1$, the preimage of every torsion class can be chosen to be torsion as well. 
\end{prop}
\begin{proof}
Let $i\geq 1$ and $x \in \pi_i\TMF_0(n)_{(3)}$ be a torsion class that reduces to an non-trivial permanent cycle $\overline{x}\in H^q(\MM_0(n)_{(3)};\omega^{\tensor p})$ in the descent spectral sequence with $i = 2p-q$. As $H^0(\MM_0(n)_{(3)}; \omega^{\tensor p})$ is torsionfree, we have $q> 0$.

By \cite[Lemma 5.9]{MeiDecMod} the map
$$H^q(\MMb_0(n)_{(3)};g^*\omega^{\tensor p}) \to H^q(\MM_0(n)_{(3)};g^*\omega^{\tensor p})$$
is an isomorphism for all $q\geq 2$ and for $q=1$ it is a surjection if $p\geq 1$. Thus, we can choose a preimage $\overline{y}$ of $\overline{x}$ and this is automatically a permanent cycle. A corresponding element in the $q$-th filtration $y\in F^q\pi_i \Tmf_0(n)_{(3)}$ maps to $x$ modulo $F^{q+1}\pi_i \TMF_0(n)_{(3)}$. By the same argument we can choose a preimage of $\phi(y)-x$ modulo higher filtration and so on. The eventual horizontal vanishing line from Proposition \ref{prop:0affine} implies that this process terminates and that there is actually a choice of $y$ with $\phi_*(y) = x$.

If $i>1$, \cite[Proposition 2.8]{MeiDecMod} and the eventual horizontal vanishing line imply that all elements of $\pi_i\Tmf_0(n)$ of positive filtration are torsion. No element of filtration $0$ can map to a torsion element in $\pi_i\TMF_0(n)$ as
$$H^0(\MMb_0(n); \omega^{\tensor \ast}) \to H^0(\MM_0(n);\omega^{\tensor \ast})$$
is injective because the source is an integral domain by \cite[Proposition 2.13]{MeiDecMod}.
\end{proof}

\subsection{Derived algebraic geometry and a compactness statement}\label{sec:DAG}
The first goal of this section is to explain how the construction of $\Tmf(\Gamma)$ for congruence subgroups $\Gamma$ fits into the language of derived algebraic geometry. To this purpose we will use the following definition. 
\begin{defi}\label{defi:der}
	Let $\XX_0$ be a Deligne--Mumford stack and let $\XX_0^{\acute{e}t, \mathrm{aff}}$ be the site of affine schemes with an \'etale map to $\XX_0$. Let further $\OO_{\XX}$ be a sheaf of $E_{\infty}$-ring spectra on $\XX_0^{\acute{e}t, \mathrm{aff}}$ together with an isomorphism $\pi_0\OO_{\XX} \cong \OO_{\XX_0}$. 
	
	Then we call $(\XX_0, \OO_{\XX})$ a \emph{derived stack} if 
	the assignment $U\mapsto \pi_n(\OO_{\XX}(U))$ defines for every $n\in\Z$ a quasi-coherent sheaf on $\XX_0^{\acute{e}t, \mathrm{aff}}$.
\end{defi}
This is actually a special case of Lurie's notion of a non-connective spectral Deligne--Mumford stack as will be explained in Lemma \ref{lem:derivedstacksequivalence} in the appendix. We have chosen to work with the simpler concept above to ease the life of readers not fully fluent in spectral algebraic geometry. We remark further that there are other concepts that might deserve the name derived stack, but we stick to our terminology for ease of expression.

%

The main result of \cite{HL13} actually equips for $\Gamma = \Gamma_1(n),\, \Gamma(n)$ or $\Gamma_0(n)$ the \'etale site of $\MMb(\Gamma)$ with a sheaf of $E_\infty$-ring spectra $\OO^{top}_{\MMb(\Gamma)}$ as the composite of an \'etale with a log-\'etale map is still log-\'etale. It follows from the construction in \cite{HL13} that $\pi_{2k} (\OO^{top}_{\MMb(\Gamma)}(U)) \cong g^*\omega^{\tensor k}(U)$ for every affine scheme $U$ with an \'etale map to $\MMb(\Gamma)$ and that the odd homotopy is zero. Note that these sheaves are quasi-coherent and thus $(\MMb(\Gamma),\OO^{top}_{\MMb(\Gamma)})$ is a derived stack. We will sometimes abbreviate this derived stack to $\MMb(\Gamma)^{top}$ (and likewise in similar contexts). 

Clearly, the map $g\colon \MMb(\Gamma) \to \MMb_{ell}$ induces a map $\MMb(\Gamma)^{top} \to \MMb_{ell}^{top}$ of derived stacks, which we denote by the same letter. This map is flat in the sense of \cite[Definition 2.8.2.1]{SAG} as $g$ is flat (e.g.\ by \cite[Proposition 2.4]{MeiDecMod}) and $\pi_i\OO^{top}_{\MMb(\Gamma)} \cong g^*\pi_i\OO^{top}$.

To every derived stack $(\XX, \OO_{\XX})$ we can associate an $\infty$-category $\QCoh(\XX)$ of quasi-coherent $\OO_{\XX}$-modules (see \cite[Section 2]{SAG} for a full treatment or \cite[Section 2.3]{MM15} for a short summary in our context). If $(\Spec A, \OO_{\XX})$ is a derived stack, then it can be identified with the spectrum of the $E_{\infty}$-ring $\Gamma(\OO_{\XX})$ of global sections \cite[Corollary 1.4.7.3]{SAG}. In this case the functor $\QCoh(\Spec A, \OO_{\XX}) \to \Gamma(\OO_{\XX})\modules$ is an equivalence of $\infty$-categories \cite[Proposition 2.2.3.3]{SAG}.

We will need the following lemma.
\begin{lemma}\label{lem:pushpull}Let $g\colon \XX \to \YY$ be a morphism of derived stacks. Denote by $g_0\colon \XX_0 \to \YY_0$ the morphism of underlying Deligne--Mumford stacks. Assume that $g_0$ is affine and $\YY_0$ separated. Then $g$ induces an adjunction
\[\xymatrix{ \QCoh(\XX) \ar@<0.5ex>[r]^-{g_*} & \QCoh(\YY) \ar@<0.5ex>[l]^-{g^*} }\]
and we have a natural isomorphism $\pi_*g_*\FF \cong (g_0)_*\pi_*\FF$ for every $\FF \in \QCoh(\XX)$.

If $g$ is moreover flat, then we have a natural isomorphisms $\pi_*g^*\GG \cong (g_0)^*\pi_*\GG$ for every $\GG \in \QCoh(\YY)$ as well.
\end{lemma}
\begin{proof}
The adjunction
\[\xymatrix{ \Mod_{\OO_{\XX}} \ar@<0.5ex>[r]^-{g_*} & \Mod_{\OO_{\YY}}\ar@<0.5ex>[l]^-{g^*} }\]
from \cite[Section 2.5]{SAG} restricts to the $\infty$-categories of quasi-coherent sheaves by Propositions 2.5.0.2 and 2.5.1.1 from \cite{SAG}. Here we use that $g$ is affine by \cite[Remark 2.4.4.2]{SAG}.

We want to show that $g_*$ commutes with $\pi_*$. Let $\FF$ be in $\QCoh(\XX)$. We can identify $\FF$ with an $\OO_{\XX}$-module on the \'etale site of $\XX_0$. Consider a cartesian square
\begin{align}\label{eq:cartesian}
\xymatrix{
\Spec B \ar[r]^-{j} \ar[d]^h & \XX_0\ar[d]^{g_0}\\
\Spec A \ar[r]^-i & \YY_0
}
\end{align}
with $i$ (and hence $j$) \'etale. As follows from \cite[Remark 1.3.2.8]{SAG}, $(g_*\FF)(\Spec A)$ is naturally equivalent to $\FF(\Spec B)$.

 Although $\pi_*\FF$ denotes in general only the sheafification of the presheaf of objectwise homotopy groups of $\FF$, we have nevertheless a natural isomorphism $\pi_*(\FF(\Spec B)) \cong (\pi_*\FF)(\Spec B)$. Indeed, the descent spectral sequence to calculate $\pi_*(\FF(\Spec B))$ collapses as $H^i(\Spec B, \pi_*\FF) = 0$ for $i>0$ (here we use that $\pi_*\FF$ is a quasi-coherent $\OO_{\XX_0}$-module by \cite[Proposition 2.2.6.1]{SAG})). Thus, we obtain a natural isomorphism $\pi_*g_*\FF \cong (g_0)_*\pi_*\FF$ on the category of affine schemes \'etale over $\YY_0$. This forms a subsite of the full \'etale site of $\YY_0$ as $\YY_0$ is separated and is indeed equivalent to it. Thus, $\pi_*g_*\FF \cong (g_0)_*\pi_*\FF$.
 
Let $\GG$ be a quasi-coherent sheaf on $\YY$. The site of \'etale morphisms to $\XX_0$ that factor over $\Spec A \times_{\YY_0} \XX_0 \xrightarrow{\pr_2}\XX_0$ for some \'etale map $i\colon \Spec A \to \YY_0$ is equivalent to the full \'etale site of $\XX_0$. Thus, it suffices to provide a natural isomorphism
$$(\pi_*g^*\GG)(\Spec B) \cong (g_0^*\pi_*\GG)(\Spec B)$$
for $B$  as in the cartesian square \eqref{eq:cartesian}above and $i\colon \Spec A \to \YY_0$ an arbitrary \'etale map.
We have the following chain of natural isomorphisms:
\begin{align*}
(\pi_*g^*\GG)(\Spec B) &\cong \pi_*((g^*\GG)(\Spec B))\\
&\cong \pi_*((j^*g^*\GG)(\Spec B)) \\
&\cong \pi_*((h^*i^*\GG)(\Spec B))
\end{align*}
As $i$ is \'etale, $i^*$ is given by restriction to a subsite and commutes thus with $\pi_*$. As $h$ is a morphism between affine schemes, $h^*$ can be identified with the base change morphism
$$\tensor_{\OO_{\YY}(\Spec A)}\OO_{\XX}(\Spec B)\colon \OO_{\YY}(\Spec A)\modules  \to \OO_{\XX}(\Spec B)\modules.$$
By a degenerate K\"unneth spectral sequence we see that $$\pi_*((h^*\GG)(\Spec B)) \cong \pi_*\GG(\Spec A) \tensor_{\pi_*\OO_{\YY}(\Spec A)} \pi_*\OO_{\XX}(\Spec B)$$
as $\pi_*\OO_{\XX}(\Spec B) \cong \pi_*\OO_{\YY}(\Spec A) \tensor_A B$ is flat over $\pi_*\OO_{\YY}(\Spec A)$. Here we use that the morphism $g$ and hence the morphism $\OO_{\YY}(\Spec A) \to \OO_{\XX}(\Spec B)$ of $E_\infty$-rings is flat.
\end{proof}

Next, we will explain how to form Hom sheaves between $\OO$-modules $\FF$ and $\GG$ for $\OO$ a sheaf of $E_\infty$-rings on a site $\CC$. By \cite[Remark 2.1.5]{SAG} there is a sheaf $\mathfrak{M}\colon \CC^{op} \to \widehat{\Cat}_\infty$ with values in the $\infty$-category of (possibly large) $\infty$-categories that assigns to each $U\in\CC$ the $\infty$-category of $(\OO|_{\CC/U})\modules$ and the maps are given by pullback of sheaves. The choice of $\FF$ and $\GG$ allows to lift $\mathfrak{M}$ to a functor $\mathfrak{M}_{\FF,\GG}\colon \CC^{op} \to (\widehat{\Cat}_\infty)_{\Delta^0\sqcup \Delta^0/}$ with values in $\infty$-categories under $\Delta^0\sqcup\Delta^0$. This functor is a sheaf as well because analogously to \cite[1.2.13.8]{HTT}, the forgetful functor $(\widehat{\Cat}_\infty)_{\Delta^0\sqcup \Delta^0/} \to \widehat{\Cat}_\infty$ detects limits.

Let $I$ be the $\infty$-category under $\Delta^0\sqcup \Delta^0$ given by the inclusion $\Delta^0\sqcup \Delta^0\hookrightarrow \Delta^1$ of end points.  We define the space valued Hom sheaf between $\FF$ and $\GG$ as the functor
$$\mathcal{H}om^{\mathcal{S}}_{\OO}(\FF,\GG) = \Map_{\widehat{\Cat}_{\infty, \Delta^0\sqcup \Delta^0 /-}}(I, -) \circ \mathfrak{M}_{\FF,\GG}\colon \CC^{op} \to \mathcal{S}$$
into the $\infty$-category $\mathcal{S}$ of spaces. It is a sheaf as mapping out of $I$ preserves limits. For an arbitrary $\infty$-category $\DD$ together with a morphism $\Delta^0\sqcup \Delta^0 \xrightarrow{(X,Y)} \DD$, the space of morphisms $I \to \DD$ under $\Delta^0\sqcup \Delta^0$ is equivalent to the space of morphisms from $X$ to $Y$ in $\DD$ \cite[Proposition 1.2]{D-S11}. Thus, we obtain an equivalence
$$(\mathcal{H}om_{\OO}^{\mathcal{S}}(\FF,\GG))(U)\simeq \Map_{\OO|_{\CC/U}}(\FF|_{\CC/U}, \GG|_{\CC/U})$$
that is natural in $\FF$ and $\GG$.

Before we lift this functor to a spectrum-valued functor, we need to talk about finiteness conditions.

\begin{defi}
If $R$ is an $A_\infty$-ring spectrum, a (left) $R$-module $M$ is called \emph{finite} or \emph{compact} if it is a compact object in $\Ho(R\modules)$. Equivalently, it lies in the thick subcategory generated by $R$ (see e.g.\ \cite[Theorem 2.13]{HPS97}).
\end{defi}

\begin{lemma}\label{lem:HomSheaves}
With notation as above, there is a natural lift
\[
\xymatrix{ &&\Sp \ar[d]^{\Omega^\infty}\\
\CC^{op}\ar@{-->}[urr]^{\mathcal{H}om_{\OO}(\FF,\GG)} \ar[rr]_{\mathcal{H}om^{\mathcal{S}}_{\OO}(\FF,\GG)}&& \mathcal{S}
}
\]
that is a sheaf of $\OO$-modules. Moreover, there is an equivalence
$$(\mathcal{H}om_{\OO}(\FF,\GG))(U)\simeq\Hom_{\OO|_{\CC/U}}(\FF|_{\CC/U}, \GG|_{\CC/U})$$
that is natural in $\FF$ and $\GG$. If $\CC$ is the \'etale site of a Deligne--Mumford stack so that $(\CC, \OO)$ defines a derived stack, and $\FF$ and $\GG$ are quasi-coherent, then $\mathcal{H}om_{\OO}(\FF,\GG)$ is quasi-coherent as well if $\FF(U)$ is a compact $\OO(U)$-module for every $U\in \CC$ with $U$ affine.
\end{lemma}
\begin{proof}
By construction, $\mathfrak{M}$ and $\mathfrak{M}_{\FF,\GG}$ take values in the $\infty$-category $\widehat{\Cat}_\infty^{st}$ of \emph{stable} $\infty$-categories (under $\Delta^0\sqcup \Delta^0$) where functors are required to preserve finite limits.

We have a natural equivalence
\[\mathcal{H}om_{\OO}(\FF,\GG) \to \mathcal{H}om_{\OO}(\FF,\Omega\Sigma\GG)\simeq \Omega \mathcal{H}om_{\OO}(\FF,\Sigma\GG).\]
Thus, the sequence $\mathcal{H}om_{\OO}(-,\Sigma^n-)$ defines a functor
$$\mathcal{H}om_{\OO}(\FF,\GG)\colon \CC^{op} \to \lim (\cdots \xrightarrow{\Omega} \mathcal{S}_* \xrightarrow{\Omega} \mathcal{S}_*) \simeq \Sp$$
that lifts $\mathcal{H}om^{\mathcal{S}}_{\OO}(\FF,\GG)$. We see that $\mathcal{H}om_{\OO}(\FF,\GG)$ is a presheaf of $\OO$-modules.

A diagram in $\Sp$ is a limit diagram if and only if all its postcompositions with $\Omega^\infty \Sigma^n$ are limit diagrams. As the $\mathcal{H}om^{\mathcal{S}}_{\OO}(\FF, \Sigma^n\GG)$ are sheaves of spaces, $\mathcal{H}om_{\OO}(\FF,\GG)$ is thus a sheaf of spectra.

Now assume that $\CC$ is the \'etale site of a Deligne--Mumford stack $\XX$ so that $(\CC, \OO)$ defines a derived stack; assume furthermore that $\FF$ and $\GG$ are quasi-coherent. By \cite[Proposition 2.2.4.3]{SAG}, the sheaf $\mathcal{H}om_{\OO}(\FF,\GG)$ is quasi-coherent if for every $V\to U$ in $\CC$ with $U,V$ affine, the map
$$ \Hom_{\OO|_U}(\FF|_U,\GG|_U) \sm_A B \to \Hom_{\OO|_V}(\FF|_V,\GG|_V)$$
is an equivalence, where $A = \OO(U)$ and $B = \OO(V)$. We can identify the source with $\Hom_A(\FF(U), \GG(U)) \sm_A B$ as $U$ is affine. As $\FF(V) \simeq \FF(U)\sm_A B$ and $\GG(V) \simeq \GG(U)\sm_A B$, we can identify the target with $\Hom_B(\FF(V), \GG(V)) \simeq \Hom_A(\FF(U), \GG(U)\sm_A B)$. The map is an equivalence as $B$ is finite as an $A$-module.
\end{proof}

After this preparation we aim to show that the spectra $\Tmf(\Gamma)$ are compact $\Tmf[\tfrac1n]$-modules. The crucial ingredient is the following statement.

\begin{prop}\label{prop:0affine}
The stacks $\MMb(\Gamma)^{top}$ are $0$-affine in the sense that the global sections functor
$$\Gamma\colon \QCoh(\MMb(\Gamma)) \to \Tmf(\Gamma)\modules$$
is a (symmetric monoidal) equivalence. Moreover, the descent spectral sequence \eqref{eq:DSS} has a horizontal vanishing lines at $E_r$ for some $r\geq 2$.
\end{prop}
\begin{proof}
By construction, $\MMb(\Gamma) \to \MMb_{ell}$ is an affine map. As $\MMb_{ell} \to \MM_{FG}$ is quasi-affine (see e.g.\ the proof of Theorem 7.2 in \cite{MM15}), the main theorem of \cite{MM15} implies the first statement. The eventual horizontal vanishing line follows from \cite[Theorem 3.14]{Mat13Thick}.
\end{proof}

\begin{prop}\label{prop:compactness}
For an arbitrary congruence subgroup $\Gamma$ of level $n$, the $\Tmf[\tfrac1n]$-module $\Tmf(\Gamma)$ is compact. Likewise, $\TMF(\Gamma)$ is a compact $\TMF[\tfrac1n]$-module.
\end{prop}
\begin{proof}
By one of the main results from \cite{MM15}, the global sections functor defines an equivalence
$$\QCoh(\MMb_{ell,\Z[\tfrac1n]},\OO^{top}) \to \Tmf[\tfrac1n]\modules$$
of $\infty$-categories.

Thus, $\Tmf(\Gamma)$ is a compact $\Tmf[\tfrac1n]$-module if and only if $g_*\OO^{top}_{\MMb(\Gamma)}$ is a compact object in the $\infty$-category of quasi-coherent $\OO^{top}$-modules on $\MMb_{ell,\Z[\tfrac1n]}$. By the last lemma, we have an isomorphism $\pi_*g_*\OO^{top}_{\MMb(\Gamma)}\cong g_*g^*\pi_*\OO^{top}$. As $g$ is finite and flat (see e.g.\ \cite[Proposition 2.4]{MeiDecMod}), this is an \'etale locally free $\pi_*\OO^{top}$-module of finite rank. Thus, $g_*\OO^{top}_{\MMb(\Gamma)}$ is an \'etale locally free $\OO^{top}$-module of finite rank as well.

If a quasi-coherent $\OO^{top}$-module $\FF$ is \'etale locally compact, then $\FF$ is a compact object in $\QCoh(\MMb_{ell,\Z[\tfrac1n]}, \OO^{top})$. Indeed, the map
$$\bigoplus_{i\in I}\mathcal{H}om_{\OO^{top}}(\FF,  \GG_i) \to \mathcal{H}om_{\OO^{top}}(\FF, \bigoplus_{i\in I} \GG_i)$$
is an equivalence if it is an equivalence locally, where $\mathcal{H}om$ denotes the Hom-sheaves as in the last lemma. Then we use that taking global sections on quasi-coherent $\OO^{top}$-modules commutes with infinite direct sums to deduce that
$$\bigoplus_{i\in I}\Hom_{\OO^{top}}(\FF,  \GG_i) \to \Hom_{\OO^{top}}(\FF, \bigoplus_{i\in I} \GG_i) $$
is an equivalence.

The proof for $\TMF(\Gamma)$ is the same.
\end{proof}

In particular, $\pi_*\Tmf(\Gamma)$ is finitely generated over $\pi_*\Tmf[\tfrac1n]$ as $\Tmf(\Gamma)$ is in the thick subcategory of $\Ho(\Tmf[\tfrac1n]\modules)$ generated by the unit. This can be used to give a topological proof for the following (known) fact.

\begin{cor}[\cite{MeiDecMod}]
  For $\Gamma \subset SL_2(\Z)$ a congruence subgroup of level $n$, the ring of modular forms $M_*(\Gamma;\Z[\tfrac1n]) = H^0(\MMb(\Gamma);\omega^{\tensor *})$ for $\Gamma$ over $\Z[\tfrac1n]$ is finitely generated as a module over the ring of modular forms for $SL_2(\Z)$ over $\Z[\tfrac1n]$.
\end{cor}
\begin{proof}
 We can assume that $n\geq 2$ and also that $\Gamma = \Gamma_1(n)$ or $\Gamma(n)$ as $M_*(\Gamma;\Z[\tfrac1n]) \subset M_*(\Gamma_1(n);\Z[\tfrac1n])$ for $\Gamma_1(n)\subset \Gamma\subset \Gamma_0(n)$ by Lemma \ref{lem:compactificationsGamma} and \[M_*(SL_2(\Z);\Z[\tfrac1n]) \cong \Z[\tfrac1n][c_4,c_6,\Delta]/(1728\Delta = c_4^3-c_6^2)\]
  is noetherian \cite{Del75}. In particular, we can assume that $\Gamma$ is tame. 

 The inclusion $\pi_{odd}\Tmf(\Gamma) \subset\pi_*\Tmf(\Gamma)$ is one of (ungraded) $\pi_*\Tmf[\tfrac1n]$-modules as the elements in odd degree are in positive filtration in the descent spectral sequence and the multiplication respects that. Thus, $\pi_{even}\Tmf(\Gamma)$ is a finitely generated $\pi_*\Tmf[\tfrac1n]$-module if we view it as a quotient of $\pi_*\Tmf(\Gamma)$. The ring $\pi_{even}\Tmf(\Gamma)$ agrees with the ring of modular forms for $\Gamma$ over $\Z[\tfrac1n]$ and is torsionfree. Thus, it is also finitely generated over $\pi_{even}\Tmf[\tfrac1n]/\mathrm{torsion}$, which is a subring of $H^0(\MMb_{ell};\omega^{\tensor *})[\tfrac1n] = M_*(SL_2(\Z);\Z[\tfrac1n])$.
\end{proof}

\section{Splittings for topological modular forms with level structure}\label{sec:splittings}
The goal of this section will be to demonstrate the main splitting results from the introduction for $\TMF_1(n)$ and $\Tmf_1(n)$ and related spectra. This relies on algebraic splitting results that we will recall and develop in the next subsection.

\subsection{Algebraic splitting results}
In \cite{MeiDecMod}, we have proven splitting results for vector bundles on the moduli stack of elliptic curves, of which we want to summarize the relevant parts in this section.

Fix a prime $l$, let $\Gamma$ be a \emph{tame} congruence subgroup in the sense of Definition \ref{def:tame} and denote by
\[g\colon \MMb(\Gamma)_{(l)} \to \MMb_{ell, (l)}\]
 the projection. Consider the vector bundle
\[
\EE_l = \begin{cases}(f_3)_*\OO_{\MMb_1(3)_{(2)}} & \text{ if }l = 2\\
(f_2)_*\OO_{\MMb_1(2)_{(3)}} & \text{ if }l = 3\\
\OO_{\MMb_{ell,(l)}} & \text { if } l>3,
\end{cases}
\]
where $f_n\colon \MMb_1(n) \to \MMb_{ell}$ denotes the projection.

We recall the following two splitting results, which have been proven in \cite{MeiDecMod} in the cases $\Gamma = \Gamma_0(n), \Gamma_1(n)$ or $\Gamma(n)$. The proof in the cases $\Gamma_1(n) \subsetneq \Gamma \subsetneq \Gamma_0(n)$ is entirely analogous, using Proposition \ref{prop:coh}.

\begin{prop}[\cite{MeiDecMod}, Proposition 3.5]\label{prop:split}
Let $\FF$ be a vector bundle on $\MMb(\Gamma)_{\F_l}$. Then $g_*\FF$ decomposes into a direct sum of vector bundles of the form $(\EE_l)_{\F_l} \tensor \omega^{\tensor i}$.
\end{prop}

\begin{thm}[\cite{MeiDecMod}, Theorem 3.12]\label{thm:deccompact}
The vector bundle $g_*\OO_{\MMb(\Gamma)_{(l)}}$ decomposes into copies of $\EE_l\tensor \omega^{\tensor ?}$ if and only if $H^1(\MMb_{ell};\omega)$ has no $l$-torsion.
\end{thm}
As explained in the last section, this happens if and only if $\pi_1\Tmf(\Gamma)$ has no $l$-torsion. As noted in \cite[Remark 3.17]{MeiDecMod} another equivalent characterization is that every modular form for $\Gamma$ over $\F_l$ lifts to a modular form for $\Gamma$ over $\Z_{(l)}$. For $l=2$ and $\Gamma = \Gamma_1(n)$, the smallest counterexample is $n=65$.

There is also an uncompactified version of the last theorem. We will work on the $l$-completion (see e.g.\ \cite{Con} for a reference) and denote by abuse of notation the projection $\widehat{\MM(\Gamma)}_l \to \widehat{\MM}_{ell, l}$ by $g$ as well. Likewise, we will denote the restrictions of the previous $\EE_l$ to $\widehat{\MM}_{ell,l}$ also by $\EE_l$.

\begin{thm}\label{thm:decuncompact}
The vector bundle $g_*\OO_{\widehat{\MM(\Gamma)}_l}$ decomposes into copies of $\EE_l\tensor \omega^{?}$. More generally the same is true for $\widehat{g_*\FF}_l$ for any vector bundle $\FF$ on $\MM(\Gamma)$.
\end{thm}

We will deduce the theorem from the corresponding theorem over $\F_l$ by the following result from obstruction theory, which is analogous to \cite[Cor 8.5.5]{FAG}.

\begin{prop}
 Let $A$ be a noetherian ring with a maximal ideal $I$. Let $\XX$ be a locally noetherian Artin stack over $A$. We denote by $\hat{\XX}$ the completion with respect to $I\OO_\XX$ and by $\XX_0$ the closed substack corresponding to $I\OO_{\XX}$. Let $\FF,\GG$ be vector bundles on $\XX$ such that $H^1(\XX_0; \mathcal{H}om_{\OO_{\XX_0}}(\FF|_{\XX_0}, \GG|_{\XX_0})) = 0$. Then the completions $\widehat{\FF}$ and $\widehat{\GG}$ are isomorphic on $\widehat{\XX}$ if and only if $\FF|_{\XX_0}\cong \GG|_{\XX_0}$.
\end{prop}
\begin{proof}The only-if direction is clear. We will prove the if direction.

 Set $\XX_n = \XX\times_{\Spec A} \Spec A/I^{n+1}$ and $\FF_n =\FF|_{\XX_n}$ and $\GG_n = \GG|_{\XX_n}$. Choose an isomorphism $f_0\colon \FF_0 \to \GG_0$. We want to show inductively that it lifts to a morphism $f_n\colon \FF_n \to \GG_n$ (whose pullback to $\XX_0$ agrees with $f_0$). Assume that a lift $f_{n-1}$ of $f_0$ to $\XX_{n-1}$ is already constructed. The ideal defining $\XX_{n-1}\subset \XX_n$ is $$I^n\OO_\XX/I^{n+1}\OO_\XX \cong \OO_\XX \tensor_A I^n/I^{n+1}.$$
 By \cite[Thm 8.5.3]{FAG}, the obstruction for lifting $f_{n-1}$ to $f_n\colon \FF_n \to \GG_n$ lies in
 \begin{align*}
  H = H^1(\XX_{n-1}; I^n\OO_\XX/I^{n+1}\OO_\XX\tensor_{\OO_{\XX_{n-1}}} \mathcal{H}om_{\OO_{\XX_{n-1}}}(\FF_{n-1},\GG_{n-1})).
 \end{align*}
Because the coefficient sheaf is killed by $I$, it is the pushforward of its pullback along the closed immersion $\XX_0\to \XX_n$. Thus,
\begin{align*}
H &\cong
 H^1(\XX_0; I^n\OO_\XX/I^{n+1}\OO_\XX\tensor_{\OO_{\XX_0}} \mathcal{H}om_{\OO_{\XX_0}}(\FF_0,\GG_0)) \\
 &\cong H^1(\XX_0; I^n/I^{n+1}\tensor_{A/I} \mathcal{H}om_{\OO_{\XX_0}}(\FF_0,\GG_0)) \\
 &\cong I^n/I^{n+1} \tensor_{A/I} H^1(\XX_0; \mathcal{H}om_{\OO_{\XX_0}}(\FF_0,\GG_0))  \\
 &\cong 0.
\end{align*}
Thus, we get a compatible system of morphisms $(f_n)_{n\in\N_0}$. All of these are isomorphisms by the Nakayama lemma. Thus, $(f_n)_{n\in\N_0}$ defines an isomorphism of adic systems and thus one of completions by \cite{Con}.
\end{proof}

\begin{proof}[Proof of Theorem \ref{thm:decuncompact}]
	Let $\FF$ be a vector bundle on $\MM(\Gamma)$. 
We will show first that $\FF$ is the restriction of a vector bundle on $\MMb(\Gamma)$. Indeed, by Lemma 3.2 from \cite{Mei13}, there is a reflexive sheaf $\FF'$ on $\MMb(\Gamma)$ restricting to $\FF$. Choose a surjective \'etale map $U \to \MMb(\Gamma)$ from a scheme $U$. By Proposition 2.4 from \cite{MeiDecMod}, the scheme $U$ is smooth of relative dimension $1$ over $\Z_{(l)}$ and thus a regular scheme of dimension $2$. By \cite[Cor 1.4]{Har80}, the pullback of $\FF'$ to $U$ is a vector bundle and thus $\FF'$ itself as well.

 By Proposition \ref{prop:split}, it follows that $g_*\FF$ and a vector bundle $\GG$ of the form
 $$\bigoplus_{i\in I} \EE_l \tensor \omega^{\otimes n_i}$$
 are isomorphic after base change to $\F_l$. Set $\FF_0 = \FF|_{\MM(\Gamma)_{\F_l}}$ and $\GG_0 = \GG|_{\MM_{ell, \F_l}}$. By the last proposition, it remains to show that
 $$H^1(\MM_{ell, \F_l}; \mathcal{H}om(\GG_0, g_*\FF_0)) = 0.$$
 We have
 \begin{align*}
  H^1(\MM_{ell, \F_l}; \mathcal{H}om(\GG_0, g_*\FF_0)) &\cong H^1(\MM_{ell,\F_l}; g_* \mathcal{H}om(g^*\GG_0, \FF_0)) \\
  &\cong H^1(\MM(\Gamma)_{\F_l}; \mathcal{H}om(g^*\GG_0, \FF_0))= 0
 \end{align*}
 as the cohomology of every quasi-coherent sheaf on $\MM(\Gamma)_{\F_l}$ vanishes by Proposition \ref{prop:coh}.
\end{proof}

\subsection{The topological splitting results}
 The algebraic decomposition theorems from the last subsection rather easily imply topological counterparts. Throughout this section, we will fix a prime $l$ and choose a tame congruence subgroup $\Gamma\subset SL_2(\Z)$. We denote again by $g\colon \MMb(\Gamma)_{(l)} \to \MMb_{ell,(l)}$ the projection.

\begin{thm}\label{thm:dectopcompact}
The $\Tmf_{(l)}$-module $\Tmf(\Gamma)_{(l)}$ splits into a sum of evenly shifted copies of
\begin{itemize}
 \item $\Tmf_1(3)_{(2)}$ if $l=2$,
 \item $\Tmf_1(2)_{(3)}$ if $l=3$, and
 \item $\Tmf_{(l)}$ if $l>3$
\end{itemize}
if and only if $\pi_1\Tmf(\Gamma)$ is $l$-torsionfree.
\end{thm}
\begin{proof}
 The ``only if'' part is clear as $\pi_*\Tmf_1(3)$, $\pi_*\Tmf_1(2)$ and $\pi_*\Tmf_{(l)}$ for $l\geq 5$ are torsionfree.

 From now on we will implicitly localize at $l$ everywhere. We will only consider the case $l=3$ as the cases $l=2$ and $l>3$ are very similar. Consider the quasi-coherent $\OO^{top}$-module $\EE = g_*\OO^{top}_{\MMb(\Gamma)} =g_*g^*\OO^{top}$, whose global sections are $\Tmf(\Gamma)$. We know that $\OO^{top}$ is even and $\pi_{2k}\OO^{top} \cong \omega^{\tensor k}$. By Lemma \ref{lem:pushpull}, the sheaf $\EE$ is even as well and
 $$\pi_{2k}\EE \cong g_*g^*\omega^{\tensor k} \cong (g_*\OO_{\MMb(\Gamma)}) \tensor \omega^{\tensor k}.$$
 As every quasi-coherent sheaf on $\MMb(\Gamma)$ has cohomological dimension $\leq 1$ by Proposition \ref{prop:coh}, the same is true for pushforwards of such sheaves to $\MMb_{ell}$ by the Leray spectral sequence as $g$ is affine. In particular, the descent spectral sequence is concentrated in lines $0$ and $1$ and this implies that
 $$\pi_1\Tmf_1(n) \cong H^1(\MMb_{ell}, \pi_2\EE)$$
 and this is $l$-torsionfree by assumption.

 Thus, by Theorem \ref{thm:deccompact} we obtain a splitting of the form
 $$g_*\OO_{\MMb(\Gamma)} \cong \bigoplus_{i\in\Z} \bigoplus_{k_i} (f_2)_*\OO_{\MMb_1(2)}\tensor \omega^{\tensor -i}.$$
 Set
 $$\FF = \bigoplus_{i\in\Z} \bigoplus_{k_i} (f_2)_*(f_2)^*\Sigma^{2i}\OO^{top}$$
 and note that $\Gamma(\FF) = \bigoplus_{i\in\Z} \bigoplus_{k_i} \Sigma^{2i}\Tmf_1(2)$. As we have just seen, there is an isomorphism $h\colon \pi_0\FF \xrightarrow{\cong} \pi_0\EE$. Note that $\FF$ is even as well and $\pi_{2k}\FF \cong \pi_0\FF \tensor \omega^{\tensor k}$ so that we actually have $\pi_*\EE \cong \pi_*\FF$.

 There is a descent spectral sequence
 \[H^q(\MMb_{ell}; \mathcal{H}om_{\pi_*\OO^{top}}(\pi_*\FF,\pi_{*+p}\EE)) \Rightarrow \pi_{p-q}\mathrm{Hom}_{\OO^{top}}(\FF,\EE)\]
 for $\mathcal{H}om$ the Hom-sheaf; indeed, $\FF$ is locally free and thus
$$\pi_p\mathcal{H}om_{\OO^{top}}(\FF,\EE) \cong \mathcal{H}om_{\pi_*\OO^{top}}(\pi_*\FF,\pi_{*+p}\EE).$$
 We claim that $\mathcal{H}om_{\pi_*\OO^{top}}(\pi_*\FF,\pi_{*+p}\EE) \cong \mathcal{H}om_{\OO_{\MMb_{ell}}}(\pi_0\FF,\pi_p\EE)$ has cohomological dimension $\leq 1$. Indeed, it is the sum of vector bundles of the form
 $$\mathcal{H}om_{\OO_{\MMb_{ell}}}((f_2)_*\OO_{\MMb_1(2)}\tensor \omega^{\tensor j}, g_*\OO_{\MMb(\Gamma)}),$$
each of which is isomorphic to
$$g_*\mathcal{H}om_{\OO_{\MMb(\Gamma)}}(g^*((f_2)_*\OO_{\MMb_1(2)}\tensor \omega^{\tensor j}), \OO_{\MMb(\Gamma)}).$$
As $g$ is affine, there can be thus no differentials or extension issues in the spectral sequence.
 The homomorphism $h$ defines an element in the spot $p=q=0$. As there are no differentials, $h$ lifts uniquely (up to homotopy) to a morphism $\FF \to \EE$ that induces an isomorphism on $\pi_*$ and is thus an equivalence.
\end{proof}

As we proved our algebraic splitting results in the uncompactified case only in a completed setting, we will need to complete our spectra as well to prove a splitting result for periodic $\TMF(\Gamma)$. We will denote for a spectrum $E$ by $\widehat{E}_l = \holim_n E/l^n$ its $l$-completion (see e.g.\ \cite[Section II.7.3]{SAG} for a conceptual treatment). Note $\widehat{E}_l/l \simeq E/l$.

\begin{thm}\label{thm:dectopuncompact}
The $\widehat{\TMF}_l$-module $\widehat{\TMF(\Gamma)}_l$ splits into a sum of evenly shifted copies of the $l$-completions of
\begin{itemize}
 \item $\TMF_1(3)$ if $l=2$,
 \item $\TMF_1(2)$ if $l=3$, and
 \item $\TMF$ if $l>3$.
\end{itemize}
\end{thm}
\begin{proof}
Note first that $l$-completion commutes with homotopy limits. In particular, if we define $\widehat{\OO}^{top}_l$ by $\widehat{\OO}^{top}_l(U) = (\OO^{top}(U))^{\wedge}_l$ for every Deligne--Mumford stack $U$ \'etale over $\MM_{ell}$, we see that $\widehat{\OO}^{top}_l$ is still a sheaf on the \'etale site of $\MM_{ell}$ with global sections $\widehat{\TMF}_l$. By \cite[Cor 7.3.5.2]{SAG}, $l$-completion is monoidal and thus $\widehat{\OO}^{top}_l$ is at least a sheaf of $A_\infty$-ring spectra. As $\pi_*\OO^{top}$ is concentrated in even degrees and $l$-torsionfree, $\pi_*\widehat{\OO}^{top}_l$ agrees with the $l$-completion of $\pi_*\OO^{top}$.

We claim that for every coherent sheaf $\FF$ on the completion $\widehat{\MM(\Gamma)}_l$ the cohomology groups $H^i(\widehat{\MM(\Gamma)}_l; \FF)$
 vanish for $i>0$. Indeed, by \cite[Theorem 2.3]{Con}, the category of coherent sheaves on $\widehat{\MM(\Gamma)}_l$ is equivalent to that of adic systems $(\FF_i)_{i\geq 0}$.
  The global section functor corresponds to $\lim_i H^0(\MM(\Gamma); \FF_i)$. As by Proposition \ref{prop:coh} the global sections functor on $\MM(\Gamma)$ is exact,
   $H^0(\MM(\Gamma); \FF_i)$ forms a Mittag-Leffler system and we see that the global sections functor on coherent sheaves on $\widehat{\MM(\Gamma)}_l$ is exact as well.

Working everywhere with $l$-completions, the proof is from this point on analogous to the one in the compactified case if we use Theorem \ref{thm:decuncompact}. Thus, we obtain the theorem.
\end{proof}

If $\pi_1\Tmf_1(n)$ is $l$-torsionfree, the completions in the previous example are not necessary as we can just invert $\Delta$ in Theorem \ref{thm:dectopcompact}.

\begin{example}
The $\Tmf$-module $\Tmf_1(5)$ decomposes after $2$-localization as
$$\Tmf_1(3) \oplus \Sigma^2\Tmf_1(3) \oplus \Sigma^4\Tmf_1(3)$$
and after $3$-localization as
$$\Tmf_1(2) \oplus \Sigma^2 \Tmf_1(2)^{\oplus 2} \oplus \Sigma^4 \Tmf_1(2)^{\oplus 2}\oplus \Sigma^6 \Tmf_1(2)^{\oplus 2} \oplus \Sigma^8 \Tmf_1(2).$$
As $\TMF_1(2)$ is $8$-periodic, this decomposition simplifies after inverting $\Delta$ and still localizing at $3$ to
$$\TMF_1(5) \simeq \TMF_1(2)^{\oplus 2} \oplus \Sigma^2 \TMF_1(2)^{\oplus 2} \oplus \Sigma^4 \TMF_1(2)^{\oplus 2}\oplus \Sigma^6 \TMF_1(2)^{\oplus 2}.$$
In general, the formulae for the decompositions for the $\Tmf$-modules are the same as the corresponding ones for the decompositions of the vector bundles. See Appendix C of \cite{MeiDecMod} for tables of decompositions and Section 4 of op.\ cit.\ for the general formulae.
\end{example}

\begin{remark}
In the previous example we saw that in the $2$-local and $3$-local decompositions for $\TMF_1(5)$ all even shifts of $\TMF_1(3)$ and $\TMF_1(2)$, respectively, occur equally often up to the respective periodicities. Moreover, \cite[Theorem 4.18]{MeiDecMod} and the paragraph thereunder can be used to show that $\TMF_1(n)$ will $l$-locally always decompose into (unshifted) copies of $\TMF_1(5)$, at least if $\pi_1\Tmf_1(n)$ is $l$-torsionfree. Thus, all even shifts of $\TMF_1(3)$ and $\TMF_1(2)$ up to their respective periodicities occur equally often in the $2$- and $3$-local decompositions of $\TMF_1(n)$ if $\pi_1\Tmf_1(n)$ is $l$-torsionfree for $l=2$ and $l=3$, respectively.
\end{remark}

\section{Equivariant $\TMF$}\label{sec:equivariant}
The main goal of this section is to apply our previous splitting results to equivariant $\TMF$ as introduced in \cite{Lur07} and \cite{LurEllIII} (based on earlier ideas by Grojnowski and many others). The required background material on spectral algebraic geometry and oriented elliptic curves can be found in \cite{SAG} and \cite{LurEllII}. In the first subsection we will formulate certain complements to \cite{LurEllII} to obtain the \emph{universal oriented elliptic curve}, whose basic properties we study. This is the basis for the treatment of equivariant $\TMF$ in the second subsection. We remark that only the interpretation of our results requires the machinery of equivariant elliptic cohomology, but our results can be formulated without it as well.

\subsection{The universal oriented elliptic curve}
Recall from Definition \ref{defi:der} our use of the term \emph{derived stack}. We want to recall the notions of an (oriented) elliptic curve from \cite{LurEllII}, but possibly with a derived stack instead of just an $E_\infty$-ring as a base. We will use the shorthand $\tau_{\geq 0}X$ for the spectral Deligne--Mumford stack $(\XX, \tau_{\geq 0}\OO_\XX)$ if $X = (\XX,\OO_{\XX})$ is a derived stack.

\begin{defi}
Let $X$ be a derived stack. A \emph{strict elliptic curve} $E$ over $X$ is an abelian group object\footnote{See \cite[Section 1.2]{LurEllI} for a definition.} in the $\infty$-cat $\DerSt/X$ of derived stacks over $X$ such that
\begin{enumerate}
\item $E \to X$ is flat in the sense of \cite[Definition 2.8.2.1]{SAG},
\item $\tau_{\geq 0}E \to \tau_{\geq 0}X$ is proper and locally almost of finite presentation in the sense of \cite[Definitions 5.1.2.1 and 4.2.0.1]{SAG},
\item For every morphism $i\colon \Spet k \to \tau_{\geq 0}X$ with $k$ an algebraically closed (classical) field the pullback $i^*E \to \Spet k$ is a classical elliptic curve. (See Section 1.4.2 in \cite{SAG} for the definition of $\Spet$.)
\end{enumerate}
We denote by $\Ell^s(X)$ the full sub-$\infty$-category of the $\infty$-category of abelian group objects in $\mathrm{DerSt}/X$ on the strict elliptic curves.
\end{defi}
For $X = \Spet R$ for an $E_\infty$-ring $R$, this recovers the notion of a strict elliptic curve from \cite[Definition 2.0.2]{LurEllI}.

\begin{lemma}\label{lem:strictdescent}
The functor
\[\DerSt^{op} \to \mathcal{S},\qquad X \mapsto \Ell^s(X)^{\simeq}\]
to the $\infty$-category of spaces is a sheaf for the \'etale topology, where $\Ell^s(X)^{\simeq}\subset \Ell^s(X)$ is the maximal sub-$\infty$-groupoid.
\end{lemma}
\begin{proof}
It suffices to show that being a strict elliptic curve is local on the target for the \'etale topology. This follows from \cite{SAG}, Proposition 4.2.1.5 and Corollary 5.3.1.2, together with the following fact: For every \'etale cover $Y \to X$ and every $\Spet k \to \tau_{\geq 0}X$ (with $k$ algebraically closed), there is a lift to $\tau_{\geq 0} Y$ as the map $\tau_{\geq 0}Y \to \tau_{\geq 0} X$ is still an \'etale cover and every \'etale map to $\Spet k$ splits.
\end{proof}

For the following definition let us recall the following notation. Let $E$ be a derived stack over a derived stack $X$ and let $Y$ be another derived stack over $X$. Then we denote by $E(Y)$ the mapping space from $Y$ to $E$ in $\DerSt_{/X}$.

\begin{defi}
Let $X$ be a derived stack and $E$ a strict elliptic curve over $X$. A \emph{preorientation} on $E$ is a pointed map $S^2 \to E(\tau_{\geq 0} X)$ and the space of preorientations is $\Omega^2 E(\tau_{\geq 0} X)$. A preorientation on $E$ is an \emph{orientation} if its pullback along every map $\Spet A \to X$ for $A$ an $E_\infty$-ring is an orientation in the sense of \cite[Definition 7.2.7]{LurEllII}.
\end{defi}
Thus, we obtain for each $X$ a functor $\Ell^s(X)^{op} \to \mathcal{S}$, mapping each strict elliptic curve $E$ to its space $\Omega^2 E(\tau_{\geq 0} X)$ of preorientations. By the unstraightening construction, this functor is classified by a left fibration $\Ell^{pre}(X) \to \Ell^s(X)$ (see \cite[Section 2.2.1]{HTT} or \cite{H-MLeftII} for a different exposition). We denote by $\Ell^{or}(X)\subset \Ell^{pre}(X)$ the full sub-$\infty$-category of oriented elliptic curves. The functoriality \cite[Proposition 2.2.1.1]{HTT} of the unstraightening construction implies that we obtain functors
$$\Ell^{pre}, \Ell^{or}\colon \DerSt \to \Cat_{\infty}$$

\begin{lemma}
The functors
\[\DerSt^{op} \to \mathcal{S},\qquad X \mapsto \Ell^{pre}(X)^{\simeq}\]
and
\[\DerSt^{op} \to \mathcal{S},\qquad X \mapsto \Ell^{or}(X)^{\simeq}\]
are sheaves for the \'etale topology.
\end{lemma}
\begin{proof}
Let $U \to X$ be an \'etale cover of derived stacks and let $U_\bullet$ (with $U_0 = U$, $U_1 = U\times_X U$ etc.) be the associated \v{C}ech simplicial object.

We have to show that in the diagram
\[\xymatrix{
\Ell^{pre}(X)^{\simeq} \ar[d]^{p}\ar[r]^-f & \lim_{\Delta} \Ell^{pre}(U_{\bullet})^{\simeq}\ar[d]^q\\
\Ell^s(X)^{\simeq} \ar[r]^-g & \lim_{\Delta} \Ell^s(U_{\bullet})^{\simeq}
}\]
the map $f$ is an equivalence.
Note that $\Ell^{pre}(X)^{\simeq} \to \Ell^s(X)^{\simeq}$ is a left fibration of Kan sets and thus a Kan fibration \cite[Lemma 2.1.3.3]{HTT} and likewise the maps $\Ell^{pre}(U_{\bullet})^{\simeq} \to \Ell^s(U_{\bullet})^{\simeq}$ are Kan fibrations. By picking a suitable model for the homotopy limit, we can assume that $q$ is a Kan fibration as well.

We have the following properties:
\begin{enumerate}
\item $g$ is an equivalence by Lemma \ref{lem:strictdescent},
\item for every $E \in \Ell^s(X)$ the map
$$p^{-1}(E) \simeq \Omega^2E(\tau_{\geq 0}X)\to \lim_{\Delta}\Omega^2E(\tau_{\geq 0}U_\bullet) \simeq q^{-1}(g(E))$$
 is an equivalence. Here, we use $E(\tau_{\geq 0}U_\bullet) \simeq E_{U_\bullet}(\tau_{\geq 0}U_\bullet)$, where $E_{U_n}$ is the base change of $E$ to $U_n$.
\end{enumerate}
This implies that $f$ is an equivalence as well.

To show that $X \mapsto \Ell^{or}(X)^{\simeq}$ is an \'etale sheaf, it suffices to show that every preorientation that is \'etale-locally an orientation is also an orientation. This follows directly from the case of $X$ being affine \cite[Proposition 7.2.10]{LurEllII}.
\end{proof}
This allows us to give a slight extension of one of the main theorems of \cite{LurEllII}.
\begin{thm}[Lurie]\label{thm:Luriemain}
The functors
\[\DerSt^{op} \to \mathcal{S},\qquad X \mapsto \Ell^{pre}(X)^{\simeq}\]
and
\[\DerSt^{op} \to \mathcal{S},\qquad X \mapsto \Ell^{or}(X)^{\simeq}\]
are representable by derived stacks $\MM_{ell}^{pre} = (\MM_{ell}, \OO^{pre})$ and $\MM_{ell}^{top} = (\MM_{ell}, \OO^{top})$ respectively.\footnote{We have chosen the name $\MM_{ell}^{top}$ instead of Lurie's $\MM_{Ell}^{or}$ because the notation $\OO^{top}$ for its structure sheaf is very common in the Goerss--Hopkins--Miller tradition.} More precisely, the map $\MM_{ell} \to \MM_{ell}^{pre}$ classifying the unique preorientation on a classical elliptic curve and the map $\MM^{pre}_{ell} \to \MM^{top}_{ell}$ classifying the preorientation underlying an orientation induce equivalences on the underlying classical stacks.

 Moreover, $\pi_{2k-1}\OO^{top} = 0$ and $\pi_{2k}\OO^{top}\cong \omega^{\tensor k}$ for all $k\in\Z$ and $\OO^{top}$ is equivalent to the sheaf defined by Goerss, Hopkins and Miller. The sheaf $\OO^{pre}$ is connective with $\pi_0\OO^{pre} \cong \OO_{\MM_{ell}}$.
\end{thm}
\begin{proof}
The representability is contained in Theorem 7.0.1 and Propositions 7.2.5 and 7.2.10 of \cite{LurEllII} for the restrictions of the two functors above to derived stacks of the form $\Spet A$ for an $E_\infty$-ring $A$. That the representing objects $\MM_{ell}^{pre}$ and $\MM_{ell}^{top}$ are $1$-localic with the \'etale topos of $\MM_{ell}$ as underlying topos is contained in \cite[Remark 2.4.2]{LurEllI} and \cite[Remark 7.3.2]{LurEllII}. This implies in particular by Lemma \ref{lem:derivedstacksequivalence} that they are derived stacks in our sense. 
As every derived stack is \'etale locally of the form $\Spet A$, we see that the same derived stacks $\MM_{ell}^{pre}$ and $\MM_{ell}^{top}$ represent the two functors in the theorem on all derived stacks as well.

 The agreement with the Goerss--Hopkins--Miller sheaf is stated in \cite[Remark 7.0.2]{LurEllII} and proven in \cite{DaviesElliptic}.
\end{proof}

\begin{defi}
The \emph{universal oriented elliptic curve} over $\MM_{ell}^{top}$ is the oriented elliptic curve classified by the identity on $\MM_{ell}^{top}$. We will denote it by $\Etop$.
\end{defi}

Given a strict elliptic curve $E$ over a derived stack, we can define its $n$-torsion points via the pullback square
\[
\xymatrix{
E[n] \ar[r]\ar[d] & E \ar[d]^{[n]} \\
X \ar[r]^e & E,
}
\]
where $e$ is the unit section and $[n]$ the multiplication-by-$n$ morphism. This pullback is again a derived stack by Lemma \ref{lem:derivedpullback}.

\subsection{A splitting of equivariant $\TMF$}
In \cite{Lur07}, Lurie uses the universal derived elliptic curve to sketch a construction of equivariant elliptic cohomology. Details have appeared in \cite{LurEllIII} and \cite{GM20} and we will recall features of this construction now.

Let $G$ be an abelian compact Lie group. Define $\MM_G$ to be the moduli stack of elliptic curves $E$ together with a morphism $\widehat{G} \to E$ from the constant group scheme of characters $\widehat{G} = \Hom(G,\C^\times)$ of $G$.
Lurie defines $\MM_G^{top}$ as the derived mapping stack $\Hom(\widehat{G}, \Etop)$. As $\widehat{G}$ is a finitely generated abelian group, this can concretely be constructed as follows: We set $\Hom(\Z,\Etop)$ to be $\Etop$ itself and $\Hom(\Z/n, \Etop)$ to be the $n$-torsion points $\Etop[n]$.
Furthermore, $\Hom(H_1\oplus H_2,\Etop)$ is $\Hom(H_1,\Etop)\times_{\MM_{ell}^{top}} \Hom(H_2, \Etop).$  We denote the structure sheaf of $\MM_G^{top}$ by $\OO^{top}_G$. Set $\MM_G = \Hom(\widehat{G}, \Euni)$. 

Lurie defines in \cite{Lur07} $\infty$-functors $$\FF_G\colon (\text{finite }G\text{-CW complexes})^{op} \to \QCoh(\MM_G^{top}),$$
with details appearing in \cite{LurEllIII} and \cite{GM20}. 
These functors satisfy the following properties:
\begin{enumerate}
 \item $\FF_G$ sends finite homotopy colimits of $G$-CW complexes to finite homotopy limits,
 \item $\FF_G(\pt) = \OO^{top}_G$, and
 \item For $H\subset G$ and $X$ a finite $H$-CW complex, we have $\FF_G(X\times_H G) \simeq f_*\FF_H(X)$ for $f\colon \MM_H^{top} \to \MM_G^{top}$ the morphism defined by restriction.
\end{enumerate}

The $G$-equivariant $\TMF$-cohomology of some finite $G$-CW complex $X$ is then computed as the (homotopy groups of the) global sections of $\FF_G(X)$. We set the $G$-fixed points $\TMF^G$ of $G$-equivariant $\TMF$ to be $\Gamma(\OO^{top}_G) = \OO^{top}_G(\MM_G).$ Its homotopy groups are the value of $G$-equivariant $\TMF$ at a point.

We need the following lemma.

\begin{lemma}\label{lem:FiniteFlat}
 For $G$ finite abelian, the morphism $\MM_G \to \MM_{G/H}$ is finite and flat for every split inclusion $H\subset G$.
\end{lemma}
\begin{proof}By writing $H$ as a sum of cyclic groups and using induction, we can assume that $H$ is cyclic of order $k$. As $G \cong H \oplus G/H$, we obtain $\MM_G \simeq \MM_{G/H} \times_{\MM_{ell}} \Euni[k]$. By \cite[Theorem 2.3.1]{K-M85}, the stack $\EE[k]$ is finite and flat over $\MM_{ell}$ and the result follows.
\end{proof}

\begin{prop}\label{prop:0affineG}
For $G$ finite abelian, the global sections functor
\[\Gamma\colon \QCoh(\MM_G^{top}) \to \TMF^G\modules\]
is an equivalence of $\infty$-categories.
\end{prop}
\begin{proof}
 By construction, there is a morphism $\MM_G^{top} \to \MM_{ell}^{top}$ of derived stacks. By one of the main results of \cite{MM15}, the derived stack $\MM_{ell}^{top}$ is $0$-affine in the sense that the global sections functor
 $$\Gamma\colon \QCoh(\MM_{ell}^{top}) \to \TMF\modules$$
 is an equivalence of $\infty$-categories. By the last lemma the underlying morphism of stacks $\MM_G \to \MM_{ell}$ is finite and flat. It follows by \cite[Prop 3.29]{MM15} that $\MM_G^{top}$ is $0$-affine as well.
\end{proof}

In particular, we do not lose information when we use the global sections functor. This gives special importance to the spectra $\TMF^G$. We will determine $\TMF^G$ in terms of better known spectra after completing at a prime not dividing $|G|$, or at least inverting $|G|$. To that purpose, let $\MM^G$ be the moduli stack of elliptic curves with $G$-level structure in the sense of \cite[Section 1.5]{K-M85}, i.e.\ for an elliptic curve $E/S$ with $|G|$ invertible in $S$ an $S$-homomorphism $G\times S \to E$ that defines an injection $G \to E(k)$ for every geometric point $\Spec k \to S$. 

\begin{lemma}
For $G$ finite abelian, there is a splitting
$$\Hom(G,\Euni)_{\Z[\tfrac1{|G|}]} \simeq \coprod_{K\subset G} \MM^{G/K}_{\Z[\tfrac1{|G|}]},$$
where $K$ runs over all subgroups of $G$ such that $G/K$ embeds into $(\Z/|G|)^2$.
 \end{lemma}
\begin{proof}Let $S$ be a connected scheme with $|G|$ invertible. We have to show that there is a natural decomposition $\MM_G(S) \simeq \coprod_{K\subset G} \MM^{G/K}(S)$ of groupoids (where $K$ is as above). Let $E/S$ be an elliptic curve and $f\colon G\times S \to E$ be a homomorphism over $S$. This factors over $E[|G|]$, which is a finite \'etale group scheme over $S$. Choose a geometric point $s$ of $S$. Then the category of finite \'etale group schemes over $S$ is equivalent to the category of finite groups with continuous $\pi_1(S,s)$-action of the \'etale fundamental group (\cite[Exp 5]{SGA1} or \cite[Thm 33]{Stix}). For example, $E[|G|]$ is $(\Z/|G|)^2$ with a certain action and $f$ corresponds to a map $f'\colon G \to (\Z/|G|)^2$ with image in the $\pi_1(S,s)$-invariants.

 Let $K$ be the kernel of $f'$. The resulting map $\overline{f}'\colon G/K \to (\Z/|G|)^2$ corresponds to a homomorphism $\overline{f}\colon G/K\times S \to E[|G|]$ over $S$ that is an injection on the fiber over $s$ and hence over every other geometric point. 

 If we have conversely a subgroup $K\subset G$ and a $G/K$-level structure on $\EE$, this defines an $S$-homomorphism $G\times S \to E$ by precomposition with $G \to G/K$.
\end{proof}

Using that $\MM_G\simeq \Hom(G,\Euni)$ for $G$ finite abelian, this induces a corresponding splitting of the $E_\infty$-ring spectra $\TMF^G[\tfrac1{|G|}]$ (as taking global sections on $\MM_G^{top}$ commutes with localizations by Proposition \ref{prop:0affineG}). For example,
\[\TMF^{\Z/n}[\tfrac1n] \simeq \prod_{k|n} \TMF_1(k)[\tfrac1n]\]
as $E_\infty$-ring spectra and by Theorem \ref{thm:dectopuncompact} every factor with $k>3$ decomposes after further completion into well-understood pieces. We obtain more generally:

\begin{thm}\label{thm:Gsplitting}
 Let $G$ be a finite abelian group. After completion at a prime $l$ not dividing $|G|$, the $\TMF$-module $\TMF^G$ splits into one copy of $\TMF$ and even suspensions of $\TMF_1(3)$ (for $l=2$), $\TMF_1(2)$ (for $l=3$) or $\TMF$ (if $l>3$).
\end{thm}
\begin{proof} We localize implicitly at $l$.
 By the last lemma, we have
 $$\TMF^G \simeq \prod_{K\subset G} \OO^{top}(\MM^{G/K}).$$
 Assume that $G/K$ is non-trivial and write $G/K \cong \Z/k_1 \oplus \Z/k_2$ with $k_1\geq 2$. Then the map $h\colon \MM^{G/K} \to \MM_{ell}$ factors over $\MM_1(k_1)$. By Lemma \ref{lem:FiniteFlat} the map $\MM^{G/K} \to \MM_1(k_1)$ is finite and flat and thus the pushforward of $\OO_{\MM^{G/K}}$ to $\MM_1(k_1)$ is a vector bundle. By Theorem \ref{thm:decuncompact}, the pushforward of this vector bundle to $\MM_{ell}$, that is $h_*\OO_{\MM^{G/K}}$, decomposes into vector bundles of the form $\omega^{\tensor j}$ (for $l>3$), $(f_2)_*\OO_{\MM_1(2)} \tensor \omega^{\tensor j}$ (for $l=3$) or $(f_3)_*\OO_{\MM_1(3)}\tensor \omega^{\tensor j}$ (for $l=2$) after $l$-completion. Analogously to the proof of Theorem \ref{thm:dectopuncompact}, we conclude a topological splitting from this algebraic one.
\end{proof}


To make good on one more claim in the introduction, we still want to construct the map $\TMF^{\Sigma_n} \to \TMF_0(n)$ from the fixed points for the symmetric group. We note that the association $G\mapsto \TMF^G$ actually refines to a functor 
\[\left(\Orb^{\fin}\right)^{op} \to E_{\infty}\text{-ring Spectra}\] 
from the global orbit $2$-category, whose objects are finite groups, whose morphisms are group homomorphisms and whose $2$-morphisms are given by conjugation in the target \cite{LurEllIII}. The category $\Orb^{\fin}$ is denoted by $\mathrm{Group}^+$ in \cite[Remark 3.1.3]{LurEllIII}. In this $2$-category, there is a natural action of $(\Z/n)^\times$ on $\Z/n\to \Sigma_n$ as an object of the slice category $\Orb^{\fin}_{/\Sigma_n}$. Here we use that every element of $(\Z/n)^\times$ defines an element of $\Sigma_n$ by acting on $\Z/n$ if we identify $\Sigma_n$ with the set automorphisms of $\Z/n$. Thus, we obtain a morphism $\TMF^{\Sigma_n} \to (\TMF^{\Z/n})^{h(\Z/n)^\times}$. Using the map
\[\TMF^{\Z/n} \to \TMF^{\Z/n}[\tfrac1n] \simeq \prod_{k|n}\TMF_1(k)[\tfrac1n] \to \TMF_1(n) \]
we obtain a map $(\TMF^{\Z/n})^{h(\Z/n)^\times} \to \TMF_1(n)^{h(\Z/n)^\times} \simeq \TMF_0(n)$. Composition produces the claimed map $\TMF^{\Sigma_n} \to \TMF_0(n)$. 

\section{Duality}\label{sec:dualityboth}
In this section, we will give precise conditions when the spectrum $\Tmf_1(n)$ is Anderson self-dual up to integral shift. In the first subsection, we recall the definition and basic properties of Anderson duality. In the second subsection, we study the dualizing sheaf of the stack $\MMb_1(n)$, which forms the algebraic basis for the proof of our main theorem about duality in the third subsection.

\subsection{Anderson duality}\label{sec:AndersonDuality}
Let us recall the definition of Anderson duality, which was first studied by Anderson (only published in mimeographed notes \cite{Anderson}) and Kainen \cite{Kainen}, mainly for the purpose of universal coefficient sequences, and further investigated in the context of topological modular forms in \cite{Sto12}, \cite{H-M17}, \cite{G-M16} and \cite{G-SGorenstein}.

For an injective abelian group $J$, the functor
\[\mathrm{Spectra} \to \text{graded abelian groups},\quad X \mapsto \Hom_\Z(\pi_{-*}X, J)\]
is representable by a spectrum $I_J$, as follows from Brown representability. If $A$ is an abelian group and $A \to J^0 \to J^1$ an injective resolution, we define the spectrum $I_A$ to be the fiber of $I_{J^0}\to I_{J^1}$. Given a spectrum $X$, we define its \emph{$A$-Anderson dual} $I_AX$ to be the function spectrum $F(X, I_A)$. It satisfies for all $k\in\Z$ the following functorial short exact sequence:
\[0 \to \Ext^1_\Z(\pi_{-k-1}X, A) \to \pi_kI_AX \to \Hom_\Z(\pi_{-k}X, A) \to 0.\]
Note that if $A$ is a subring of $\Q$ and $\pi_{-k-1}X$ is a finitely generated $A$-module, the $\Ext$-group is isomorphic to the torsion in $\pi_{-k-1}X$. Considering for spectra $E$ and $X$ the Anderson dual of $E\sm X$, we obtain more generally a universal coefficient sequence
\[0 \to \Ext^1_\Z(E_{k-1}X, A) \to (I_AE)^kX \to \Hom_\Z(E_kX, A) \to 0.\]
This is most useful in the case of \emph{Anderson self-duality}, i.e.\ if $I_AE$ is equivalent to $\Sigma^mE$ for some $m$, as then the middle-term can be replaced by $E^{k+m}X$. Such Anderson self-duality is, for example, true for $E = KU$ (with $m=0$) and $E = KO$ (with $m=4$) by \cite{Anderson}. Moreover, Stojanoska showed
\begin{thm}\label{thm:Stojanoska}
 The Anderson dual $I_{\Z}\Tmf$ of $\Tmf$ is equivalent to $\Sigma^{21}\Tmf$.
\end{thm}

This was shown in \cite{Sto12} after inverting $2$ and announced in \cite{Sto14} in general. The author has learned in a talk by John Greenlees that another proof can be given by showing that $\tmf$ is Gorenstein.

Our goal will be to investigate which $\Tmf_1(n)$ are Anderson self-dual. The crucial ingredient will be to understand the dualizing sheaf on $\MMb_1(n)$, which will be the subject of the next subsection.

\subsection{The dualizing sheaf of $\MMb_1(n)$}\label{sec:duality}
In this section, we work implicitly over a field $K$ of characteristic zero until further notice. Recall that we use the notation $f_n$ for the projection $\MMb_1(n) \to \MMb_{ell}$. The question we aim to answer in this subsection is the following.
\begin{question}
When does it happen that $(f_n)^*\omega^{\tensor k}$ is isomorphic to the dualizing sheaf $\Omega^1_{\MMb_1(n)/K}$ for $\MMb_1(n)$ for some $k$?
\end{question}
  We remark that this question is not only relevant for the study of Anderson duality, but a positive answer also implies that for the decompositions
\begin{align*}(f_n)_*\OO_{\MMb_1(n)} &\cong \bigoplus_{i\in\Z} \bigoplus_{l_j} \omega^{\tensor (-j)}\\
\Tmf_1(n)_{\Q} &\simeq \bigoplus_{i\in\Z}\bigoplus_{l_j} \Sigma^{2j}\Tmf_{\Q}
\end{align*}
the numbers $l_j$ are symmetric in the sense that $l_{10+k-j} = l_j$. Indeed, using that $\MMb_{ell}$ has dualizing sheaf $\omega^{\tensor (-10)}$ (see \cite[Lemma 2.8]{MeiDecMod}), we see that
\[(f_n)_*\OO_{\MMb_1(n)} \text{ and }((f_n)_*\OO_{\MMb_1(n)})^{\vee} \tensor \omega^{\tensor (k-10)}\] have the same global sections, even after tensoring with an arbitrary $\omega^{\tensor j}$. Thus, these vector bundles must be isomorphic as all vector bundles on $\MMb_{ell} \simeq \PP_K(4,6)$ decompose into copies of $\omega^{\tensor ?}$ by \cite[Proposition A.4]{MeiDecMod}.
  We recommend to look at the tables in \cite[Appendix C]{MeiDecMod} for this symmetry.

Recall that $H^1(\MMb_1(n); \Omega^1_{\MMb_1(n)/K}) \cong K$ and $\dim_K H^0(\MMb_1(n);\Omega^1_{\MMb_1(n)/K})$ is the genus $g$ of $\MMb_1(n)$.\footnote{If $n\leq 4$ and thus $\MMb_1(n)$ is not representable by a scheme, it is rather the genus of the coarse moduli space of $\MMb_1(n)$. As these four cases are easily dealt with by hand, we will usually concentrate on the cases $n\geq 5$.}
As $H^1(\MMb_1(n);(f_n)^*\omega^{\tensor k}) = 0$ for $k\geq 2$ by Proposition \ref{prop:coh},
\begin{align}\label{eq:dualiso}(f_n)^*\omega^{\tensor k}\cong\Omega^1_{\MMb_1(n)/K}\end{align}
can only happen for $k\leq 1$. Our strategy will be to treat our question step by step for $k\leq -1$, for $k=0$ and for $k=1$. Our main method will be to use the following degree formula, which follows from \cite{D-R73} and \cite{D-S05} as explained in \cite[Lemma 4.3]{MeiDecMod}.

\begin{lemma}\label{lem:degomega}
For $n\geq 5$, the degree $\deg (f_n^*\omega)$ equals $\tfrac1{24} d_n$ for $d_n$ the degree of the map $f_n\colon \MM_1(n) \to \MM_{ell}$. We have
\begin{align*}
 d_n &= \sum_{d|n}d\varphi(d)\varphi(n/d) \\
 &= n^2\prod_{p|n}(1-\tfrac1{p^2})
\end{align*}
\end{lemma}

As there are no modular forms of negative weight, by Proposition \ref{prop:coh} the isomorphism \eqref{eq:dualiso} can only hold for $i\leq -1$ if $g = 0$.
The genus zero cases are only $1\leq n \leq 10$ and $n =12$. We claim that
\begin{align*}
\Omega^1_{\MMb_{ell}/K} &\cong \omega^{\tensor -10} \\
\Omega^1_{\MMb_1(2)/K} &\cong (f_2)^*\omega^{\tensor -6} \\
\Omega^1_{\MMb_1(3)/K} &\cong (f_3)^*\omega^{\tensor -4} \\
\Omega^1_{\MMb_1(4)/K} &\cong (f_4)^*\omega^{\tensor -3} \\
\Omega^1_{\MMb_1(5)/K} &\cong (f_5)^*\omega^{\tensor -2} \\
\Omega^1_{\MMb_1(6)/K} &\cong (f_6)^*\omega^{\tensor -2} \\
\Omega^1_{\MMb_1(7)/K} &\cong (f_7)^*\omega^{\tensor -1} \\
\Omega^1_{\MMb_1(8)/K} &\cong (f_8)^*\omega^{\tensor -1}
\end{align*}
and that \eqref{eq:dualiso} does not hold for any $k$ if $n=9,10$ or $12$. Indeed, in the non-representable cases $1\leq n\leq 4$ this can be checked explicitly as in \cite[Example 2.1, Theorem A.2] {MeiDecMod}. In the other cases, we have $\MMb_1(n) \simeq\mathbb{P}^1$ (see e.g.\ \cite[Example 2.5]{MeiDecMod}). On $\mathbb{P}^1$ a line bundle is determined by its degree and $\Omega^1_{\mathbb{P}^1/K}$ has degree $-2$. Lemma \ref{lem:degomega} implies our claim.

A projective smooth curve has genus $1$ if and only if the canonical sheaf agrees with the structure sheaf. The curve $\MMb_1(n)$ has genus $g=1$ if and only if $n=11,14,15$. This is well-known and can be easily proven using the genus formula
$$g = 1+ \tfrac{d_n}{24} - \tfrac14 \sum_{d|n}\varphi(d)\varphi(n/d)$$
from \cite[Sec 3.8+3.9]{D-S05} and easier analogues of the methods of Proposition \ref{prop:degreecomp}.

Now assume that $\Omega^1_{\MMb_1(n)/K} \cong (f_n)^*\omega$. Then, in particular, we have $2g-2 = \deg (f_n)^*\omega$. We want to show the following proposition, whose proof we learned from Viktoriya Ozornova.

\begin{prop}[Ozornova]\label{prop:degreecomp}
 We have $2g-2 = \deg (f_n)^*\omega$ if and only if $n = 23,32,33,35,40$ or $42$.
\end{prop}

By Lemma \ref{lem:degomega} and the genus formula above, we see that we have just have to solve for $n$ in the equation
\[
 \tfrac1{12}\sum_{d|n}d\varphi(d)\varphi(n/d)=\sum_{d|n} \varphi(d)\varphi\left(\tfrac{n}{d}\right).\qedhere
\]

 \begin{lemma}
The inequality
\[
 \tfrac{1}{12}\sum_{d|n} d\varphi(d)\varphi\left(\tfrac{n}{d}\right)>\sum_{d|n} \varphi(d)\varphi\left(\tfrac{n}{d}\right)
\]
holds for every natural number $n>144$.
 \end{lemma}
\begin{proof}We have the following chain of (in)equalities:
 \[
  \begin{aligned}
    \tfrac{1}{12}\sum_{d|n} d\varphi(d)\varphi\left(\tfrac{n}{d}\right)&= \tfrac{1}{12}\sum_{d|n} \tfrac{1}{2}\left(d+\tfrac{n}{d}\right)\varphi(d)\varphi\left(\tfrac{n}{d}\right)\\
&\geq   \tfrac{1}{12}\sum_{d|n} \sqrt{n}\varphi(d)\varphi\left(\tfrac{n}{d}\right)\\
&\stackrel{{\tiny \sqrt{n}>12}}{>}\sum_{d|n} \varphi(d)\varphi\left(\tfrac{n}{d}\right).\qedhere
  \end{aligned}
 \]
\end{proof}
\begin{proof}[Proof of proposition:]
 The proof can be easily finished by a computer search of all values up to $144$. For a proof by hand, one can argue as follows: Set $f(n) = \sum_{d|n} d\varphi(d)\varphi\left(\tfrac{n}{d}\right)$ and $g(n) = \sum_{d|n} \varphi(d)\varphi\left(\tfrac{n}{d}\right)$. Both functions are multiplicative for coprime integers; thus, we only have to understand them on prime powers smaller or equal to $144$. For $p$ a prime, we have $\tfrac1{12}f(p) = g(p)$ exactly for $p=23$ and $\tfrac1{12}f(p)>g(p)$ for $p>23$; thus, the equation cannot have solutions $n$ with prime factors $>23$. The other solutions can now easily be deduced from the following table with the relevant values of $\tfrac{g(p^m)}{f(p^m)}$:
\begin{center}
\renewcommand{\arraystretch}{1.2}
\begin{tabular}{|c|c|c|c|c|c|c|c|c|}
\hline
$m$, $p$ & $2$ & $3$ & $5$ & $7$ & $11$ & $13$ & $17$ & $19$\\ \hline
$1$ & $\frac{2}{3}$ & $\frac{1}{2}$ & $\frac{1}{3}$ & $\frac{1}{4}$ & $\frac{1}{6}$ & $\frac{1}{7}$ & $\frac{1}{9}$ & $\frac{1}{10}$\\ \hline
$2$ & $\frac{5}{12}$ & $\frac{2}{9}$ & $\frac{7}{75}$ & $\frac5{98}$ & $\frac8{363}$ &  &  & \\ \hline
$3$ & $\frac{1}{4}$ & $\frac{5}{54}$ & $\frac{3}{125}$ &  &  &  &  & \\ \hline
$4$ & $\frac{7}{48}$ & $\frac{1}{27}$ &  &  &  &  &  & \\ \hline
$5$ & $\frac{1}{12}$ &  &  &  &  &  &  & \\ \hline
$6$ & $\frac3{64}$ & & & & & & & \\ \hline
$7$ & $\frac{5}{192}$ & & & & & & & \\ \hline
\end{tabular}
\end{center}
\end{proof}

A line bundle $\LL$ on $\MMb_1(n)$ is isomorphic to $\Omega^1_{\MMb_1(n)/K}$ if and only if $\deg \LL = \deg \Omega^1_{\MMb_1(n)/K}$ and $\dim_K H^1(\MMb_1(n); \LL) = 1$. (Indeed, then $\Omega^1_{\MMb_1(n)/K} \tensor \LL^{-1}$ is a line bundle with a nonzero global section by Serre duality and this global sections induces a morphism $\OO_{\MMb_1(n)} \to \Omega^1_{\MMb_1(n)/K} \tensor \LL^{-1}$ with no zeros as the target has degree $0$.) In the case of $\LL = (f_n)^*\omega$, we have $\dim_K H^1(\MMb_1(n); \LL) = s_1$, the dimension of the space of weight-$1$ cusp forms for $\Gamma_1(n)$ (see e.g.\ \cite[Corollary 2.11]{MeiDecMod}). Among the values from Proposition \ref{prop:degreecomp}, we only have $s_1 = 1$ for $n=23$ as a simple \texttt{MAGMA} computation shows. (We remark that the case $n=23$ was already treated in \cite[Section 1]{Buz} by hand.)

Thus, we have proven the following proposition:
\begin{prop}
 We have $\Omega^1_{\MMb_1(n)/K} \cong (f_n)^*\omega^{\tensor k}$ for some $k$ if and only if
 \[n=1,2,3,4,5,6,7,8, 11, 14, 15 \text{ or }23.\]
\end{prop}

\begin{cor}\label{cor:Omega}
 If we view $\MMb_1(n)$ as being defined over $\Z[\tfrac1n]$, we have
 $$\Omega^1_{\MMb_1(n)/\Z[\tfrac1n]} \cong (f_n)^*\omega^{\tensor k}$$
 for some $k$ if and only if $n=1,2,3,4,5,6,7,8, 11, 14, 15 \text{ or }23.$
\end{cor}
\begin{proof}
For $n=1$, this is \cite[Lemma 2.8]{MeiDecMod}. It is easy to check in the cases $2\leq n\leq 5$ as in these cases $\MMb_1(n)$ is a weighted projective line by \cite[Examples 2.1]{MeiDecMod}. We want to show that the restriction functor $\Pic(\MMb_1(n)) \to \Pic((\MMb_1(n))_\Q)$ is injective for $n\geq 5$, where $\MMb_1(n)$ is representable. If two line bundles $L,L'$ on $\MMb_1(n)$ are isomorphic on $\MMb_1(n)_{\Q}$, they are already isomorphic on the non-vanishing locus of some  $f\in \Z$. The elements $f$ factors into prime factors $p_1\cdots p_n$. By succesively applying \cite[Proposition 6.5]{Har77}, we see that $\Pic(\MMb_1(n)) \to \Pic(\MMb_1(n)_{\Z[\tfrac1f]})$ is an isomorphism.
\end{proof}
\subsection{Anderson self-duality of $\Tmf_1(n)$}\label{sec:AD}
 In this section, we will investigate possible Anderson self-duality of $\Tmf_1(n)$.

Recall that in Corollary \ref{cor:Omega}, we gave conditions when the dualizing sheaf of $\MMb_1(n)$ is a power of $f_n^*\omega$ for $f_n\colon \MMb_1(n) \to \MMb_{ell,\Z[\tfrac1n]}$ the structure map. We will explain how this implies Anderson self-duality for $\Tmf_1(n)$ in these cases once we know that $\pi_*\Tmf_1(n)$ is torsionfree. We will assume throughout that $n\geq 2$.

\begin{lemma}\label{lem:torsionfree2}
 If $\Omega^1_{\MMb_1(n)/\Z[\tfrac1n]} \cong (f_n)^*\omega^{\tensor k}$, then the cohomology groups $H^1(\MMb_1(n); (f_n)^*\omega^{\tensor j})$ are torsionfree for all $j$.
\end{lemma}
\begin{proof}Proposition \ref{prop:coh} states the torsionfreeness of $H^1(\MMb_1(n); (f_n)^*\omega^{\tensor j})$ for $j\neq 1$. It remains to show it for $j=1$. We can assume that $n\geq 5$ so that $\MMb_1(n)$ is representable as the other cases are easily dealt with by hand. By Section \ref{sec:duality} we furthermore know that $i\leq 1$.

If $k \leq 0$, then $H^1(\MMb_1(n)_{\Q};f_n^*\omega) = 0 = H^1(\MMb_1(n)_{\F_p};f_n^*\omega)$ by Serre duality for all primes $p$ not dividing $n$ because there are no modular forms of negative weight by Proposition \ref{prop:coh} again. Thus, by cohomology and base change $H^1(\MMb_1(n);f_n^*\omega) = 0$ (see e.g.\ \cite[Corollary 12.9]{Har77}). If $k = 1$, then Grothendieck duality states that $H^1(\MMb_1(n);f_n^*\omega) \cong \Z[\tfrac1n]$.
\end{proof}
This implies that $\pi_*\Tmf_1(n)$ is torsionfree if $\Omega^1_{\MMb_1(n)/\Z[\tfrac1n]} \cong (f_n)^*\omega^{\tensor k}$ for some $k$. Indeed, as noted in Section \ref{sec:TMFLevel} we have
$$\pi_{2i}\Tmf_1(n) \cong H^0(\MMb_1(n); (f_n)^*\omega^{\tensor i}) \;\text{ and }\; \pi_{2i-1}\Tmf_1(n)\cong H^1(\MMb_1(n); (f_n)^*\omega^{\tensor i}).$$
The former is torsionfree because $\MMb_1(n)$ is flat over $\Z[\tfrac1n]$ and the latter by the last lemma.

The following two lemmas will also be useful.
\begin{lemma}\label{lem:rational}
Let $A$ be a subring of $\Q$ and $X$ a spectrum whose homotopy groups are finitely generated $A$-modules. Then $(I_AX)_{\Q} \to I_{\Q}X_{\Q}$ is an equivalence, where $X_{\Q}$ denotes the rationalization.
\end{lemma}
\begin{proof}
Recall that $I_{\Q}X$ is defined as $F(X,I_{\Q})$, which is equivalent to $F(X_{\Q}, I_{\Q})$ as $I_{\Q}$ is rational. We have to show that the natural map $I_AX \to I_{\Q}X_{\Q}$ is an isomorphism on homotopy groups after rationalization. This boils down to the facts that for a finitely generated $A$-module $M$ we have
$$\Ext_A(M,A) \tensor \Q = 0$$
and that
$$\Hom(M,A)\tensor \Q \to \Hom(M, \Q) \cong \Hom_{\Q}(M\tensor \Q, \Q)$$
is an isomorphism.
\end{proof}

\begin{lemma}\label{lem:ample}
Let $\Gamma$ be $\Gamma_1(n)$ or $\Gamma(n)$ and assume that $\MMb(\Gamma)_R$ is representable. Then the line bundle $g^*\omega$ is ample on $\MMb(\Gamma)_R$.
\end{lemma}
\begin{proof}
Denote the map $\MMb_{ell,R} \to \mathbb{P}_R^1$ to the coarse moduli space by $\pi$. We claim first that $\pi^*\OO(1) \cong \omega^{\tensor 12}$. Because $\Pic(\MMb_{ell}) \cong \Z$ is generated by $\omega$ \cite{F-O10}, we just have to show which $\omega^{\tensor m}$ the pullback $\pi^*\OO(1)$ is isomorphic to and we can do it over $\C$. As $\MMb_1(5)$ is a smooth proper curve of genus $0$, we have $\MMb_1(5)_{\C} \simeq \mathbb{P}^1_{\C}$ and by \cite[Sec.\ 3.8+3.9]{D-S05} we know that the composition
$$\MMb_1(5)_{\C} \xrightarrow{g} \MMb_{ell,\C} \xrightarrow{\pi} \mathbb{P}_{\C}^1$$
 has degree $12$. Thus, $(\pi g)^*\OO(1) \cong \OO(12)$. By Lemma \ref{lem:degomega}, $g^*\omega$ has degree $1$.  Thus, it follows that $\pi^*\OO(1) \cong \omega^{\tensor 12}$.

The composition $\pi \colon \MMb(\Gamma)_R \to \mathbb{P}_R^1$ is finite as $\MMb(\Gamma)_R \to \MMb_{ell,R}$ is finite and $\MMb_{ell,R} \to \mathbb{P}_R^1$ is quasi-finite and proper. Thus, $g^*\omega^{\tensor 12} \cong (\pi g)^*\OO(1)$ is ample and thus $g^*\omega$ is ample as well (see \cite[Prop 13.83]{G-W10}).
\end{proof}

\begin{prop}
 We have $I_{\Z[\tfrac1n]}\Tmf_1(n) \simeq \Sigma^m \Tmf_1(n)$ as $\Tmf_1(n)$-modules for $n\geq 2$ if and only if $m$ is odd and $(f_n)^*\omega^{\tensor k}$ is isomorphic to $\Omega^1_{\MMb_1(n)/\Z[\tfrac1n]}$ for $k = (1-m)/2$.
\end{prop}
\begin{proof}
We will denote throughout the proof the sheaf $(f_n)^*\omega$ by $\omega$. Furthermore, we abbreviate $\Tmf_1(n)$ to $R$. As noted before, we have $\pi_{2i}R \cong H^0(\MMb_1(n); \omega^{\tensor i})$ and $\pi_{2i-1}R \cong H^1(\MMb_1(n); \omega^{\tensor i})$.

First suppose that $\omega^{\tensor k}$ is a dualizing sheaf for $\MMb_1(n)$. Then $H^1(\MMb_1(n); \omega^{\tensor k}) \cong \Z[\tfrac1n]$ and the pairing
$$H^0(\MMb_1(n); \omega^{\tensor (k-i)}) \tensor H^1(\MMb_1(n); \omega^{\tensor i}) \to  H^1(\MMb_1(n); \omega^{\tensor k})\cong \Z[\tfrac1n]$$
is perfect by Grothendieck--Serre duality; note here that all occurring groups are finitely generated $\Z[\tfrac1n]$-modules and torsionfree by Lemma \ref{lem:torsionfree2}. As note above, this implies that $\pi_*R$ is torsionfree as well.

Choose a generator $D$ of $H^1(\MMb_1(n); \omega^{\tensor k})$.  This is represented by a unique element in $\pi_{2k-1}R \cong \Z[\tfrac1n]$, which we will also denote by $D$. Denote by $\delta$ the element in $\pi_{1-2k} I_{\Z[\tfrac1n]} R$ with $\phi(\delta)(D) = 1$, where $\phi\colon \pi_{1-2k} I_{\Z[\tfrac1n]} R \xrightarrow{\cong} \Hom(\pi_{2k-1}R,\Z[\tfrac1n])$. The element $\delta$ induces an $R$-linear map $\widehat{\delta}\colon \Sigma^{1-2k}R\to I_{\Z[\tfrac1n]} R$.

We obtain a diagram
\begin{equation}\label{AndersonSquare}
\begin{gathered}
 \xymatrix@C=1.3em{\pi_{i+2k-1}R \tensor \pi_{-i}R \ar[rr]^-{\widehat{\delta}_*\tensor \id}\ar[d] && \pi_iI_{\Z[\tfrac1n]}R\tensor \pi_{-i}R \ar[rr]_-{\cong}^-{\phi\tensor\id} && \Hom(\pi_{-i}R, \Z[\tfrac1n]) \tensor \pi_{-i}R\ar[d]^{\ev} \\
\pi_{2k-1}R\ar[rrrr]^{\phi(\delta)}_{\cong} &&&& \Z[\tfrac1n], }
\end{gathered}
\end{equation}
which is commutative up to sign.

The left vertical map is a perfect pairing because of Serre duality (as described above), as is the right vertical map by definition. Thus, the map
$$\widehat{\delta}_*\colon \pi_{i+2k-1}R \to \pi_kI_{\Z\left[\tfrac1n\right]}R$$
is an isomorphism for all $i$. This shows that $\widehat{\delta}$ is an equivalence of $R$-modules.

Now assume on the other hand that there is an equivalence $I_{\Z[\tfrac1n]}R \simeq \Sigma^m R$ as $R$-modules. By Lemma \ref{lem:rational}, this implies an equivalence $I_{\Q}R_\Q \simeq \Sigma^m R_\Q$ of $R_\Q$-modules. In the following, we will rationalize everything implicitly.

If $m$ is even, this implies
$$H^0(\MMb_1(n); \omega^{\tensor i}) \cong \pi_{2i}\Tmf_1(n) \cong (\pi_{-2i}\Sigma^m\Tmf_1(n))^{\vee} \cong H^0(\MMb_1(n); \omega^{\tensor (-m/2-i)})^{\vee}.$$
As there are no modular forms of negative weight, this would imply that $H^0(\MMb_1(n); \omega^{\tensor i})$ is zero for $i$ big; this is absurd as the ring of modular forms does not contain nilpotent elements. Thus, $m$ can be written as $1-2k$.

 Let
$$D \in H^1(\MMb_1(n); \omega^{\tensor k}) \cong \pi_{2k-1}R$$
be the element corresponding under the isomorphism
\begin{align*}
 \pi_{2k-1}R &\cong \pi_0\Sigma^{1-2k}R \\
 &\cong \pi_0I_{\Q}R \\
 &\cong \Hom_{\Q}(\pi_0R, \Q) \\
 &\cong \Hom_{\Q}(\Q,\Q)
\end{align*}
to $1$. Consider now again the diagram (\ref{AndersonSquare}), but now tensored with $\Q$. Now all the horizontal arrows are isomorphisms and the right vertical map is a perfect pairing. Therefore, the left hand vertical arrow is a perfect pairing as well. This implies that
$$H^0(\MMb_1(n); \omega^{\tensor k-i}) \tensor H^1(\MMb_1(n); \omega^{\tensor i}) \to  H^1(\MMb_1(n); \omega^{\tensor k})\cong \Q$$
is a perfect pairing. As $\omega$ is ample by Lemma \ref{lem:ample} if $n\geq 5$ and by \cite[Examples 2.1, Theorem A.2]{MeiDecMod} if $n<5$, one can repeat the proof of \cite[Thm 7.1b]{Har77} to see that $\omega^{\tensor k}$ is dualizing on $\MMb_1(n)_{\Q}$ and thus isomorphic to the dualizing sheaf $\Omega^1_{\MMb_1(n)_{\Q}/\Q} = \left(\Omega^1_{\MMb_1(n)/\Z[\tfrac1n]}\right)_{\Q}$. As in Corollary \ref{cor:Omega}, this implies that $\omega^{\tensor k} \cong \Omega^1_{\MMb_1(n)/\Z[\tfrac1n]}$ also before rationalizing.
\end{proof}

This implies the following theorem.
\begin{thm}\label{thm:Duality}
 We have $I_{\Z[\tfrac1n]}\Tmf_1(n) \simeq \Sigma^m \Tmf_1(n)$ as $\Tmf_1(n)$-modules for some $m$ if and only if

   \begin{tabular}{ll}
      $n = 1$ &with\;  $m=21$, \\
  $n = 2$ & with\;  $m=13$, \\
  $n = 3$ & with\; $m = 9$, \\
  $n=4$ & with\; $m = 7$,\\
  $n = 5,6$ & with\;  $m = 5$, \\
  $n = 7,8$ & with\; $m = 3$,\\
  $n = 11,14,15$ &with\; $m = 1$, or \\
  $n = 23$ & with\;  $m= -1$.
   \end{tabular}
\end{thm}
\begin{proof}
 The only case not dealt with by the last proposition and Corollary \ref{cor:Omega} is the case $n=1$, which is Stojanoska's Theorem \ref{thm:Stojanoska}.
\end{proof}

\section{$C_2$-equivariant refinements}\label{sec:C2}
The goal of this section, is to refine the previous decomposition results by taking certain $C_2$-actions into account, where $C_2$ denotes the group with two elements.

Many $C_2$-spectra in chromatic homotopy theory are underlying even and $C_2$ acts on $\pi_{2n}$ as $(-1)^n$. This determines the $E_2$-term of the homotopy fixed point spectral sequence completely and when the differentials behave ``as expected'' we call the spectral sequence \emph{regular}; see Definition \ref{def:regular} for the precise definition. We will show that this has very strong implications for the $C_2$-spectrum; we obtain for example all slices, which are automatically concentrated in even degrees. 

Such a theory would be pointless without a good supply of examples. The archetypical example is the Real bordism spectrum $M\R$ as introduced by Landweber \cite{Lan68} and studied in detail by Araki \cite{Ara79} and Hu--Kriz \cite{H-K01}. Hahn and Shi \cite{HahnShi} used an equivariant map from $M\R$ to Lubin--Tate spectra to show that the latter have regular homotopy fixed point spectral sequenc as well. 

Our first main aim in this section will be to show that the spectra $\TMF_1(n)$ have for all $n$ regular homotopy fixed point spectral sequence as well, thereby determining this spectral sequence completely. This uses the result of Hahn--Shi together with the descent result Proposition \ref{prop:injective} : we can descend the property of having a regular homotopy fixed point spectral sequence from $R'$ to $R$ along a map $R \to R'$ of $C_2$-spectra, provided that it is ``injective enough'' on homotopy groups. 

The computation of $\pi_*\TMF_1(n)^{hC_2}$ is new for all $n>5$. In contrast to some other methods it is completely independent from our knowledge of $\pi_*\mathbb{S}$. Moreover, our method is rather formal, essentially using only that $\TMF_1(n)$ is Landweber exact of height $n$ and that $\pi_*\TMF_1(n)$ and $\pi_*\TMF_1(n)/2$ are integral domains. Thus our method should be easily applicable to the computation of $C_2$-homotopy fixed points of other Landweber exact spectra of height $2$ and with some extra effort also to some of higher height. 

Our new knowledge of $\TMF_1(n)^{hC_2}$  will be the basis for $C_2$-refinements of our main splitting results in the last subsection. In Remark \ref{rem:AndersonC2} we will discuss the implications for Anderson duality.

\subsection{Regular homotopy fixed point spectral sequences}
Consider the homotopy category of spectra with a $C_2$ action and underlying weak equivalences and let $R$ be a commutative monoid in this category. Assume that the underlying homotopy groups of $R$ are $2$-local, torsionfree and concentrated in even degrees; moreover assume that the $C_2$-action acts on $\pi_{2n}R$ as $(-1)^n$. If we want to view $R$ as a genuine $C_2$-spectrum, we will always view it as cofree, i.e.\ as $R^{(EC_2)_+}$ so that $\pi^{C_2}_*R = \pi_*R^{hC_2}$. We will denote by $\sigma$ the sign representation of $C_2$ and by $\rho$ the real regular representation.

Let us consider the $RO(C_2)$-graded homotopy fixed point spectral sequence (HFPSS), converging to $\pi^{C_2}_{\bigstar}R^{(EC_2)_+}$, where $\bigstar$ stands for a grading in $RO(C_2)$ (see Section 2.3 of \cite{H-M17}). Its $E_2$-term is canonically and functorially isomorphic to $\overline{\pi_{2*}R}\tensor \Z[a,u^{\pm 1}]/2a$, where $|a| = (-\sigma, 1)$ and $|u| = (2-2\sigma,0)$ and $\overline{\pi_{2n}R}$ is isomorphic to $\pi_{2n}R$, but shifted to degree $(n\rho,0)$ (see \cite[Corollary 4.7]{H-M17}); here and everywhere we use Adams grading, i.e.\ $a \in H^1(C_2; \pi^e_{1+\sigma}R)$. We will denote the $E_n$-term of the HFPSS of $R$ by $E_n(R)$.

By \cite[Lemma 4.8]{H-M17}, the element $a$ is always a permanent cycle and the resulting element in $\pi_{-\sigma}^{C_2}R$ is the Hurewicz image of the inclusion $S^0 \to S^\sigma$. Moreover, an element $x\in E_k(R)$ is a $d_k$-cycle if $ax \in E_k(R)$ is a $d_k$-cycle. Indeed, the image of the $d_i$-differential is generated as an $\F_2[a]$-module by the elements $d_i(z)$ with $z$ of cohomological degree $0$. This implies that $E_k(R)$ has no $a$-torsion above the $(k-1)$-line.

Recall that the elements $v_n\in \pi_{2(2^n-1)}R$ are well-defined modulo $(2,v_1,\dots, v_{n-1})$. We obtain elements $\vb_n$ in the $E_2$-term of the $RO(C_2)$-graded homotopy fixed point spectral sequence that are well-defined modulo $(2,\vb_1,\dots, \vb_{n-1})$. We set $v_0 = 2$ and $\vb_0 =2$ as well.

\begin{defi}\label{def:regular}
 We say that the HFPSS for $R$ is \emph{regular} if the following three conditions are fulfilled:
 \begin{enumerate}
  \item $\overline{\pi_{2*}R}$ consists of permanent cycles.
  \item The element $u^{2^n}$ survives to the $E_{2^{n+2}-1}$-page and we have $d_{2^{n+2}-1}(u^{2^n}) = a^{2^{n+2}-1}\vb_{n+1}$.
 \item If $d_{2^{n+2}-1}(u^{2^nm}\vb)$ is zero for some odd number $m$ and some $\vb \in \overline{\pi_{2*}R}$, then the element $u^{2^nm}\vb$ is already a permanent cycle.
 \end{enumerate}
\end{defi}

Our next aim is to see that the three conditions above determine the structure of a regular HFPSS rather completely. This will allow us to show the crucial Proposition \ref{prop:injective}, which in turn is crucial for proving that many $C_2$-spectra have regular HFPSS.

\begin{lemma}\label{lem:regmisc}
Assume that $R$ has regular HFPSS.
\begin{enumerate}
\item If $d_i$ is a non-trivial differential, then $i= 2^k-1$ for some $k\geq 2$.
\item If $d_{2^k-1}(x) \neq 0$, then $x = u^{2^{k-2}}y$ with $d_{2^k-1}(y) = 0$.
\item Let $Z_i \subset E_2(R)$ be the subgroup of all $x \in E_2$ with $d_j(x) = 0$ for $j< i$. With this notation, the map
$$Z_{2^{k+1}} \subset E_2(R) \to E_2(R)/(\vb_1a^3, \dots, \vb_k a^{2^{k+1}-1})$$
descends to a map $g_k\colon E_{2^{k+1}}(R) \to  E_2(R)/(\vb_1a^3, \dots, \vb_k a^{2^{k+1}-1}).$
\end{enumerate}
\end{lemma}
\begin{proof}
Assume that an element $x = a^lu^{2^km}\vb$ is a $d_j$-cycle for all $j<i$ for some $\vb \in \overline{\pi_{2*}R}$ and $m$ odd. As noted above, $d_i(a^lu^{2^km}\vb) =0$ if and only if $d_i(u^{2^km}\vb) =0$. Thus, we can assume for the first two items that $l=0$.

By definition of a regular HFPSS, $d_i(x)$ is zero for $i<2^{k+2}-1$. If $d_{2^{k+2}-1}(x) =0$, the element $x$ is a permanent cycle. This shows the first item. For the second, note that $x = u^{2^k}(u^{2^{k+1}})^{\tfrac{m-1}2}\vb$.

For the last item, we have to check that the kernel of $Z_{2^{k+1}} \to E_{2^{k+1}}(R)$ is contained in the ideal $(\vb_1a^3,\dots, \vb_ka^{2^{k+1}-1})$ in $E_2(R)$. We will argue by induction and assume that
$$\ker(Z_{2^k} \to E_{2^k}(R)) \subset (\vb_1a^3,\dots, \vb_{k-1}a^{2^k-1})E_2(R).$$
Let $C$ be the image of $Z_{2^{k+1}}$ in $E_{2^k}$. Using the first two items of this lemma, the kernel of $C \to E_{2^{k+1}}$ is generated by elements of the form
$$d_{2^{k+1}-1}(u^{2^{k-1}}y) = \vb_ka^{2^{k+1}-1}y.$$
As the kernel of $Z_{2^{k+1}} \to C$ is contained in $(\vb_1a^3,\dots, \vb_{k-1}a^{2^k-1})E_2(R)$, the statement follows.
\end{proof}

\begin{lemma}\label{lem:reginjection}
Assume that $R$ has regular HFPSS and that $v_k$ is either zero or a non zero-divisor in $\pi_{2*}R/(2,v_1,\dots, v_{k-1})$ for all $k$. Then the maps $g_k$ from the last lemma are injections.
\end{lemma}
\begin{proof}
We need to show that if $z\in (\vb_1a^3, \dots, \vb_ka^{2^{k+1}-1})E_2(R)\cap Z_{2^{k+1}}$, then $z=0$ in $E_{2^{k+1}}$. Thus let $z = a^{2^{c+1}-1}\vb_cy_c + \cdots a^{2^{k+1}-1}\vb_k y_k$ be in $Z_{2^{k+1}}$ with $y_i = a^{l_i}u^{2^{t_i}m_i}\wb_i \in E_2$ with $m_i$ odd and $\wb_i \in \overline{\pi_{2*}R}$. We can and will assume that $y_i$ is zero if $\wb_i$ or $\vb_i$ is in $(\vb_1,\dots, \vb_{i-1})$.

Let $j\leq c$ and assume by induction that all $y_i$ are $d_{2^{j'+1}-1}$-cycles for all $j' < j$. We want to show that all $y_i$ are $d_{2^{j+1}-1}$-cycles. Let $\lambda_i = 1$ if $y_i\neq 0$ with $t_i = j-1$ and $0$ else. For $l$ the cohomological degree of $z$ we have
\begin{align*}
d_{2^{j+1}-1}(z) =
 a^l u^? a^{2^{j+1}-1}\vb_j(\lambda_c\vb_c\wb_c+\cdots \lambda_k\vb_k\wb_k),
\end{align*}
where $? = \tfrac14(f-l-2^{j+1})$ with $z$ of degree $f+ \ast \rho$. This vanishes in $E_{2^{j+1}-1}(R)$ and thus
$$a^{l+2^{j+1}-1} u^?\vb_j(\lambda_c\vb_c\wb_c+\cdots \lambda_k\vb_k\wb_k) \in  (\vb_1a^3,\dots, \vb_{j-1}a^{2^j-1})E_2(R)$$
by the last lemma.
This implies that
$$\vb_j(\lambda_c\vb_c\wb_c+\cdots \lambda_k\vb_k\wb_k) = 0 \in \overline{\pi_{2*}R}/(2, \vb_1, \dots, \vb_{j-1}),$$ similarly to the argument used later in Proposition \ref{prop:injective}. Assume first that $\vb_j$ is zero in this quotient. Then $d_{2^{j+1}-1}(y_i) = \lambda_ia^{2^{j+1}-1}u^{2^{t_i}(m_i-1)}\vb_j\wb_i$ vanishes for all $i$ as required. Else our assumptions imply that $\lambda_c\vb_c\wb_c+\cdots \lambda_k\vb_k\wb_k$ is zero in $\overline{\pi_{2*}R}/(2, \vb_1, \dots, \vb_{j-1})$. Let $h$ be the highest index such that $\lambda_h \neq 0$ if such an index exists. Then $\vb_h\wb_h \in (2, \vb_1, \dots, \vb_{h-1})$, but this is impossible as $\vb_h$ is a non zero-divisor and $\wb_h$ is nonzero in this quotient. Thus, $\lambda_i$ must vanish for all $i$. This implies that the $y_i$ are cycles for $d_{2^{j+1}-1}$.

Continuing the argument, we see that $y_c$ is in $Z_{2^{c+1}}$ and thus
$$a^{2^{c+1}-1}\vb_c y_c = d_{2^{c+1}-1}(u^{2^{c-1}}y_c).$$
 Thus, $z$ equals $a^{2^{c+2}-1}\vb_{c+1}y_{c+1} + \cdots a^{2^{k+1}-1}\vb_k y_k$ in $E_{2^{k+1}}$ and inductively repeating this argument shows that $z$ is actually zero in $E_{2^{k+1}}$.
\end{proof}

\begin{remark}\label{rem:vi}
 Let $R$ be as in the last lemma and assume that $v_k \in (2,v_1,\dots, v_{k-1})\pi_{2*}R$, i.e.\ that it is zero up to its indeterminacy. Then the preceding lemma implies that $v_{k+1} \in (2,v_1,\dots, v_k)\pi_{2*}R$. Indeed: We have $d_{2^{k+1}-1}(u^{2^{k-1}}) = a^{2^{k+1}-1}\vb_k$. As $a^{2^{k+1}-1}\vb_i$ vanishes on $E_{2^{k+1}-1}$ for all $i<k$, our assumptions imply that $d_{2^{k+1}-1}(u^{2^{k-1}}) =0$ and thus that it is a permanent cycle. Thus, $d_{2^{k+2}-1}(u^{2^k}) =  a^{2^{k+2}-1}\vb_{k+1}$ must vanish as well. As $E^{2^{k+2}-1}$ injects into $E_2(R)/(\vb_1a^3, \dots, \vb_k a^{2^{k+1}-1})$ by the preceding two lemmas, it is not hard to show that $v_{k+1} \in (2,v_1,\dots, v_k)\pi_{2*}R$.

 In particular, this implies that if $R$ has regular HFPSS, the sequence $2, v_1, \dots, v_{i-1}$ is regular in $\pi_*R$ and $v_i = 0$, then $v_j = 0$ for all $j\geq i$. Here we use crucially that $R$ is a ring spectrum.
\end{remark}

\begin{prop}\label{prop:injective}
 Let $f\colon R \to R'$ be a morphism of homotopy commutative ring spectra with $C_2$-action as above. Assume that the HFPSS for $R'$ is regular and that
 $$\pi_*R/(2,v_1,\dots, v_i) \to \pi_*R'/(2,v_1,\dots, v_i)$$
 is injective for all $i\geq 0$. Moreover assume that $v_k$ is either zero or a non zero-divisor in $\pi_{2*}R/(2,v_1,\dots, v_{k-1})$ for all $k$. Then the HFPSS for $R$ is regular as well.
\end{prop}
\begin{proof}
 If the map induces an injection on all $E_i$-pages for $i\leq n$, then the HFPSS for $R$ is regular up to $E_n$ (which includes the statement about $d_n$-differentials). As the only change occurs when $n$ is a power of $2$, assume that the map of $E_i$-pages is injective for $i\leq 2^{k+1}-1$. We have a commutative square
 \[\xymatrix{ E_{2^{k+1}}(R) \ar[r]\ar[d]^{f_*} & \overline{\pi_{2*}R}\tensor \Z[a,u^{\pm 1}]/(2a, \vb_1a^3,\dots, \vb_ka^{2^{k+1}-1}) \ar[d]^{\overline{f}_*} \\
 E_{2^{k+1}}(R') \ar[r] & \overline{\pi_{2*}R'}\tensor \Z[a,u^{\pm 1}]/(2a, \vb_1a^3,\dots, \vb_ka^{2^{k+1}-1}),
  }
 \]
 where we want to show that $f_*$ is an injection.
 The upper horizontal map is an injection by Lemma \ref{lem:reginjection}. Thus, it suffices to show that the right vertical map is an injection; we can also prove it without the $u^{\pm 1}$-part.

 We grade the ring $\overline{\pi_{2*}R}\tensor \Z[a]$ with respect to the powers of $a$. Consider a homogeneous element $x\in \overline{\pi_{2*}R}\tensor \Z[a]$ such that we can write
 $$f(x) = \lambda_0 \cdot (2a) + \lambda_1 \cdot (\vb_1 a^3) + \cdots \lambda_k \cdot(\vb_ka^{2^{k+1}-1})$$
 with $\lambda_0,\dots, \lambda_k \in \overline{\pi_{2*}R'}\tensor \Z[a]$. Clearly, $|x| \geq 2^{k+1}-1$ and $x = a^{|x|} x'$. Likewise, every $\lambda_i$ is divisible by $a^{|x|-2^{i+1}+1}$. Thus, we can write $f(x') = 2\lambda_0'  + \lambda_1'\vb_1 + \cdots \lambda_k'\vb_k$ with $\lambda_i' \in \overline{\pi_{2*}R}$ and hence $x' \in f^{-1}(2,\vb_1,\dots, \vb_k)$. This is contained in $(2,\vb_1,\dots, \vb_k) \subset \overline{\pi_{2*}R}$ by the injectivity assumption. Thus, $x$ defines $0$ in $\overline{\pi_{2*}R} \tensor \Z[a,u^{\pm 1}]/(2a, \vb_1a^3,\dots, \vb_ka^{2^{k+1}-1})$. Thus, there is no homogeneous element in the kernel of the map $\overline{f}_*$. As every element in the kernel is the sum of homogeneous elements in the kernel, the kernel is zero.
\end{proof}

Recall the following definition from \cite{H-M17}.
\begin{defi}
 A $C_2$-spectrum $R$ is called \emph{strongly even} if $\underline{\pi}_{*\rho -1}R = 0$ and the restriction $\pi_{*\rho}^{C_2}R \to \pi_{2*}R$ is an isomorphism; here $\underline{\pi}_{\bigstar}$ denotes the homotopy Mackey functor and $\rho = \rho_{C_2}$ is still the regular representation.
\end{defi}

\begin{prop}\label{prop:stronglyeven}
 Let $R$ have regular HFPSS and assume that $v_k$ is either zero or a non zero-divisor in $\pi_{2*}R/(2,v_1,\dots, v_{k-1})$ for all $k$. Then $R$ is strongly even.
\end{prop}
\begin{proof}
 Consider a permanent cycle $x = a^lu^m\overline{w} \in \overline{\pi_{2*}R}\tensor \Z[a,u^{\pm 1}]/2a$ with $\overline{w} \in \overline{\pi_{2*}R}$ in degree $*\rho -\varepsilon$ for $\varepsilon =0,1$ and $l>0$ if $\varepsilon = 0$. We have to show that $x$ converges to $0$. Recall that $|a| = 1-\rho$ and $|u| = 4-2\rho$ and thus  $l+4m = -\varepsilon$.

  Let $m = 2^kn$ with $n$ odd. Then
 $$d_{2^{k+2}-1}(x) = a^l\overline{w}d_{2^{k+2}-1}((u^{2^k})^n) = a^{l+2^{k+2}-1}u^{2^k(n-1)}\overline{w}\vb_{k+1}.$$
 This has to be zero on
 $$E_{2^{k+2}-1}(R) \subset \overline{\pi_{2*}R}\tensor \Z[a,u^{\pm 1}]/(2a, a^3\vb_1,\dots, a^{2^{k+1}-1}\vb_k).$$
 Thus, $\overline{w} \in (2,\vb_1,\dots, \vb_k)$.

 On the other hand,
 \[l = -\varepsilon -2^{k+2}n\geq |2^{k+2}n|-\varepsilon \geq 2^{k+2}-\varepsilon.\]
  Already $a^{2^{k+1}-1}$ kills the whole ideal $(2,\vb_1,\dots, \vb_k)$. Thus, $x$ has to be zero on $E_{2^{k+2}-1}(R)$.
\end{proof}

Note that strongly even $C_2$-spectra are, in particular, Real orientable in the sense of Hu and Kriz \cite[Lemma 3.3]{H-M17}.
Moreover we recall that strongly even $C_2$-spectra can be characterized in a number of ways.

\begin{prop}\label{prop:stronglyequivalent}
For a $C_2$-spectrum $R$ the following three conditions are equivalent.
\begin{enumerate}
\item $R$ is strongly even,
\item $\pi_{*\rho-i}^{C_2}R = 0$ for $0<i<3$,
\item The odd slices\footnote{See \cite[Section 4]{HHR} for a comprehensive treatment of slices of equivariant spectra.} $P^{2i+1}_{2i+1}R$ are zero and the even slices $P^{2i}_{2i}R$ are equivalent to $\Sigma^{i\rho} H\underline{\pi_{2i}R}$.
\end{enumerate}
\end{prop}
\begin{proof}
The first two items are equivalent by \cite[Lemma 1.2]{GreReality}. For the first item implying the third see e.g.\ \cite[Section 2.4]{H-M17}. Because $C_2$-spectra as in item $3$ are in particular pure and isotropic, this implies the gap as in item 2 by Theorem 8.3 from \cite{HHR}; note here that if $R$ satisfies item $3$ the same is true for any suspension $\Sigma^{i\rho}R$.
\end{proof}

We end this subsection with a proposition about vanishing lines in regular HFPSS.

\begin{prop}\label{prop:vanishingline}
Assume that $R$ has regular HFPSS up to $E_{2^{n+1}-1}$ (including differentials), that $2,v_1,\dots, v_{n-1}$ forms a regular sequence and that $v_n$ is invertible in the quotient ring $\pi_{2*}/(2,v_1,\dots, v_{n-1})$. Then $E_{2^{n+1}}$ vanishes above the line of height $2^{n+1}-2$ and thus the spectral sequence collapses at $E_{2^{n+1}}$.
\end{prop}
\begin{proof}
We claim the following for $i\leq n$: All $d_{2^{i+1}-1}$-cycles of the form $u^{2^{i-1}m}\overline{w}$ with $m$ odd and $w \in \overline{\pi_{2*}R}$ are $a^{2^i-1}$-torsion on $E_{2^i}$. Indeed, suppose
$$0 = d_{2^{i+1}-1}(u^{2^{i-1}m}\overline{w}) = a^{2^{i+1}-1}\vb_i\overline{w}.$$
By Lemma \ref{lem:reginjection}, $\vb_i\overline{w} \in (2, \vb_1,\dots, \vb_{i-1})\overline{\pi_{2*}R}$. As $\vb_i$ acts as a non-zero divisor on the quotient ring $\overline{\pi_{2*}R}/(2,\vb_1,\dots, \vb_{i-1})$, this implies $\wb \in (2, \vb_1,\dots, \vb_{i-1})$ and the whole ideal is killed by $a^{2^i-1}$ on $E_{2^i}$.

Suppose that $x = a^ku^{2^{i-1}m}\overline{w}$ is a cycle on $E_{2^{n+1}-1}$, where $k\geq 2^{n+1}-1$, the number $m$ is odd, $\overline{w}\in\overline{\pi_{2*}R}$ and $i\geq n$. Automatically, $u^{2^{i-1}m}\overline{w}$ is a cycle as well. We have seen above that $u^{2^{i-1}m}\overline{w}$ is killed by $a^{2^i-1}$ on $E_{2^i}$.

 Thus, it remains to consider the case $x = a^ku^{2^nm}\overline{w}$ with $k$ and $\overline{w}$ as above, but $m$ \emph{even}. Let $v_n'$ be an element with $v_nv_n' \equiv 1 \mod (2,v_1,\dots, v_{n-1})$. Then
 \[d_{2^{n+1}-1}(a^{k-2^{n+1}+1}u^{2^n(m+1)}\wb\vb_n') = a^ku^{2^nm}\wb\vb_n'\vb_n = x\]
 as $(2,\vb_1,\dots, \vb_{n-1})$ is killed by $a^k$.
\end{proof}

\subsection{Examples}\label{sec:examples}
\begin{example}
 The ring spectrum $M\R$ of Real bordism has regular HFPSS (see \cite{H-K01} or the appendix of \cite{G-M16} for a rederivation).
\end{example}

\begin{example}\label{exa:Lubin}
 Let $\Gamma$ be a height-$n$ formal group law of height $n$ over a finite field $k$ of characteristic $2$. By \cite[Section 7]{G-H04} there is a corresponding Lubin--Tate spectrum $E(\Gamma,k)$. Hahn and Shi \cite{HahnShi} showed that $E(\Gamma,k)$ has regular HFPSS;\footnote{Actually, in their original preprint Hahn and Shi only addressed the case of $\Gamma$ being pushed forward from $\F_2$, but they explain in Remark 5.2 of the published version how to deduce the general case using Proposition \ref{prop:injective}.} here, the $C_2$-action is induced by the automorphism $[-1]$ of $\Gamma$. The case $n=1$ of this is classical and the case $n=2$ could have been deduced from \cite{M-R09} or \cite{H-M17} if $\Gamma$ is the formal group law of a supersingular elliptic curve.

 As having regular HFPSS is closed under retracts, it follows that $E(\Gamma,k)^{hG}$ has regular HFPSS as well for $G \subset \mathbb{G}_n$ a subgroup of odd order of the Morava stabilizer group. Moreover, if $H$ is a subgroup of $\Gal(k/\F_2)$ and $G$ is closed under conjugation by $H$, then we claim that $\left(E(\Gamma,k)^{hG}\right)^{hH} \simeq E(\Gamma,k)^{hG\rtimes H}$ has regular HFPSS as well. Indeed, as $G$ acts by $W(k)$-linear maps, $H$ acts $W(k)$-semilinearly on the fixed points $\pi_*E(\Gamma,k)^{hG}$. The functor of taking $H$-fixed points from $W(k)$-modules with semilinear $H$-action to $W(k^H)$-modules is an equivalence by Galois descent and in particular exact. We see that $$\pi_*\left(E(\Gamma,k)^{hG}\right)^{hH} \cong \left(\pi_*E(\Gamma,k)\right)^{G\rtimes H}.$$
Because of the exactness of fixed points, we see that
$$\left(\pi_*E(\Gamma,k)\right)^{G\rtimes H}/(2,v_1,\dots, v_i) = \left(\pi_*E(\Gamma,k)/(2,v_1,\dots, v_i) \right)^{G\rtimes H}$$
injects into $\pi_*E(\Gamma,k)/(2,v_1,\dots, v_i)$. Thus, we can apply Proposition \ref{prop:injective} to see that $E(\Gamma,k)^{hG\rtimes H}$ has regular HFPSS.
\end{example}

Combined with Proposition \ref{prop:injective} this provides quite a powerful tool as one can often write the $K(n)$-localization of $E(n)$-local spectra $R$ as a product of factors of the form $E(k,\Gamma)^{hG}$. If $R$ has a $C_2$-action compatible with this splitting and the orders of $G$ are odd, showing that $R$ has regular HFPSS becomes then a matter of checking the injectivity assumption from Proposition \ref{prop:injective}, which often can be done. We will follow this route for the $C_2$-spectra $\TMF_1(n)$ and deduce consequences for $\Tmf_1(n)$ and $\tmf_1(n)$. Actually, we will work more generally with $\TMF(\Gamma)$, where $\Gamma = \Gamma(n)$ or $\Gamma_1(n)\subset \Gamma \subset \Gamma_0(n)$ with $[\Gamma:\Gamma_1(n)]$ odd.

\begin{example}\label{exa:tmf}
Let $\Gamma\subset SL_2(\Z)$ be a congruence subgroup.
 There are two possible $C_2$-actions on $\TMF(\Gamma)$. Understanding them requires first a study of the functoriality of $\OO^{top}$. If $\XX$ is a stack on a given site, there is a strict $2$-category $\mathrm{Stacks}/\XX$. Its objects are morphisms $\YY \to \XX$ of stacks. Given two such morphisms $f\colon \YY \to \XX$ and $g\colon \YY' \to \XX$ a morphism between them consists of a morphism $h\colon \YY \to \YY'$ and a 2-morphism $H\colon f \Rightarrow gh$. A $2$-morphism between $(h, H)$ and $(h', H')$ consists of a $2$-morphism $K\colon h \to h'$ such that $(gK)H = H'$. This is a specialization of the general theory of slice $\infty$-categories to the context of $2$-categories.

 Constructing the sheaf $\OO^{top}$ on $\MM_{ell}$ in the world of $\infty$-categories (as in \cite{LurEllII}) implies that $\OO^{top}$ defines an $\infty$-functor
 $$(\mathrm{Stacks}/\MM_{ell})^{\acute{e}t, op} \to \CAlg(\Sp)$$
 from the $2$-category of stacks \'etale over $\MM_{ell}$ into the $\infty$-category of $E_\infty$-ring spectra.

 Recall that $\TMF(\Gamma)$ is defined as $\OO^{top}(\MM(\Gamma) \xrightarrow{f} \MM_{ell})$. Let $t\colon \MM_1(n) \to \MM_1(n)$ be the automorphism sending $(E,P)$ to $(E, -P)$. For every $\Gamma_1(n) \subset \Gamma\subset \Gamma_0(n)$ this induces an automorphism $t$ on the quotient stack $\MM(\Gamma)$ as well. On $\MM(n)$, we define $t$ by sending $(E, \alpha)$ to $(E, -\alpha)$, where $\alpha\colon E[n] \to (\Z/n)^2$ is an isomorphism.

 As $ft = f$, the pair $(t, \id_f)$ is an automorphism of $f$ of order $2$ and thus $f \in (\mathrm{Stacks}/\MM_{ell})^{\acute{e}t}$ obtains a $C_2$-action. Another such $C_2$-action is given by $(\id_{\MM(\Gamma)}, [-1])$, where $[-1]$ stands short for the natural transformation sending a pair $(E,P)$ to the automorphism $[-1]\colon E \to E$. In the $2$-category of stacks over $\MM_{ell}$, there is $2$-isomorphism $\phi\colon (t, \id_t) \Rightarrow (\id_{\MM(\Gamma)}, [-1])$. This $\phi$ consists of the $2$-morphism $[-1]\colon t \to \id_{\MM(\Gamma)}$ and actually shows that the two $C_2$-actions on $f$ are equivalent (even isomorphic) in the $2$-category of objects with $C_2$-action in $(\mathrm{Stacks}/\MM_{ell})^{\acute{e}t}$. As $\OO^{top}$ preserves equivalences (as every $\infty$-functor does), the actions $(t,\id_f)^*$ and $(\id_f, [-1])^*$
on $\TMF_1(n)$ become equivalent in $\Fun(BC_2, \CAlg(\Sp))$. In particular, the action on $\TMF_1(n)$ is trivial for $n\leq 2$.

In the following, we fix the prime $2$ and will always assume that the level of $\Gamma$ is odd. The construction of $L_{K(2)}\OO^{top}$ in Section 12.4 of \cite{TMF} implies that $L_{K(2)}\TMF(\Gamma)$ is equivalent to $\prod E(k_i, \Gamma_i)^{hG_i}$ if the fiber of $\MM(\Gamma) \to \MM_{ell}$ over the supersingular locus is of the form $\coprod \Spec k_i/G_i$ and $\Gamma_i$ is the corresponding formal group of the supersingular elliptic curve over $k_i$. The stack $\MM(n)$ is representable for $n\geq 3$ and the stack $\MM_1(n)$ is representable for $n\geq 4$. Thus the groups $G_i$ can be chosen to be trivial in these cases. In the case of $\Gamma_1(3)$, we claim that the fiber of $\MM_1(3) \to \MM_{ell}$ over the supersingular locus is equivalent to $\Spec \F_4/(C_3\rtimes C_2)$. Here, we use the map $\Spec \F_4/(C_3\rtimes C_2) \to \MM_1(3)$ corresponding to the supersingular elliptic curve
 $$C: y^2+y = x^3$$
 with $3$-torsion point $(0,0)$ and automorphisms given by $(x,y) \mapsto (x,\zeta_3y)$ and the Galois action of $C_2$ on $\F_4$. Note that $C_3 \subset \Aut_{\Fb_2}(C)$ is exactly the stabilizer of $(0,0)$. Moreover, the action of $\Aut_{\Fb_2}(C)$ on $E[3]_{\Fb_2}$ identifies it with $SL_2(\F_3)$ (see e.g.\ the proof of Proposition 3.3 in \cite{MeiDecMod}) and thus $\Aut_{\Fb_2}(C)$ acts transitively on all $\Gamma_1(3)$-level structures on $C$. If $\Gamma_1(n) \subset \Gamma \subset \Gamma_0(n)$ for $n\geq 4$, the groups $G_i$ are isomorphic to subgroups of $\Gamma/\Gamma_1(n)$.

The identifications of the $K(2)$-localizations are compatible with the $C_2$-action on the Lubin--Tate spectra (given by $[-1]$) if we choose the action by $(\id_f, [-1])^*$ on $\TMF(\Gamma)$. By Example \ref{exa:Lubin}, the homotopy fixed point spectral sequence for $L_{K(2)}\TMF(\Gamma)$ is thus regular if $n\geq 3$ and $\Gamma = \Gamma(n)$ or $\Gamma_1(n)\subset \Gamma \subset \Gamma_0(n)$ with $[\Gamma:\Gamma_1(n)]$ odd. These are precisely those congruence subgroups $\Gamma$ that are tame for $l=2$. This will be our assumption on $\Gamma$ from now on.

Note that as $\TMF(\Gamma)$ is complex oriented and $E(2)$-local,
 \[\pi_*L_{K(2)}\TMF(\Gamma) \cong \widehat{\pi_*\TMF(\Gamma)}_{(2,v_1)}\]
 by \cite[Proposition 7.10]{H-S99}.

 The rings $H^0(\MMb(\Gamma); \omega^{\tensor *})$ and $H^0(\MMb(\Gamma); \omega^{\tensor *})/2\subset H^0(\MMb(\Gamma)_{\F_2}; \omega^{\tensor *})$ are integral domains by (the proof of) \cite[Proposition 2.13]{MeiDecMod}. Inverting $\Delta$ yields that $\pi_*\TMF(\Gamma)$ and $\pi_*\TMF(\Gamma)/2$ are integral domains. In general, the Krull intersection theorem implies that the map from a noetherian integral domain into any of its completion is an injection. We obtain that
 \[\pi_*\TMF(\Gamma) \to \pi_*L_{K(2)}\TMF(\Gamma) \quad\text{ and }\quad \pi_*\TMF(\Gamma)/2 \to \pi_*L_{K(2)}\TMF(\Gamma)/2\]
 are injections. Moreover,
 \[\pi_*\TMF(\Gamma)/(2,v_1) \cong \pi_*L_{K(2)}\TMF(\Gamma)/(2,v_1)\quad\text{ and }\quad\pi_*\TMF(\Gamma)/(2,v_1,v_2) = 0\]
  as the common vanishing locus of $2,v_1$ and $v_2$ is empty because there is no elliptic curve of height higher than $2$. Thus, we can apply Proposition \ref{prop:injective} and Example \ref{exa:Lubin} to conclude that the homotopy fixed point spectral sequence for the natural $C_2$-action on $\TMF(\Gamma)$ is regular. Proposition \ref{prop:stronglyeven} implies that the $C_2$-spectra $\TMF_1(n)$ are strongly even.

We remark that the derivation of the differentials in the homotopy fixed point spectral sequence for $\TMF_1(3)$ here is philosophically different from the one in \cite{H-M17}. In \cite{H-M17}, we used basic knowledge of $\pi_*\mathbb{S}$ (namely the behavior of the Hopf maps $\eta$ and $\nu$) and the map $\mathbb{S} \to \TMF_1(3)$. In contrast, here we are mapping $\tmf_1(3)$ into something we understand better, namely the Lubin--Tate spectrum. The differentials for the Lubin--Tate spectra are deduced from $M\R$, which does not involve any knowledge about $\pi_*\mathbb{S}$.
\end{example}

\begin{example}\label{exa:Tmf}
We can use the preceding example to understand $\Tmf(\Gamma)$ as a $C_2$-spectrum as well, where we consider again a congruence subgroup $\Gamma$ that is tame for $l=2$.

We have a fiber square
\begin{align}\label{eq:Tmf(Gamma)}
\xymatrix{
\Tmf(\Gamma) \ar[r] \ar[d] & \TMF(\Gamma)\ar[d] \\
\Tmf(\Gamma)[c_4^{-1}] \ar[r] &\TMF(\Gamma)[c_4^{-1}]
}
\end{align}
that is compatible with the $(\Z/n)^\times$-actions and in particular the $C_2$-actions; this follows as $D(c_4)$ and $D(\Delta)$ cover $\MMb(\Gamma)$ and $\OO^{top}$ is a sheaf.

We claim that $c_4 \equiv v_1^4 \mod 2$ in the base ring of any elliptic curve in Weierstrass form. Here we use Hazewinkel's $v_1$, which is twice the $x^2$-coefficient in the logarithm of the corresponding formal group. A computation shows that this coefficient is $\frac12 a_1$ and moreover $a_1^4 \equiv c_4 \mod 2$ (see the formulae in \cite[Section III.3]{Sil09}).

As the HFPSS of $\TMF(\Gamma)$ is regular, the same is true for that of $\TMF(\Gamma)[c_4^{-1}]$. Indeed, it is clear that the HFPSS of the latter is regular up to $E_3$ (including differentials). By Proposition \ref{prop:vanishingline}, the HFPSS of $\TMF(\Gamma)[c_4^{-1}]$ collapses at $E_4$.

Note that
$$\pi_*\TMF(\Gamma)[c_4^{-1}] \cong \pi_*\Tmf(\Gamma)[c_4^{-1}][\Delta^{-1}]$$
is concentrated in even degrees and
$$\pi_{2*}\Tmf(\Gamma)[c_4^{-1}] \cong H^0(\MMb(\Gamma); \omega^{\tensor *})[c_4^{-1}].$$
Here, we use that all the elements of $\pi_*\Tmf(\Gamma)$ corresponding to elements in $H^1(\MMb(\Gamma); \omega^{\tensor *})$ are killed by powers of $c_4$ as $H^1(\MMb(\Gamma); \omega^{\tensor i})= 0$ for $i>1$ by Proposition \ref{prop:coh}. As noted above, $H^0(\MMb(\Gamma); \omega^{\tensor *})$ and $H^0(\MMb(\Gamma); \omega^{\tensor *})/2$ are integral domains. Thus
\[\pi_*\Tmf(\Gamma)[c_4^{-1}] \to \pi_*\TMF(\Gamma)[c_4^{-1}] \quad\text{ and }\quad \pi_*\Tmf(\Gamma)[c_4^{-1}]/2 \to \pi_*\TMF(\Gamma)[c_4^{-1}]/2 \]
are injections.  As $c_4$ is congruent to $v_1^4$ modulo $2$,
$$\pi_*\Tmf(\Gamma)[c_4^{-1}]/(2,v_1) =0.$$
Our criterion Proposition \ref{prop:injective} implies that $\Tmf(\Gamma)[c_4^{-1}]$ has regular homotopy fixed point spectral sequence. In particular, all the $C_2$-spectra in the fiber square \eqref{eq:Tmf(Gamma)} except for $\Tmf(\Gamma)$ are strongly even.

Let $K_k$ be the kernel and $R_k$ be the cokernel of the map
$$\pi_{2k}\TMF(\Gamma) \oplus \pi_{2k}\Tmf(\Gamma)[c_4^{-1}] \to \pi_{2k}\TMF(\Gamma)[c_4^{-1}].$$
As the Mayer--Vietoris sequence for cohomology shows, we have isomorphisms
\[K_k \cong H^0(\MMb(\Gamma); \omega^{\tensor k}) \quad\text{ and }\quad R_k \cong H^1(\MMb(\Gamma); \omega^{\tensor k}).\]
Here, we use that the non-vanishing loci of $c_4$, $\Delta$ and $c_4\Delta$ are of the form $\Spec A/\G_m$ and thus have no higher cohomology.

If $R$ is strongly even, we have an isomorphism $\underline{\pi}_{k\rho+1}(R) \cong G \tensor_{\Z} \pi_{2k+2}R$ of Mackey functors, where $G$ denotes the Mackey functor with $G(C_2/C_2) = \Z/2$ and $G(C_2/e) = 0$ \cite[Lemma 2.15]{H-M17}. Thus, we have a short exact sequence
\begin{align}\label{eq:GK}0 \to G \tensor R_{k+1} \to \underline{\pi}_{k\rho}\Tmf(\Gamma) \to \underline{K_k} \to 0\end{align}
and an isomorphism $\underline{\pi}_{k\rho-1}\Tmf(\Gamma) \cong \underline{R_k}$ of Mackey functors. The image of $\Z/2 \tensor R_{k+1}$ in $\pi_{k\rho}^{C_2}\Tmf(\Gamma)$ is detected in the HFPSS in the group
$$H^1(C_2, \pi_{k\rho+1}\Tmf(\Gamma)) = H^1(C_2, R_{k+1}),$$
where the action on $R_{k+1}$ is by sign. Indeed, $\pi_{k\rho +1}^{C_2}\TMF(\Gamma)[c_4^{-1}]$ is spanned by classes of the form $a\overline{x}$ with $|\overline{x}| = (k+1)\rho$ as the HFPSS of $\TMF(\Gamma)[c_4^{-1}]$ is regular and all multiples of $a^3$ are sources or targets of $d_3$-differentials. These classes are detected in the first line of the HFPSS and map non-trivally to the HFPSS for $\Tmf(\Gamma)$ if the corresponding class $x \in \pi_{2k+2}\TMF(\Gamma)[c_4^{-1}]$ is nonzero in $R_{k+1}\tensor \Z/2$. The action on $\pi_{k\rho+1}\Tmf(\Gamma)$ is by sign as $C_2$ acts by $(-1)^{k+1}$ on $H^1(\MMb(\Gamma); \omega^{\tensor (k+1)})$, as $k\rho +1$ differs from $2(k+1)-1$ by $k(1-\sigma)$ and $(-1)^k(-1)^{k+1} = -1$.

Note that either kernel or cokernel in the sequence \eqref{eq:GK} is zero unless $k=0$ by Proposition \ref{prop:coh}. We claim that the sequence splits for $k=0$ as Mackey functors if $\Gamma_1(n)\subset \Gamma\subset \Gamma_0(n)$. Note that $\pi_0 \Tmf(\Gamma) = \Z[\tfrac1n]$ and $\pi_0^{C_2}\Tmf(\Gamma) \cong \Z[\tfrac1n] \oplus (\Z/2 \tensor R_1)$. We can thus write $\tr(1) = 2+g$ with $2\in\Z[\tfrac1n]$ and $g\in (\Z/2 \tensor R_1)$. Hence, \eqref{eq:GK} splits if $g = 0$ in this decomposition. The homotopy groups of the spheres spectrum $\mathbb{S}$ map trivially to $\pi_*\Tmf(\Gamma)$ in positive degrees as $\eta$ is zero in $\pi_1\Tmf(\Gamma)$ by \cite[Proposition 2.16]{MeiDecMod} and $\pi_i\Tmf(\Gamma)$ is torsionfree for $i>1$ by Proposition \ref{prop:coh}. We have seen above that every non-trivial element in $\pi_0^{C_2}\Tmf(\Gamma)$ is detected in the zero- or one-line of the HFPSS. Thus, every element in $\pi_0 \mathbb{S}^{hC_2}$ that is not detected in the zero-line maps to zero in $\pi_0^{C_2}\Tmf(\Gamma)$. Indeed, an element of $\pi_0 \mathbb{S}^{hC_2}$ of filtration $\geq 1$ must map to filtration $\geq 2$ in $\pi_0^{C_2}\Tmf(\Gamma)$ and is thus zero. In particular, this is true for $\tr(1) -2$ (if $1$ is viewed as an element of $\pi_0\mathbb{S}$) and thus $g$ is zero.

As summarized in \cite[Section 2.4]{H-M17}, this information is for all $\Gamma$ that are tame for $l=2$ enough to calculate all slices and we get
\[P^{2k}_{2k}\Tmf(\Gamma) \simeq \Sigma^{k\rho}H\underline{K_k} \quad\text{ and }\quad P^{2k-1}_{2k-1}\Tmf(\Gamma) \simeq \Sigma^{k\rho-1}H\underline{R_k}.\]
In particular, the connective cover $\tau_{\geq 0}\Tmf(\Gamma)$ is strongly even if $H^1(\MMb(\Gamma); \omega) =0$ by Proposition \ref{prop:stronglyequivalent} and because connective cover and slice connective cover agree by \cite[Proposition 4.11]{HHR}. If $\Gamma = \Gamma_1(n)$, this holds for all $n<23$ by \cite[Lemma 2]{Buz}. In these cases, $\tmf(\Gamma) = \tau_{\geq 0}\Tmf(\Gamma)$ seems to be the appropriate connective version of $\Tmf(\Gamma)$.
\end{example}

We summarize the main results of the previous two examples in the following theorem.

\begin{thm}
Let $\Gamma \subset SL_2(\Z)$ be a congruence subgroup that is tame for $l=2$. Then $\TMF(\Gamma)$ has regular HFPSS for its natural $C_2$-action and is strongly even. The slices of $\Tmf(\Gamma)$ can be computed as
\[P^{2k}_{2k}\Tmf(\Gamma) \simeq \Sigma^{k\rho}H\underline{H^0(\MMb(\Gamma);\omega^{\tensor k})} \quad\text{ and }\quad P^{2k-1}_{2k-1}\Tmf(\Gamma) \simeq \Sigma^{k\rho-1}H\underline{H^1(\MMb(\Gamma);\omega^{\tensor k})}.\]
\end{thm}

\begin{remark}
 We can also apply the same methods to the kind of topological automorphic forms considered in \cite{H-L10b} and \cite{Law15}; in particular, it seems that we can recover the calculations of \cite[Section 12]{Law15}, at least after forcing the spectrum to be periodic.

 If we want to apply the same style of argument to topological automorphic forms of higher heights $n$, the crucial point to check is that for the given Shimura variety $\XX$, each locus $\XX^{[h]}$ of points of height $\geq h$ contains in each irreducible component a point of height $n$: in \cite[Chapter 14]{B-L10} it is shown that completing at these points gives a map into a product of homotopy fixed points of Lubin--Tate spectra; if enough level is present these fit into the scope of Example \ref{exa:Lubin} and the assumption about irreducible components allows to apply Proposition \ref{prop:injective}. It is not clear to the author, how generally this assumption holds.
\end{remark}

\subsection{$C_2$-equivariant decompositions}
Let $\Gamma$ throughout this section be a congruence subgroup that is tame for $l=2$ and which satisfies $\Gamma_1(n) \subset \Gamma \subset \Gamma_0(n)$. The goal of this section is to prove the following two theorems, answering a question by Mike Hill.
\begin{thm}\label{thm:C2}
 There is a $C_2$-equivariant decomposition of $\TMF(\Gamma)_{(2)}$ into copies of $\Sigma^{?\rho}\TMF_1(3)_{(2)}$ if we have this decomposition non-equivariantly. The same holds after $2$-completion so that we obtain for all $n$ odd
 $$\widehat{\TMF(\Gamma)}_2 \simeq_{C_2} \bigoplus \Sigma^{?\rho}\widehat{\TMF_1(3)}_{2}.$$
\end{thm}
\begin{thm}\label{thm:C2cpt}
 There is a $C_2$-equivariant decomposition of $\Tmf(\Gamma)_{(2)}$ into copies of the $C_2$-spectra $\Sigma^{?\rho}\Tmf_1(3)_{(2)}$ if we have this decomposition non-equivariantly, i.e.\ if and only if $\pi_1\Tmf(\Gamma)$ has no $2$-torsion.
\end{thm}
In the following, we will always implicitly $2$-localize. Let us now explain how to prove the second theorem if the level $n$ of $\Gamma$ is divisible by $3$ and the first will be analogous in this case. We claim that the map $\MMb(\Gamma) \to \MMb_{ell}$ factors over a map $h\colon \MMb(\Gamma) \to \MMb_1(3)$ and this map is finite and flat. By the universal property of normalization (see \cite[Tag 035I]{STACKS}, which is easily adapated to stacks), it suffices to construct such a morphism $\MM(\Gamma) \to \MM_1(3)$ on uncompactified moduli. If $n=3k$, multiplication by $k$ defines a morphism $h\colon \MM_1(3k) \to \MM_1(3)$ and this factors over the stack quotient $\MM_1(3k)/G$ for $G \subset (\Z/n)^\times$ a group projecting trivially to $(\Z/3)^\times$ (e.g.\ if $|G|$ is odd). The morphism $h$ is finite because source and target are finite over $\MMb_{ell}$ and it is flat because source and target are smooth over $\Z[\tfrac1n]$ (arguing as in \cite[Proposition 2.4]{MeiDecMod}).

We see that $h_*\OO_{\MMb(\Gamma)}$ is a vector bundle. As by our assumptions, the first cohomology of $h_*\OO_{\MMb(\Gamma)}\tensor (f_3)^*\omega^{\tensor k}$ is torsionfree for all $k$ and $\MMb_1(3)$ is a weighted projective line, \cite[Theorem A.8]{MeiDecMod} implies that $h_*\OO_{\MMb(\Gamma)}$ splits into a sum of $(f_3)^*\omega^{\tensor ?}$; in the uncompactified case (corresponding to Theorem \ref{thm:C2}) one can use \cite[Theorem 3.9]{Mei13} instead. An argument analogously to Theorem \ref{thm:dectopcompact} implies that $\Tmf(\Gamma)$ splits as a $\Tmf_1(3)$-module into copies of $\Sigma^{2?}\Tmf_1(3)$. Choosing such an equivalence provides an element in
\begin{align*}
\pi_0\Hom_{\Tmf_1(3)}(\bigoplus_i \Sigma^{2k_i}\Tmf_1(3), \Tmf(\Gamma)) &\cong \bigoplus_i \pi_{2k_i}\Hom_{\Tmf_1(3)}(\Tmf_1(3), \Tmf(\Gamma))\\
 &\cong \bigoplus_i \pi_{2k_i} \Tmf(\Gamma).\end{align*}
By \eqref{eq:GK}, the restrictions $\pi_{k_i\rho}^{C_2}\Tmf(\Gamma) \to \pi_{2k_i}\Tmf(\Gamma)$ are surjective. Thus, we can lift this element to an element in
$$\pi_0^{C_2}\Hom_{\Tmf_1(3)}(\bigoplus_i \Sigma^{k_i\rho}\Tmf_1(3), \Tmf(\Gamma))  \cong \bigoplus_i \pi^{C_2}_{k_i\rho} \Tmf(\Gamma).$$
Automatically, the corresponding $C_2$-equivariant map
\[\bigoplus_i \Sigma^{k_i\rho}\Tmf_1(3) \to \Tmf(\Gamma)\]
 is an (underlying) equivalence.

Our goal is to follow an analogous (but more involved) strategy if $3$ does not divide $n$, which we will assume from now on. A key step is the following proposition.

\begin{prop}\label{prop:functionstronglyeven}
If $n$ is not divisible by $3$, the function spectrum
$$\Hom_{\TMF}(\TMF_1(3), \TMF(\Gamma))$$
 with $C_2$-action by conjugation has regular HFPSS and is thus strongly even.
\end{prop}

We can deduce Theorem \ref{thm:C2} from this proposition in a manner very similar to the above. Indeed, assume that we have a $\TMF$-linear equivalence
\begin{align}\label{eq:splittingper}\bigoplus_{i} \Sigma^{2k_i}\TMF_1(3) \to \TMF(\Gamma).\end{align}
This defines an element in
$$\pi_0\Hom_{\TMF}(\bigoplus_i \Sigma^{2k_i}\TMF_1(3), \TMF(\Gamma)) \cong \bigoplus_i \pi_{2k_i}\Hom_{\TMF}(\TMF_1(3), \TMF(\Gamma)).$$
By Proposition \ref{prop:functionstronglyeven}, we can lift this to an element in
$$\pi_0^{C_2}\Hom_{\TMF}(\bigoplus_i \Sigma^{k_i\rho}\TMF_1(3), \TMF(\Gamma)) \cong \bigoplus_i \pi_{k_i\rho}^{C_2}\Hom_{\TMF}(\TMF_1(3), \TMF(\Gamma)).$$
The corresponding $C_2$-equivariant map
$$\bigoplus_{i} \Sigma^{k_i\rho}\TMF_1(3) \to \TMF(\Gamma),$$
is an (underlying) equivalence. The same proof works $2$-completely using the following lemma.
\begin{lemma}
The $p$-completion of a strongly even $C_2$-spectrum $X$ is strongly even again if the underlying homotopy of $X$ has no infinitely $p$-divisible elements.
\end{lemma}
\begin{proof}
Recall that the $p$-completion is the homotopy limit over $X/p^k$ and that we have short exact sequences
$$0 \to (\pi_m X)/p^k \to \pi_m(X/p^k) \to \Hom(\Z/p^k, \pi_{m-1}X) \to 0$$
and similarly for the $C_2$-equivariant groups. As $\pi_*X$ has no infinitely $p$-divisible elements, $\lim_k \Hom(\Z/p^k, \pi_*X) =0$ and thus $\pi_m$ of the $p$-completion of $X$ is an extension of $\widehat{(\pi_mX)}_p$ and $\lim_k^1\Hom(\Z/p^k, \pi_mX)$. In particular, the $p$-completion of $X$ is underlying even again.

Now we can use that $\pi_{*\rho-i}^{C_2}X = 0$ for $i=1,2$ by Proposition \ref{prop:stronglyequivalent}.
\end{proof}
Finally note that the $2$-complete version of the splitting \eqref{eq:splittingper} is always possible by Theorem \ref{thm:dectopuncompact}. This reduces the proof of Theorem \ref{thm:C2} to Proposition \ref{prop:functionstronglyeven}.

For the proof of Proposition \ref{prop:functionstronglyeven}, we need two lemmas.

\begin{lemma}\label{lem:SWdual}
 There are $C_2$-equivariant equivalences
$$ \Hom_{\Tmf}(\Tmf_1(3), \Tmf(\Gamma))\simeq \Sigma^{-16+2\rho}\Tmf_1(3)\sm_{\Tmf}\Tmf(\Gamma)$$
  and
 $$\Hom_{\TMF}(\TMF_1(3), \TMF(\Gamma))\simeq \Sigma^{-6\rho}\TMF_1(3)\sm_{\TMF}\TMF(\Gamma).$$
\end{lemma}
\begin{proof}
 As $\Tmf_1(3)$ is a finite $\Tmf$-module by Proposition \ref{prop:compactness}, we have
 $$\Hom_{\Tmf}(\Tmf_1(3), \Tmf(\Gamma)) \simeq \Hom_{\Tmf}(\Tmf_1(3), \Tmf) \sm_{\Tmf} \Tmf(\Gamma).$$
 Moreover, we have the following chain of $C_2$-equivariant equivalences.
 \begin{align*}
  \Hom_{\Tmf}(\Tmf_1(3), \Tmf) &\simeq \Hom_{\Tmf}(\Tmf_1(3), \Sigma^{-21}\Z^{\Tmf}) \\
  &\simeq \Sigma^{-21}\Z^{\Tmf_1(3)}\\
  &\simeq \Sigma^{-21} \Sigma^{5+2\rho}\Tmf_1(3) \\
  &\simeq \Sigma^{-16+2\rho}\Tmf_1(3)
 \end{align*}
 Here, we use the computations of the Anderson duals $\Z^{\Tmf}$ by Stojanoska (Theorem \ref{thm:Stojanoska}) and $\Z^{\Tmf_1(3)}$ (the latter $C_2$-equivariant and proven in \cite{H-M17}). Thus,
 $$\Hom_{\Tmf}(\Tmf_1(3), \Tmf(\Gamma)) \simeq \Sigma^{-16+2\rho}\Tmf_1(3)\sm_{\Tmf} \Tmf(\Gamma).$$
 This implies the first claim.
 The $C_2$-spectrum $\TMF_1(3)$ is $(16-8\rho)$-periodic as the class $u^4$ is a permanent cycle in the $RO(C_2)$-graded homotopy fixed point spectral sequence: indeed, the HFPSS for $\TMF_1(3)$ is regular by Example \ref{exa:tmf} and $v_2$ is invertible (see also \cite[Section 4.2]{H-M17} for a more detailed treatment). Thus, we obtain the second claim.
\end{proof}

\begin{lemma}\label{lem:pullback}
 If $n$ is not divisible by $3$, then $\TMF_1(3)\sm_{\TMF}\TMF(\Gamma)$ is $C_2$-equivariantly equivalent to $\TMF(\Gamma')$, where $\Gamma'$ is the preimage of
 $$\Gamma/\Gamma_1(n) \subset (\Z/n)^\times\subset (\Z/3n)^\times$$
  along the homomorphism
 $$\Gamma_0(3n) \xrightarrow{\det} (\Z/3n)^\times.$$
\end{lemma}
\begin{proof}
 Consider the pullback square
 \[
 \xymatrix{ \MM_1(3n) \ar[r] \ar[d] & \MM_1(n) \ar[d] \\
 \MM_1(3) \ar[r] & \MM_{ell}
 }\]
 Taking a stack quotient by $\Gamma/\Gamma_1(n)$ in the upper row gives a pullback square
  \[
 \xymatrix{ \MM(\Gamma') \ar[r]^{f'} \ar[d]^{g'} & \MM(\Gamma) \ar[d]^{g} \\
 \MM_1(3) \ar[r]^f & \MM_{ell}
 }\]
 Because the equivalence in Proposition \ref{prop:0affine} is symmetric monoidal, we know that
 \begin{align*}
   \TMF_1(3)\sm_{\TMF}\TMF(\Gamma) &\simeq \Gamma(f_*\OO^{top}_{\MM_1(3)})\sm_{\Gamma(\OO^{top})} \Gamma(g_*\OO^{top}_{\MM(\Gamma)}) \\
   &\simeq \Gamma(f_*\OO^{top}_{\MM_1(3)}\sm_{\OO^{top}} g_*\OO^{top}_{\MM(\Gamma)})
  \end{align*}
 By the projection formula and base change (the algebraic form implies the topological one)
   \begin{align*}
   f_*\OO^{top}_{\MM_1(3)}\sm_{\OO^{top}} g_*\OO^{top}_{\MM(\Gamma)} &\simeq f_*f^*g_*\OO^{top}_{\MM(\Gamma)}\\
  &\simeq f_*(g')_*(f')^*\OO^{top}_{\MM(\Gamma)} \\
   &\simeq (g'f)_* \OO^{top}_{\MM(\Gamma')}.
  \end{align*}
 The global sections of this sheaf are exactly $\TMF(\Gamma')$.
\end{proof}

\begin{proof}[Proof of Proposition \ref{prop:functionstronglyeven}]
As $\TMF(\Gamma')$ (for $\Gamma'$ as in the last lemma) has regular HFPSS and is strongly even by Example \ref{exa:tmf}, we get in the case of $3$ not dividing $n$ that
$$\Hom_{\TMF}(\TMF_1(3), \TMF(\Gamma)) \simeq \Sigma^{-6\rho}\TMF(\Gamma')$$
with $\Gamma'$ as in Lemma \ref{lem:pullback} is strongly even as well.
\end{proof}

While this settles the periodic case Theorem \ref{thm:C2}, we still need to prove Theorem \ref{thm:C2cpt}. To that purpose we will analyse the $C_2$-spectrum $H = \Hom_{\Tmf}(\Tmf_1(3), \Tmf(\Gamma))$, where we again act by conjugation. Recall that Lemma \ref{lem:SWdual} implies a $C_2$-equivalence
$$H \simeq \Sigma^{-16+2\rho}\Tmf_1(3)\sm_{\Tmf}\Tmf(\Gamma).$$
\begin{cor}
The HFPSS for $H[(c_4\Delta)^{-1}]$ and $H[c_4^{-1}]$ are regular and hence these $C_2$-spectra are strongly even.
\end{cor}
\begin{proof}
For $H[(c_4\Delta)^{-1}]$, this follows from Proposition \ref{prop:functionstronglyeven} analogously to the argument in the fourth paragraph of Example \ref{exa:Tmf}. For $H[c_4^{-1}]$ we want to apply the criterion of Proposition \ref{prop:injective}. Denote by $g$ the projection $\MMb(\Gamma) \to \MMb_{ell}$. As the equivalence in Proposition \ref{prop:0affine} is monoidal and by the projection formula, we obtain equivalences
\begin{align*}
\Gamma((f_3)^*g_*\OO^{top}_{\MMb(\Gamma)}) &\simeq
\Gamma((f_3)_*\OO^{top}_{\MMb_1(3)} \sm_{\OO^{top}}g_*\OO^{top}_{\MMb(\Gamma)}) \\
&\simeq \Tmf_1(3) \sm_{\Tmf} \Tmf(\Gamma).
\end{align*}
The non-vanishing loci of $c_4$ and $c_4\Delta$ on $\MMb_1(3)$ are of the form $\Spec A/\G_m$ and thus have no higher cohomology. Hence, we obtain that on underlying homotopy groups the map $H[c_4^{-1}] \to H[(c_4\Delta)^{-1}]$ is (up to shift) isomorphic to
$$H^0(\MMb_1(3); (f_3)^*g_*\omega^{\tensor *})[c_4^{-1}] \to H^0(\MMb_1(3); (f_3)^*g_*\omega^{\tensor *})[(c_4\Delta)^{-1}].$$
As the non-vanishing loci of $c_4$ and $\Delta$ are dense in $\MMb_1(3)$ and $\MMb_1(3)$ is normal (so that all local rings are integral domains), we know that the map
$$\FF\tensor_{\OO_{\MMb_1(3)}} (f_3)^*\omega^{\tensor *}[c_4^{-1}] \to \FF\tensor_{\OO_{\MMb_1(3)}} (f_3)^*\omega^{\tensor *}[(c_4\Delta)^{-1}]$$
must be an injection of sheaves for every vector bundle $\FF$ on $\MMb_1(3)$; in particular, this is true for $(f_3)^*g_*\OO_{\MMb(\Gamma)}$. Thus, $\pi_*H[c_4^{-1}] \to\pi_* H[(c_4\Delta)^{-1}]$ is an injection and the same is true after quotienting by $2$ as we can repeat the above argument after base change to $\F_2$. Lastly, $\pi_*H[c_4^{-1}]/(2,v_1)$ vanishes as $c_4$ is congruent to $v_1^4$ modulo $2$ as noted in Example \ref{exa:Tmf}. Thus, Proposition \ref{prop:injective} applies and the HFPSS for $H[c_4^{-1}]$ is regular.
\end{proof}

Let $K_k$ be the kernel and $R_k$ be the cokernel of the map
$$\pi_{2k}H[\Delta^{-1}] \oplus \pi_{2k}H[c_4^{-1}] \to \pi_{2k}H[(c_4\Delta)^{-1}].$$
As in Example \ref{exa:Tmf}, we obtain a short exact sequence
\begin{align*}0 \to G \tensor R_{k+1} \to \underline{\pi}_{k\rho}H \to \underline{K_k} \to 0\end{align*}
By the snake lemma one sees that the restriction map $\pi^{C_2}_{k\rho}H \to \pi_{2k}H$ is surjective. From this fact one argues as before to show Theorem \ref{thm:C2cpt}.

\begin{remark}\label{rem:AndersonC2}
 Theorem \ref{thm:C2} has implications for our previous applications of our decomposition results. If localized at $2$, Theorem \ref{thm:Gsplitting} becomes $C_2$-equivariant if we consider the action of $C_2 \subset \Aut(G)$ on $\TMF^G$, where $C_2$ acts by inversion.

Moreover, we can refine Theorem \ref{thm:Duality} $2$-locally to provide a $C_2$-equivariant computation of Anderson duals $I_{\Z[\tfrac1n]}\Tmf_1(n)$ in the listed cases. This is easily made explicit using the equivalence $I_{\Z[\tfrac13]}\Tmf_1(3) \simeq_{C_2}\Sigma^{5+2\rho}\Tmf_1(3)$ from \cite{H-M17} and the third table in \cite[Appendix C]{MeiDecMod}. For the odd values of $n$ in Theorem \ref{thm:Duality} we obtain
\begin{align*}
I_{\Z_{(2)}} \Tmf_1(5)_{(2)} &\simeq_{C_2} \Sigma^5\Tmf_1(5)_{(2)}\\
I_{\Z_{(2)}} \Tmf_1(7)_{(2)} &\simeq_{C_2} \Sigma^{5-\rho}\Tmf_1(7)_{(2)}\\
I_{\Z_{(2)}} \Tmf_1(11)_{(2)} &\simeq_{C_2} \Sigma^{5-2\rho}\Tmf_1(11)_{(2)}\\
I_{\Z_{(2)}} \Tmf_1(15)_{(2)} &\simeq_{C_2} \Sigma^{5-2\rho}\Tmf_1(15)_{(2)}\\
I_{\Z_{(2)}} \Tmf_1(23)_{(2)} &\simeq_{C_2} \Sigma^{5-3\rho}\Tmf_1(23)_{(2)}
\end{align*}

\end{remark}

We end with a corollary from the two theorems of this subsection.

\begin{cor}\label{cor:0n}
Assume that $n\geq 3$ is odd and $\varphi(n)$ is not divisible by $4$. Then $\TMF_0(n)$ decomposes as a $\TMF$-module into modules of the form $(\Sigma^{k\rho}\TMF_1(3))^{hC_2}$ after $2$-completion.

Assume further that $H^1(\MMb(\Gamma); g^*\omega)$
has no $2$-torsion. Then $\Tmf_0(n)_{(2)}$ decomposes as a $\Tmf$-module into modules of the form $(\Sigma^{k\rho}\TMF_1(3)_{(2)})^{hC_2}$.
\end{cor}
\begin{proof}
	Consider $\Gamma_1(n) \subset \Gamma \subset \Gamma_0(n)$ with the second inclusion being of index $2$. Then $\TMF_0(n) \simeq \TMF(\Gamma)^{hC_2}$ and similarly for $\TMF_0(n)$. Thus our claims follows directly from the two theorems of this subsection if we apply $C_2$-homotopy fixed points to $\widehat{\TMF(\Gamma)}_2$ and $\Tmf(\Gamma)$ respectively as $\Gamma$ is tame.
\end{proof}
Let $\Gamma$ be as in the last proof. As $\Gamma$ is tame and using  Lemma \ref{lem:compactificationsGamma} if $n$ is not squarefree, we can identify $H^1(\MMb(\Gamma); g^*\omega)$  with $H^1(\MMb_1(n); g^*\omega)^{\Gamma/\Gamma_1(n)}$. Moreover, $[-1]\in (\Z/n)^\times$ acts by $-1$ on this cohomology group and this class represents the non-trivial element of $\Gamma_0(n)/\Gamma$. Thus, the $2$-torsion of $H^1(\MMb_1(n); g^*\omega)^{\Gamma/\Gamma_1(n)}$ agrees with
$$H^1(\MMb_1(n); g^*\omega)^{(\Z/n)^\times} \cong (\pi_1\Tmf_1(n))^{(\Z/n)^\times}.$$

\appendix
\section{Complements on the moduli stacks $\MMb(\Gamma)$}\label{app:squarefree}
Let $\Gamma \subset \Gamma_0(n)$ be a subgroup containing $\Gamma_1(n)$. We view $\Gamma/\Gamma_1(n)$ by the map sending a matrix to its upper-left entry as a subgroup $(\Z/n)^\times$ and denote by $\MMb(\Gamma)'$ the stack quotient of $\MMb_1(n)$ by this group. The map $\MMb_1(n) \to \MMb(\Gamma)$ that we obtain by the functoriality of normalization induces a map $c\colon \MMb(\Gamma)' \to \MMb(\Gamma)$.

This map $c$ is an equivalence if and only if $\MMb(\Gamma)' \to \MMb_{ell, \Z[\tfrac1n]}$ is representable. Indeed, it is then automatically finite (as $\MMb_1(n) \to \MMb_{ell, \Z[\tfrac1n]}$ is) and as both $\MMb(\Gamma)'$ and $\MMb(\Gamma)$ are normal (even smooth over $\Z[\tfrac1n]$) the result follows by \cite[Lemma 4.4]{H-M17}.

The map $\MMb_0(n) \to \MMb_{ell,\Z[\tfrac1n]}$ is representable if the level $n$ is squarefree \cite[Remark 4.1.5]{Con07}, but not if $n$ is not squarefree \cite{Ces17}. Indeed let $n = km$ with $k|m$ and $k\neq 1$ and let $K$ be a field of characteristic not dividing $n$ and containing an $n$-th root of unity $\zeta_n$. Consider the $K$-valued point of $\MMb_1(n)$ that corresponds to the point $P = ([1],\zeta_n)$ on a N\'eron $m$-gon. The automorphism $([i], x) \mapsto ([i], \zeta_k^ix)$ with $\zeta_k = \zeta_n^m$ preserves the subgroup generated by $P$ and contracts to the identity on the corresponding N\'eron $1$-gon. Thus, the image of $P$ in $\MMb_0(n)'$ has a non-trivial automorphism that becomes trivial on $\MMb_{ell,\Z[\tfrac1n]}$ and thus $\MMb_0(n)' \to \MMb_{ell,\Z[\tfrac1n]}$ cannot be representable. In summary, $c$ is always an equivalence if $n$ is squarefree (as $\MMb(\Gamma)' \to \MMb_0(n)$ is representable), but not in general if $n$ is not squarefree.

We want to investigate the cohomological behavior of $c$ if it is not an equivalence. For this, we will use coarse moduli spaces. Recall that by Keel--Mori \cite{Con05} the coarse moduli space of a Deligne--Mumford stack $\XX$ over a scheme $S$ exists if $\XX$ is locally of finite presentation over $S$ and $\XX$ has finite inertia stack. The latter condition is always fulfilled if $\XX$ is separated. The following lemma should be well-known.

\begin{lemma}\label{lem:coarse}
Let $X$ be the coarse moduli space of a Deligne--Mumford stack $\XX$ that is locally of finite type and separated over a noetherian scheme $S$. If $\XX$ is normal, $X$ is normal as well. Moreover, $\XX \to S$ is proper if and only if $X \to S$ is proper.
\end{lemma}
\begin{proof}
Assume that $\XX$ is normal. By (the proof of) Theorem 2.12 from \cite{Ols06}, we can choose for every geometric point $x$ of $X$ an \'etale neighborhood $W$ such that $\XX\times_X W \cong U/G$ for some affine scheme $U = \Spec R$ and some finite group $G$ acting on $U$. As the formation of coarse moduli is compatible with \'etale base change, $W \simeq \Spec R^G$; as $R$ is normal, $W$ is normal as well (see e.g.\ \cite{MONormal}). As normality descends via \'etale maps, $X$ is normal.

By \cite{Con05}, the map $\XX \to X$ is proper. Thus $\XX \to S$ is proper if $X\to S$ is proper. Now assume that $\XX \to S$ is proper. As $\XX \to X$ is surjective (even a homeomorphism), $X \to S$ is universally closed. Moreover, $X \to S$ is separated and locally of finite type by \cite{Con05}. As $\XX
\to S$ is quasi-compact and $\XX \to X$ is a homeomorphism, we see that $X\to S$ is quasi-compact as well and thus actually of finite type.
\end{proof}

The following lemma is a generalization of Proposition 2.6 of \cite{MeiDecMod}.

\begin{lemma}\label{lem:compactificationsGamma}
Let $\Gamma_1(n)\subset \Gamma\subset \Gamma_0(n)$ and $c\colon \MMb(\Gamma)' \to \MMb(\Gamma)$ be as above. For every quasi-coherent sheaf $\FF$ on $\MMb(\Gamma)$, the canonical map $\FF \to c_*c^*\FF$ is an isomorphism.
  Moreover, for every quasi-coherent sheaf $\GG$ on $\MMb(\Gamma)'$, the canonical map
  $$ H^i(\MMb(\Gamma); c_*\GG) \to H^i(\MMb(\Gamma)'; \GG)$$
  is an isomorphism for all $i\geq 0$.
\end{lemma}
\begin{proof}
We claim first that $c$ induces an isomorphism on coarse moduli spaces. By \cite[Th\'eor\`eme IV.3.4]{D-R73}, $\MMb(\Gamma)$ is smooth and proper over $\Spec \Z[\tfrac1n]$ and thus its coarse moduli space $X(\Gamma)$ is normal and proper over $\Spec \Z[\tfrac1n]$ by the preceding lemma. This implies that the induced map $X(\Gamma) \to \mathbb{P}^1_{\Z[\tfrac1n]}$ to the coarse moduli space of $\MMb_{ell,\Z[\tfrac1n]}$ is proper as well. As it is quasi-finite (as $\MMb(\Gamma)\to \MMb_{ell,\Z[\tfrac1n]}$ is finite), it is finite. A similar argument shows that the coarse moduli space $X(\Gamma)'$ of $\MMb(\Gamma)'$ is normal and finite over $\mathbb{P}^1$ as well (using the known properties of $\MMb_1(n)$). As $\MMb(\Gamma)' \to \MMb_{ell,\Z[\tfrac1n]}$ and $\MMb(\Gamma) \to \MMb_{ell,\Z[\tfrac1n]}$ are faithfully flat (and hence open), $\MM(\Gamma)$ lies open and dense in $\MMb(\Gamma)'$ and $\MMb(\Gamma)$ and thus also its coarse moduli space in $X(\Gamma)'$ and $X(\Gamma)$. Thus, $c$ induces an isomorphism $X(\Gamma)' \to X(\Gamma)$ (see e.g.\ \cite[Lemma 4.4]{H-M17}).

The common nonvanishing locus $D$ of $j$ and $j-1728$ on $\MMb_{ell}$ is of the form $X/C_2$ with the $C_2$-action on $X = \Spec \Z[j, (j(j-1728))^{-1}]$ trivial \cite[Lemma 3.2]{ShinBrauer}. The corresponding automorphism of every point in $D$ is given by the $[-1]$-automorphism of the corresponding elliptic curve. We denote by
\begin{align}\label{eq:mmbsquare}
\xymatrix{\MMb(\Gamma)_X' \ar[r]^{c_X}\ar[d] & \MMb(\Gamma)_X \ar[d]\\
\MMb(\Gamma)_D' \ar[r]^{c_D} & \MMb(\Gamma)_D
}
\end{align}
the base changes of $c$ along the open inclusion $D \to \MMb_{ell}$ and the map $X \to \MMb_{ell}$.

As $c$ induces an isomorphism on coarse moduli, the same is true for $c_D$. If $-1 \notin \Gamma$, all points in $\MMb(\Gamma)_D$ have trivial automorphism group and hence $\MMb(\Gamma)_D$ is an algebraic space by \cite[Theorem 2.2.5]{Con07} and the map $c_D$ identifies $\MMb(\Gamma)_D$ with a coarse moduli space for $\MMb(\Gamma)_D'$. In general, only $\MMb(\Gamma)_X$ is an algebraic space (and even a scheme) because $\MMb(\Gamma) \to \MMb_{ell}$ is representable. If $-1 \in \Gamma$, then the $C_2$-actions on $\MMb(\Gamma)_X'$ and $\MMb(\Gamma)_X$ are isomorphic to the identity and the vertical maps in \eqref{eq:mmbsquare} induce isomorphisms on coarse moduli. Hence $c_X$ induces an isomorphism on coarse moduli in this case as well. (Note that this is also true if $-1 \notin\Gamma$ as the formation of coarse moduli commutes with \'etale base change.)

From here, the proof is analogous to \cite[Proposition 2.6]{MeiDecMod}.
\end{proof}

These two lemmas allow us to give an alternative proof of the following proposition of Hill and Lawson.

\begin{prop}
For $\Gamma_1(n) \subset \Gamma \subset \Gamma_0(n)$, the map $\Tmf(\Gamma) \to \Tmf_1(n)$ induces an equivalence $\Tmf(\Gamma) \to \Tmf_1(n)^{h\Gamma/\Gamma_1(n)}$. 
\end{prop}
\begin{proof}
The $\Gamma/\Gamma_1(n)$-action on $\MMb_1(n)$ equips the stack quotient $\MMb(\Gamma)'$ with a sheaf of $E_\infty$-rings $\OO^{top}_{\MMb(\Gamma)}$ with global sections $\Tmf_1(n)^{h\Gamma/\Gamma_1(n)}$, analogously to \cite[Proposition 2.15]{MM15}. The corresponding map of descent spectral sequences on the $E_2$-pages
$$H^q(\MMb(\Gamma);g^*\omega^{\tensor p}) \to H^q(\MMb(\Gamma)'; c^*g^*\omega^{\tensor p})$$
is an isomorphism by Lemma \ref{lem:compactificationsGamma} and thus an isomorphism on $E_\infty$. By the eventual horizontal vanishing line from Proposition \ref{prop:0affine} the result follows.
\end{proof}


\begin{remark}
The preceding proposition might let it appear that the difference between $\MMb(\Gamma)$ and $\MMb(\Gamma)'$ is a non-issue in topology. But the theory of \cite{MM15} shows that $$\Tmf_0(n) \to \Tmf_1(n)$$ is only a faithful $(\Z/n)^\times$-Galois extension in the sense of Rognes \cite{Rog08} if $\MMb_1(n) \to \MMb_0(n)$ is a $(\Z/n)^\times$-Galois cover, i.e.\ if and only if $n$ is squarefree.
\end{remark}

\section{Derived algebraic geometry}\label{sec:DAG}
Recall our definition of a derived stack. 
\begin{defi}
	Let $\XX_0$ be a Deligne--Mumford stack and let $\XX_0^{\acute{e}t, \mathrm{aff}}$ be the site of affine schemes with an \'etale map to $\XX_0$. Let further $\OO_{\XX}$ be a sheaf of $E_{\infty}$-ring spectra on $\XX_0^{\acute{e}t, \mathrm{aff}}$ together with an isomorphism $\pi_0\OO_{\XX} \cong \OO_{\XX_0}$. Then we call $(\XX_0, \OO_{\XX})$ a \emph{derived stack} if 
	the assignment $U\mapsto \pi_n(\OO_{\XX}(U))$ defines for every $n\in\Z$ a quasi-coherent sheaf on $\XX_0^{\acute{e}t, \mathrm{aff}}$.
\end{defi}

To make the connection to the extensive work on Lurie on derived algebraic geometry, we need to explain that our definition of a derived stack is essentially equivalent to that of a non-connective spectral Deligne--Mumford stack in the sense of \cite[Section 1.4.4]{SAG} whose underlying $\infty$-topos is $1$-localic. To be more precise, let us denote for a site $\CC$ by $\Shv(\CC)$ the $\infty$-category of sheaves that are valued in the $\infty$-category $\mathcal{S}$ of spaces. An $\infty$-topos is $1$-localic if and only if it is equivalent to $\Shv(\CC)$ for a site $\CC$. 
\begin{lemma}\label{lem:derivedstacksequivalence}
	Let $\XX_0$ be a Deligne--Mumford stack and $\OO_{\XX}$ a sheaf of $E_{\infty}$-rings on $\XX_0^{\acute{e}t, \mathrm{aff}}$ with an isomorphism $\pi_0\OO_{\XX} \cong \OO_{\XX_0}$. Then $(\XX_0, \OO_{\XX})$ is a derived stack if and only if $(\Shv(\XX_0^{\acute{e}t, \mathrm{aff}}), \OO_{\XX})$ is a non-connective spectral Deligne--Mumford stack. Moreover, every non-connecive spectral Deligne--Mumford stack whose underlying $\infty$-topos is $1$-localic arises up to equivalence in this way. 
\end{lemma}
\begin{proof}
	We will use the criterion \cite[Theorem 1.4.8.1]{SAG}. As $\Shv(\CC)^{\heartsuit}$ for a site $\CC$ is equivalent to the category of sheaves of sets on $\CC$, items (1) and (2) of this criterion are automatically fulfilled. Our condition for a derived stack can be separated into two conditions. Firstly that the sheafified homotopy groups $\pi_n\OO_{\XX}$ are a quasi-coherent $\pi_0\OO_{\XX}$-module (which is exactly item (3)). Secondly that $\pi_n(\OO_{\XX}(U)) \to (\pi_n\OO_{\XX})(U)$ is an isomorphism for $U$ affine, which we call condition (4'). We claim that, given the other conditions, (4') is equivalent to item (4), i.e.\ to the hypercompleteness of $\OO_{\XX}$. Assume (4) and let $q\colon\Spec A_0 \to \X_0$ be an \'etale morphism. The pullback $q^*\OO_{\XX}$ is still hypercomplete and thus $(\Spet A_0, q^*\OO_{\XX})$ defines an affine non-connective spectral Deligne--Mumford stack, which is automatically of the form $\Spet A$ for some $E_{\infty}$-ring by \cite[Proposition 1.4.7.2]{SAG}. \'Etale extensions of $A_0$ lift uniquely to \'etale extension of $A$ and thus $\pi_n(q^*\OO_{\XX}(B_0)) \cong \pi_nA \tensor_{A_0} B_0$ for any \'etale extension $A_0 \to B_0$. This already defines a sheaf and thus 
    \[(\pi_n\OO_{\XX})(A_0)\cong (\pi_nq^*\OO_{\XX})(A_0) \cong \pi_n(\OO_{\XX}(A_0)),\]
    which shows (4'). 
    
     For the opposite direction, assume that (4') holds. We claim that 
	\[\OO_{\XX}(U) \to \lim_n (\tau_{\leq n}\OO_{\XX})(U)\]
	is an equivalence for every affine scheme $U$. But because of (4'), we have an equivalence $(\tau_{\leq n}\OO_{\XX})(U) \simeq \tau_{\leq n}(\OO_{\XX}(U))$ and thus our claim reduces to the known convergence of the Postnikov tower for spectra. Thus, we can write $\OO_{\XX}$ as limit of truncated sheaves. Truncated sheaves are automatically hypercomplete \cite[Lemma 6.5.2.9]{HTT} and hypercomplete sheaves are closed under limits (since they are defined as local objects for $\infty$-connected morphisms \cite[Definition 1.3.3.4]{SAG}). 
	
	It remains to show that every non-connective spectral Deligne--Mumford stack whose underlying $\infty$-topos is $1$-localic arises in this way. Let $\XX = \Shv(\CC)$ be such a $1$-localic $\infty$-topos for $\CC$ a site and assume that $(\XX, \OO_{\XX})$ is a non-connective spectral Deligne--Mumford stack. By \cite[Theorem 1.4.8.1]{SAG} again, $(\CC, \pi_0\OO_{\XX})$ is a Deligne--Mumford stack and thus $(\XX, \OO_{\XX})$ has the required form. 
\end{proof}

Let $(\XX, \OO_{\XX})$ be a spectrally ringed $\infty$-topos. Assume that $\OO_{\XX}$ is connective. Then $\XX$ is $1$-localic iff the mapping space from $\Spec R$ (for a discrete commutative ring $R$) to $(\XX, \OO_{\XX})$ is $1$-truncated \cite[Proposition 1.6.8.5]{SAG}. This can be used to show that the $\infty$-category of derived stacks has pullbacks, at least under certain conditions. For this we recall that a sheaf $\OO$ of $E_{\infty}$-ring spectra on some $\infty$-topos is \emph{even-periodic} if $\pi_2\OO$ is an invertible $\pi_0\OO$-module, $\pi_1\OO=0$ and the canonical morphism $\pi_2\OO\tensor_{\pi_0\OO}\pi_n\OO\to \pi_{n+2}\OO$ is an isomorphism for all $n$. 

\begin{lemma}\label{lem:derivedpullback}
	Let 
	\[
	\xymatrix{
	 & (\YY, \OO_{\YY}) \ar[d] \\
		(\ZZ, \OO_{\ZZ}) \ar[r] & (\WW, \OO_{\WW})	
	}\]
	be a diagram of derived stacks such that either all structure sheaves are connective or all structure sheaves are even periodic. Then the pullback in non-connective spectral Deligne--Mumford stacks -- which exists by \cite[Proposition 1.4.11.1]{SAG} -- is a derived stack again. In the connective case or in the even-periodic case if one of the maps is flat, the underlying Deligne--Mumford stack of the pullback is the pullback of the underlying Deligne--Mumford stacks. 
\end{lemma}
\begin{proof}
	Assume first that all structure sheaves are connective. Then our claim follows directly from \cite[Proposition 1.6.8.5 and Remark 1.6.8.4]{SAG}.
	
Assume now that we are in the even-periodic case. Consider the pullback square of the connective covers
\[\xymatrix{
	(\PP, \OO_{\PP}) \ar[r] \ar[d] & (\YY, \tau_{\geq 0}\OO_{\YY}) \ar[d] \\
	(\ZZ, \tau_{\geq 0}\OO_{\ZZ}) \ar[r] & (\WW, \tau_{\geq 0}\OO_{\WW})	
}\] 
As noticed above, $(\PP, \OO_{\PP})$ is a derived stack, i.e.\ $\PP$ is $1$-localic.

We observe that for an arbitrary even-periodic non-connective Deligne--Mumford stack $(\XX, \OO_{\XX})$ and an arbitary non-connective Deligne--Mumford stack $(\XX', \OO_{\XX'})$, the map 
\[\Map((\XX', \OO_{\XX'}), (\XX, \OO_{\XX})) \to \Map((\XX', \OO_{\XX'}), (\XX, \tau_{\geq 0}\OO_{\XX}))\]
is an inclusion of those path-components that satisfy the following condition: The image of every local section of $\pi_2\OO_{\XX} = \pi_2 \tau_{\geq 0}\OO_{\XX}$ that is invertible in $\pi_*\OO_{\XX}$ is invertible in $\pi_*\OO_{\XX'}$. 


We define a new structure sheaf on $\OO_{\PP}'$ on $\PP$: It suffices to define $\OO_{\PP}'$ on affine $\Spec A$ mapping into $(\PP, \pi_0\OO_{\PP})$ such that the pullback of $\pi_2\OO_{\WW}$ to $\Spec A$ has an invertible global section $\beta$. In this case we define $\OO_{\PP}'(\Spec A) = A[\beta^{-1}]$. 

By the observation above we obtain a diagram 
\[\xymatrix{
	(\PP, \OO_{\PP}') \ar[r] \ar[d] & (\YY, \OO_{\YY}) \ar[d] \\
	(\ZZ, \OO_{\ZZ}) \ar[r] & (\WW, \OO_{\WW}),	
}\] 
which is again by the observation above cartesian. We see in particular that the underlying $\infty$-topos of the pullback is $1$-localic. 

Now assume that the morphism $(\ZZ, \OO_{\ZZ}) \to (\WW, \OO_{\WW})$ is flat. Then the same is true for the connective versions and thus also $(\PP, \OO_{\PP}) \to (\YY, \tau_{\geq 0}\OO_{\YY})$ is flat \cite[Remark 2.8.2.6]{SAG}. We obtain that $\pi_i\OO_{\PP}$ is for $i\geq 0$ the pullback of $\pi_i\OO_{\YY}$ along $(\PP^{\heartsuit}, \pi_0\OO_{\PP}) \to (\YY^{\heartsuit}, \pi_0\OO_{\YY})$. In particular, $\pi_2\OO_{\PP}$ is a line bundle and $\pi_2\OO_{\PP}^{\tensor n} \cong \pi_{2n}\OO_{\PP}$. It follows that $\pi_0\OO_{\PP} \cong \pi_0\OO_{\PP}'$ and thus that $(\PP^{\heartsuit}, \pi_0\OO_{\PP}')$ is the pullback of the underlying Deligne--Mumford stacks. 
\end{proof}

\bibliographystyle{alpha}
\bibliography{Chromaticb}
\end{document}